\documentclass[draft]{article}

\usepackage{amsmath,amsfonts,amsthm,amssymb,amscd,cancel,color,mathabx}
\usepackage{enumitem}
\usepackage{verbatim}
\usepackage[dvipdfmx]{graphicx}
\usepackage{ulem,color}

\usepackage{cleveref} 
\renewcommand{\Re}{\mathop{\rm Re}\nolimits}
\renewcommand{\Im}{\mathop{\rm Im}\nolimits}
\def\S{\mathhexbox278}

\theoremstyle{plain}
\newtheorem{theorem}{Theorem}[section]
\newtheorem{lemma}[theorem]{Lemma}
\newtheorem{proposition}[theorem]{Proposition}

\theoremstyle{definition}

\theoremstyle{remark}
\newtheorem{remark}[theorem]{Remark}
\newtheorem{notation}[theorem]{Notation}

\newcommand{\R}{{\mathbb R}}

\newcommand{\Q}{{\mathbb Q}}

\newcommand{\N}{{\mathbb N}}

\def\im{{\rm i}}

\newcommand{\C}{\mathbb{C}}

\def\({\left(}
\def\){\right)}
\def\<{\left\langle}
\def\>{\right\rangle}
\newcommand{\sech}{{\mathrm{sech}}}

\newcommand{\supp}{{\mathrm{supp}\ }}
    \newcommand{\diag}{{\mathrm{diag}}}
\newcommand{\Span}{{\mathrm{Span}}}
\newcommand{\sign}{{\mathrm{sign}}}

\makeatletter

\numberwithin{equation}{section}

\begin{document}

\title{ 
The asymptotic stability of the stable solitons of the generalized good Boussinesq equation
}

\author{Junsik Bae, Scipio Cuccagna, Daniele Del Santo  and Masaya Maeda  }
\maketitle

\begin{abstract}   This paper continues the series where  we apply a synthesis of   the virial inequalities  in the tradition of Merle, Martel  and their collaborators 
and of Kato smoothing,   to close the virial inequalities. We prove absence of nonzero imaginary eigenvalues for the linearization in unweighted space and asymptotic stability for every velocity $\frac{\sqrt{p-1}}{2}<c<1$ and every $1<p<5$.
\end{abstract}


\section{Introduction}

Let $p>1$. We consider real-valued solutions on $\R$ of the
generalized good Boussinesq equation (gB)
\begin{align}\label{eq:gB1}&
   \partial _t ^2   u    -\partial _x ^2   u  +  \partial_x  ^2\(  \partial_x ^2  u +  f(u) \) =0
\end{align}
where $f(u)=|u|^{p-1}u$   or $f(u)=|u|^{p } $. We will write it as a system  in the form
\begin{align} \label{eq:gB2}&
   \left\{
     \begin{array}{ll}
       \partial _t u_1= \partial _x u_2 , &   \\
        \partial _t u_2= \partial _x     \(  - \partial_x ^2  u _1+u_1 - f(u_1)  \)  . &
     \end{array}
   \right.
\end{align}
Specifically, in integral form we can rewrite the equation  for $\mathbf{u}=(u_1, u_2) ^\intercal $  as
\begin{align}
  \label{eq:gB3}&  \mathbf{u}(t) = U(t)\mathbf{u}_0 -\int _0 ^t  U(t-t')  \partial _x  \(   0 ,  f(u_1 (t') )   \) ^\intercal dt ' \text{ where } \mathbf{u}_0 \in H^1(\R ) \times L^2 (\R ) \text{ and }\\& U(t): = \cos \( \im  t\partial _x  \sqrt{1- \partial _x^2} \)I_2   - \im \left(
                                                                \begin{array}{cc}
                                                                  0 & (1- \partial _x^2) ^{-\frac{1}{2}} \\
                                                                  (1- \partial _x^2) ^{ \frac{1}{2}} & 0 \\
                                                                \end{array}
                                                              \right)  \sin \( \im  t\partial _x  \sqrt{1- \partial _x^2} \) .  \label{eq:group}
\end{align}
We consider $\mathcal{H}^1= H^1(\R ) \times   L^2(\R )$  with norm
\begin{equation}
  \|   (u_1, u_2) ^\intercal \| _{\mathcal{H}^1 } ^2 =    \|    u_1  \| _{ H^1(\R )} ^2  + \|     u_2   \| _{ L^2(\R )} ^2. \label{eq:ennorm}
\end{equation}
For local well-posedness in the energy space, see Linares
\cite[Theorem~4.2]{Lineares93JDE}; earlier results for more regular
data were obtained by Bona and Sachs \cite{MR954673}.
For completeness, we give in Appendix~\ref{app:them:locwp}
a self-contained proof of the following formulation, valid for
both nonlinearities in \eqref{eq:gB1} and every $p>1$.
\begin{theorem}\label{them:locwp} \begin{description}
                                    \item[1] For any initial datum   $\mathbf{u} _0  \in \mathcal{H}^1$  there exists  $T=T\( \|  \mathbf{u} _0 \| _{\mathcal{H}^1} \) >0$ and a solution $\mathbf{u}\in C ^{0}([0, T], \mathcal{H}^1 )$  of  \eqref{eq:gB3}.
                                    \item[2] There is a unique maximal solution for any initial value problem.
                                    \item[3] Given a solution   $\mathbf{u}\in C ^{0}([0, T], \mathcal{H}^1 )$
 for any $T'\in (0,T)$ there exists a neighborhood $\mathcal{V}$   of $\mathbf{u}_0$ in $\mathcal{H}^1$ such that the map 
\begin{align}\label{eq:contflow}
  \mathcal{V} \ni  \mathbf{v}_0 \rightarrow   \mathbf{v}\in  C ^{0}([0, T'], \mathcal{H}^1 )
\end{align}
which associates to any initial datum its solution, is well defined and continuous.
 \item[4] For    $\mathbf{u}=(u_1, u_2) ^\intercal $  the following functionals are invariants of motion.
\begin{align}
   \label{eq:energy}E(\mathbf{u}) &=  2^{-1} \|   \mathbf{u}  \| _{\mathcal{H}^1 } ^2 -    \int _{\R }  F ( u_1 ) dx  \, , \text{ for $F'=f$ and $F(0)=0$} \\ \label{eq:mass}M(\mathbf{u}) &=  2^{-1} \<\sigma _1   \mathbf{u}   , \mathbf{u} \>    \text{ where }   \sigma_1=\begin{pmatrix} 0 &
   1 \\
   1 & 0
    \end{pmatrix} .
\end{align}

                                  \end{description}

\end{theorem}

Using Notation \ref{not:notation}  we can rewrite system  \eqref{eq:gB2} as
\begin{align} \label{eq:gB4syst}&
       \partial _t \mathbf{u} =   \partial _x  \( \sigma _1 \diag \( -\partial _x^2 +1, 1     \)  \mathbf{u}   - (0, f  ( u_1 ) )^\intercal  \)  .
\end{align}
\begin{remark}
  \label{rem:poitsyst} Since  in $H ^{-1}(\R ) \times H ^{-2}(\R )$   for the domain we have $D\( \partial _x  \( \sigma _1 \diag \( -\partial _x^2 +1, 1     \) \) \)= \mathcal{H}^{1} $ and since  by $\mathbf{u}\in C ^{0}([0, T], \mathcal{H}^1 )$ we have $ \partial _x  (0, f  ( u_1 ) )^\intercal  \in C ^{0}([0, T],\mathcal{H}^{1}  )$,  then   we can apply Proposition 4.1.6 in Cazenave and Haraux \cite{CazHarbook} and conclude that
  system \eqref{eq:gB2}  or the equivalent system  \eqref{eq:gB4syst} holds in a pointwise way. 
\end{remark}

System  \eqref{eq:gB4syst} has solitary waves
\begin{align}\label{eq:QQ}
  \mathbf{Q}_c(x-ct-x_0) =( Q_c, -cQ_c)^\intercal (x-ct-x_0) \text{ for } x_0\in \R \text{ and  }  c \in (-1,1)
\end{align}
where
\begin{align}\label{eq:QQ1}
 Q_c( x) = \gamma ^{\frac{2}{p-1}} Q ( \gamma x)  \text{ where  }     \gamma = \sqrt{1-c^2} \text{ and  }     Q (   x) =  \(\frac {p+1}2 \)^{\frac 1{p-1}}
\sech   ^{\frac 2{p-1}}\(\frac{p-1}2 x\)  ,
\end{align}
as can be shown using equations
\begin{equation}\label{eq:soliton_eq}
	Q'' -Q +f(Q)=0, \quad \hbox{and} \quad Q_c''-(1-c^2)Q_c+f(Q_c)=0.
	\end{equation}
 From Bona and Sachs \cite{MR954673} and Maul\'{e}n and Mu\~{n}oz \cite[Theorem 2.4]{maulen2025asymptoticstabilitystablegood} 
 we have the following.

 \begin{theorem}[Orbital Stability]\label{thm:orbstab} Let $1<p< 5$. Then for any   $ \frac{p-1}{4}<c^2 <1$ and any $x_0\in \R$ there exist     $\delta_0(c),C_0(c)>0$ such that for all $0<\delta<\delta_0(c)$ for any  $\mathbf{u} _{0}\in D _{\mathcal{H}^1} \left ( \mathbf{Q}_c (\cdot -x_0), \delta    \right )    $   we have for the corresponding solution $\mathbf{u}\in C ^0 ([0, +\infty ),  \mathcal{H}^1 )$  and there is a function $x(t)$ such that
\begin{equation}\label{OS}
\sup_{t\geq 0} \| \mathbf{u} - \mathbf{Q}_c (\cdot -x(t)) \|_{\mathcal{H}^1} \leq C_0(c)\delta.
\end{equation} 
\end{theorem}

\begin{remark}
For $f(s)=|s|^{p-1}s$, the stability range in
Theorem \ref{thm:orbstab} is sharp.
 Indeed, for
$1<p<5$, the solitary wave $\mathbf Q_c$ is orbitally unstable in
$\mathcal H^1$ when 
    $c^2\leq \frac{p-1}{4}$:
  case
$c^2< (p-1)/4$ was proved by Liu \cite{Liu1993};  the degenerate
critical case
$c^2=(p-1)/4$  by Li et al.
\cite{LiOhtaWuXue2020}.
Consequently, for $1<p<5$, the solitary wave $\mathbf Q_c$ is
orbitally stable precisely when
    $\frac{p-1}{4}<c^2<1$.
    See also, \cite{SS12Nonlinearity}.
\end{remark} 

 Our aim in this paper is to prove the following result.
 
 \begin{theorem}
    \label{thm:main} For any $1<p<5$, any
$c_0\in ((p-1)^{1/2}/2,1)$ and any $a>0$ and $\epsilon>0$,
      there exists a $\delta>0$
      such that for any     $\mathbf{u}_0 \in \bigcup _{x_0\in \R }D _{\mathcal{H}^1} \left ( \mathbf{Q} _{c_0} (\cdot -x_0), \delta    \right ) $
   then  for the
  solution $\mathbf{u}\in C^0([0,+\infty),\mathcal H^1 )$ of  \eqref{eq:gB3}  
  there exist  functions
  $ c \in C^1 ([0,+\infty) , (c_0 - \epsilon ,c_0+  \epsilon) ),$ $D \in C^1 ([0,+\infty), \R ) $  
    and $c_+\in (0,1)$  such that
  \begin{align} \label{eq:asstab1}
    &\mathbf{ u}(t)=    \mathbf{Q} _{c (t)} (\cdot -D(t) ) + \mathbf{v}(t, \cdot -D(t) ) \text{   with}
 \\&  \label{eq:asstab2}   \int _0^{+\infty} \|  e^{- a|x|}   \mathbf{v} (t ) \| _{\mathcal{H}^1 }^2 dt <  \epsilon ,  \\&  \label{eq:asstab20}
      \lim _{t\to +\infty  }  \|  e^{- a\< x\>}   v _1 (t ) \| _{L^2(\R )}     =0\text{ and} \\&
  \lim _{t\to +\infty}c (t)= c _+ . \label{eq:asstab3}
\end{align}

\end{theorem}

\begin{remark} \label{rem:udermain}    Our proof is basically the same for all $1<p<5$. 
     However we will prove Theorem \ref{thm:main} assuming $p\in (1,2]$: for $p> 2$  the proof is  slightly simpler.
\end{remark}

\begin{remark} \label{rem:meaningC}
The crucial property used in the resolvent estimates away from zero is the nonvanishing of the outgoing determinant $\widehat W$ on $\im\R\backslash\{0\}$, see below Lemma \ref{lem:nozerodet}. This follows from a conserved quantity and the absence of nonzero imaginary eigenvalues of $\mathcal{L}_{[c]}$ in $L^2(\R,\C^2)$, established in Proposition \ref{prop:noeigh}.
Together with the low and high energy estimates, this yields the resolvent estimates for every
$c\in ((p-1)^{1/2}/2,1)$.
\end{remark}

 Here we focus on $c>0$ but the same  conclusions hold for $c<0$.
  Maul\'en and Mu\~noz \cite{maulen2025asymptoticstabilitystablegood}  prove an analogous   result  for $2\le p<5$  for velocities
   $c_+(p) < |c| <1$ for a constant $c_+(p) \in [ (  p-1) ^{1/2}/2 , 1   ) $ for which     they find  computationally that  there exists a $p_0\le  p<5$ for a  $ p_0>3$ close to 3 such that $ c_+(p) =\frac{\sqrt{p-1}}{2}$ for $p\ge p_0$.   Theorem \ref{thm:main} applies to the whole interval
$((p-1)^{1/2}/2,1)$ for every $1<p<5$. Stable manifolds in the unstable regimes are discussed in Maul\'en \cite{maulen2023}.
See also \cite{MPP18CMP}.

The basic idea of our proof, which we tested  in \cite{BaCuMa, CM24D1, CM243, CM24D4, CM25D1}, is to split the virial inequality method in Martel and Merle \cite{MaMeARMA2001, MaMeNonlinearity2005} and  Kowalczyk et al. \cite{KMM2017,KMM2022,KMMV2021,KM22}, Martel \cite{Martelcubquint1,Martelcubquint2} and    Rialland   \cite{Riallandcuqui2,Riallandcuqui1}, by performing only a high energy virial inequality, which here   is easy.   We recall that  the virial inequality method  meets a serious obstruction with a localized  quadratic term in $\mathbf{v}$, in fact only in the first coordinate $v_1$ here,  with a bad sign. Instead of bounding this with the first or with successive virial estimates, we   observe that this localized  quadratic term looks exactly of the type treated  using the notion of smoothing in Kato  \cite{katoMA1966}. Since \eqref{eq:gB2} is a transport equation, and  $\mathbf{v}(t)$,    unlike for gKdV,   is  for $t>0$ not  more regular than 
$\mathbf{v}(0)$, when proving Kato smoothing we have a problem  with certain derivatives in the right hand side of \eqref{eq:prekato1}. Essentially, here as well as in \cite{BaCuMa}, the energy does not provide a bound for $v_2'$ and $v_1''$, which are in the   right hand side of \eqref{eq:prekato1}. Fortunately  in this paper, the fact that the  localized  quadratic term depends only on  $v_1$ and that we need a smoothing estimate  only for  $v_1$, provides the regularizing effect needed  to offset the extra derivative in  $v_2'$ and $v_1''$.  Ultimately our method allows a synthesis of the 
virial method with the dispersive estimates method in Soffer and Weinstein \cite{SW1,SW2} and Gustafson et al. \cite{GNT}, with closure of the estimates insured by the presence of a spatial cut--off which tames the obstructions, due to slow dispersion,  faced  in  \cite{SW1,SW2, GNT}.  The taming is analogous     to arguments in the second virial inequality and in fact it is here largely inspired by     Kowalczyk et al. \cite{KMM2022}.

Related approaches combining virial or monotonicity arguments with
localized linear decay estimates were developed by
Mizumachi \cite{MR2052617} and El Dika \cite{ElDIkaasstab}.
Further developments appear in
\cite{Mizumachi09CMP,MR3059165,MT14RIMS}.
These approaches use linear estimates in exponentially weighted
spaces, following Pego and Weinstein \cite{PegoWei2};
see also \cite{MizuKdV2001}.
For the present equation, a scalar one-sided exponential weight
cannot shift both branches of the far-field continuous spectrum
into the left half-plane when $|c|<1$.
Thus the spectral-shift argument  of  Pego and Weinstein, of Mizumachi and of  El Dika   does not apply directly here.

 This paper,   with \cite{BaCuMa, CM24D1, CM243, CM24D4, CM25D1}, contributes to a growing body of applications to rather diverse models that can nevertheless be treated by essentially the same strategy. Our approach is indeed quite general. More specifically, the arguments developed from \S \ref{sec:SLM} through \S \ref{sec:virial} closely parallel computations  carried out in our previous works, and in particular those in \cite{CM25D1}. The smoothing estimates developed here starting from \S \ref{sec:disp} rely on features specific to the good Boussinesq equation. At a structural level, however, also these proofs  follow essentially the same scheme as in \cite{CM25D1}.

The present work should therefore be viewed as part of a broader program aimed at developing the virial--Kato method into a robust and general framework. The underlying idea is simple: Kato smoothing can be used to close virial inequalities, thereby making the same basic mechanism applicable to a variety of models.

Our approach provides an alternative to the methods developed by Kowalczyk et al. \cite{KMM2017,KMM2022,KMMV2021,KM22}, Martel \cite{Martelcubquint1,Martelcubquint2}, Rialland \cite{Riallandcuqui1,Riallandcuqui2}, Maul\'{e}n \cite{maulen2023}, and Maul\'{e}n and Mu\~{n}oz \cite{maulen2025asymptoticstabilitystablegood}. In particular, while our method relies on the absence of nonzero eigenvalues embedded in the continuous spectrum of the linearized operator, it does not require the repulsivity condition used, for instance, in Kowalczyk et al. \cite{KMM2022,KM22}; see in particular \cite[Formula (1.6)]{KMMV2021}.

A difficulty in approaches based solely on virial inequalities is that the Ansatz \eqref{eq:asstab1} involves a \textit{modulation} which is largely adapted to the linearized operator, whereas the virial argument naturally produces a quadratic form associated with a different linear operator. Establishing directly the positivity of this quadratic form can be nontrivial. This issue already appears in the framework developed for the mass-critical NLS by Merle and Raphael in \cite{MerleRaphaelAnnals}, and with Fibich in \cite{FibichMerleRaphael}, as well as in the work of Martel and Merle on gKdV \cite{MaMeARMA2001}. Similarly, in the cubic--quintic NLS, Martel \cite{Martelcubquint2}, in order to exploit Simon's characterization of strictly positive Schr\"odinger operators \cite{Simon1976}, restricts the analysis to small-amplitude solitons of the cubic--quintic NLS.

Of course, proving the absence of nonzero embedded eigenvalues of the linearized operator may itself be difficult. In the present paper and in \cite{CM25D1}, however, this spectral property can be established rather directly. Moreover, embedded eigenvalues are typically unstable under perturbations. This suggests that a method whose main spectral requirement is the absence of embedded eigenvalues may have a   broad range of applicability.

\section{Notation }\label{sec:notation}

\begin{notation}\label{not:notation}
Throughout this paper, we will use the following list of   notations and definitions.

\begin{enumerate}

\item \label{itm:Cpm} We will write $\C _{\pm}=\{ z\in \C : \pm \Re z >0   \}$.

\item \label{itm:neigh1}We will denote by $\Upsilon $  appropriate   open neighborhoods of   $\im \R$ in $\overline{\C} _+$.

\item    \label{itm:proj}  For  $\mathbf{z}=(z_1, ...,  z_n) ^\intercal $ we will consider the  projections  $\pi _j\mathbf{z}:=z_j$.
In $\C^n$ we will consider the pairing $\< \mathbf{x}, \mathbf{y}\> _{\C ^n}=  \sum _{j=1}^{n}x_j {y}_j$. All vectors will be column vectors.

\item    \label{itm:spop}   Given two Banach spaces $X$ and $Y$ we denote by $\mathcal{L}(X,Y)$ the space of continuous linear operators from $X$ to $Y$. We write $\mathcal{L}(X ):=\mathcal{L}(X,X)$.

\item \label{itm:der} Occasionally we will use  $u'$ instead of $\partial _x u$ and  $\dot u $ instead of $\partial _t u$.

\item \label{itm:Comega} For $\Omega $ an open subset of $\R ^d$ we write $f\in C^\omega (\Omega) $ if $f$ is real analytic in $\Omega$.

\item \label{itm:japanese-bracket} We will consider the Japanese bracket $\<x \> = \sqrt{1+x^2}$.

\item \label{itm:chi} Following  the framework in Kowalczyk et al. \cite{KMM2022} we   fix an even function $\chi\in C_c^\infty(\R , [0,1])$ satisfying
\begin{align}  \label{eq:chi} \text{$1_{[-1,1]}\leq \chi \leq 1_{[-2,2]}$ and $x\chi'(x)\leq 0$}.
\end{align}

\item \label{itm:constants-relation} In analogy to     Kowalczyk et al.   \cite{KMM2022},      we   consider constants  $A, B,  A_1, \epsilon , \delta  >0$ satisfying
 \begin{align}\label{eq:relABg}
\log(\delta ^{-1})\gg\log(\epsilon ^{-1}) \gg     A    \gg   B^2\gg B    \gg 1  \text{ and $ A_1:= B ^{3/5}$.}
 \end{align}

\item \label{itm:small-o} The notation    $o_{y}(1)$  means a function in  $y$ such that
 $o_{y}(1) \xrightarrow {y  \to 0^+   }0.$

\item \label{itm:diag} We denote by   $\diag (a_1, \cdots , a_n)$   the diagonal matrix with   $a_1, \cdots , a_n$ on the diagonal.

\item \label{itm:ball-banach} Given   a   Banach space $X$, $u\in X$  and $r>0$  we set  $D _{X}(u,r):=\{v\in X \ |\ \|u-v\| _{X}<r \}$.

\item \label{itm:L2a} 
 We define    \begin{align}\label{eq:Lebwight}&
 L ^{p,s}=   L ^{p,s}(\R ) =\{ v: \ \< x \> ^{s}   v\in L^p(\R )   \}  \text{ with }\|  v  \| _{ L ^{p,s}}:=\| \< x \> ^{s} v  \| _{L^p(\R ) }.
\end{align}

\item \label{itm:vartheta-partition} For $ \vartheta _1=\chi $ and $ \vartheta _2=1-\chi $ we consider the partition of unity $1=\vartheta _{1A_1}+\vartheta _{2A_1}$ where $\vartheta _{iA_1}(x):=\vartheta _i\(\frac{x}{A_1}\)$.

\item \label{itm:zeta-varphi} Following  the framework in Kowalczyk et al. \cite{KMM2022}  and
using the function $\chi$  in \eqref{eq:chi}
we  consider for   $C>0$  the   function
\begin{align}\label{def:zetaC}
\zeta_C(x):=\exp\(-\frac{|x|}{C}(1-\chi(x))\),
\end{align}
and
\begin{align}\label{def:phiAi}   \varphi _{iAA_1}  (x):=\int_0^x \zeta_A^2(y) \vartheta _{iA_1} ^2 (y)\,dy.
\end{align}

\item \label{itm:norms} For   functions $  \mathbf{v}=(  {v}_1, v_2 ) ^\intercal $  and  $\sech (y) = \frac{1}{\cosh  (y)}   $ 
 we  will consider the seminorms
\begin{align}
&
\| \mathbf{v}  \|_{     { \Sigma}    _{jA A_1}} :=\left \|   \vartheta _{jA_1} \zeta_A  \mathbf{v}    \right \|_{\mathcal{H}^1}    \text{   for $j=1,2$ and} \label{eq:normA1}  
\\& \| f  \|_{       \widehat{\Sigma}     } :=\left \| \sech \(     \gamma \kappa  x\)   f \right \|_{L^2(\R )}  \text{ for a fixed sufficiently small  $\kappa  >0$. }  \label{eq:sigmatilde}
\end{align}
In all nonlinear estimates for the modulated solution with reference
velocity $c_0$, the parameter $\gamma$ in the weight
\eqref{eq:sigmatilde} is fixed at $\sqrt{1-c_0^2}$.

\item \label{itm:exterior-power-operator} Given  a linear operator $A \in \mathcal{L}(\C ^n)$ then considered for $1\le k \le n $   the space $  \wedge ^k (\C ^n)$ we denote by $A ^{(k)} \in \mathcal{L}(\wedge ^k (\C ^n))$,  see \cite[Proposition 3.2]{Bealsbook1988},
the operator
\begin{align*}
   A ^{(k)}( a_1 \wedge ...\wedge a_k) =A a_1 \wedge a_2 \wedge  ...\wedge a_k + a_1 \wedge A a_2 \wedge  ...\wedge a_k +...+a_1 \wedge   a_2 \wedge  ...\wedge Aa_k .
\end{align*}
 For an introduction to exterior algebras   see \cite[Definition 2.3.6]{AbrMarsdFound}.

\item \label{itm:hodge-star} The Hodge star operator $ *:\wedge ^k (\C ^n) \to   \wedge  ^{n-k} (\C ^n)$ is such that
$\<  \alpha , \beta \> _{\wedge ^k (\C ^n)}   =\det ( \alpha \wedge *\beta ) $ where $\alpha \wedge *\beta = \det ( \alpha \wedge *\beta ) e_1 \wedge ...\wedge e_n $, with  $( e_i ) _{i=1}^{n}$ the standard basis of $\C^n$  and where the inner product  $\<  \cdot  , \cdot  \> _{\wedge ^k (\C ^n)} $ is such that $\{  e_{i_1}\wedge ...\wedge e_{i_k} \} $ for all $i_1<...<i_k$ is an orthonormal basis of
$\wedge ^k (\C ^n)$. Notice that for $n$ even we have $ * ^2 =(-1)^k$.  Notice that for vectors in $\C ^4$
 \begin{equation*}
 \det (   X _1 \wedge X_2\wedge X _3\wedge X_4 )=\det (X_1, X_2, X _3, X_4)    ,
 \end{equation*}
with in   the right hand side   the   determinant of the matrix and     for all $ X_1, X_2, X _3, X_4 \in \C ^4$
 \begin{equation}\label{eq:hodgestar}
   *:\wedge  ^3(\C ^4) \to \C ^4 \text { we have  }  \det (X_1, X_2, X _3, X_4)=-\<  X_1, * (  X_2\wedge X _3\wedge X_4)\> _{\C^4}  .
 \end{equation}
\item \label{itm:real-imag-parts} For $z\in \C$ we will write $z=z_R +\im z_I$ with $z_R =\Re z $ and $z_I =\Im z $.

\item \label{itm:Lambda-p} By elementary computation,  we have
\begin{align}
      \partial _\gamma  \(  \gamma ^{\frac{2}{p-1}}    g (  \gamma  x)\)& = \gamma ^{\frac{2}{p-1}-1}   (\Lambda _pg)( \gamma  x )
        \text{ where }
         \Lambda _pg(x) =\frac{2}{p-1}  g(x) +  x g'(x). \label{eq:partcphic}
\end{align}

\item \label{itm:resolvent} Give an operator $T$ for us the resolvent is  $R_{T }(z)=(T -z)^{-1}$.

\item \label{itm:coninfry} When we write $\xi \to \infty$, we mean $|\xi | \to \infty$.

\item  \label{itm:chvar} We will later  make change of variables (the scaling is like in \cite{millerweinstein96})
\begin{align}\label{itm:chvar1} &
  \text{  $\lambda = \gamma ^3 \Lambda $,  $\xi = \gamma x $  and $f (x) = \widetilde{f} (  \xi)$,}  \\& \label{eq:relkern1}  F(x)  =(F_1(x), F_2(x))^\intercal =  (\widetilde{F}_1(\xi ), \widetilde{F}_2( \xi))^\intercal \text{ and}
   \\  & \widetilde{\mathbf{F}}(\xi )   = \( 0,0,  \gamma^{-3} (-\widetilde{F}_2(\xi )+c\widetilde{F}_1(\xi ))  ,     \gamma^{-1} \widetilde{F}_1(\xi )  \)  ^\intercal . \label{eq:relkern2}
\end{align}

\end{enumerate}

\end{notation}

\section{  Linearization and  modulation  }\label{sec:SLM}

As in \cite{maulen2025asymptoticstabilitystablegood}  we consider \begin{align} &
  L _{+\gamma}:=- \partial _x^2   +\gamma ^2-f' (Q_c )    \text{   ,  }   \label{eq:lin0}\mathbf{L} _{ c}:=  \begin{pmatrix} L _{+\gamma} +c^2 &
   c \\
   c & 1
    \end{pmatrix}    \text{   and }
   \\&
  \mathcal{L}_{[c]}
  =  \partial _x   \sigma _1   \mathbf{L} _{ c}.  \label{eq:Lin}
\end{align}
There is an immense literature on    $L_{+\gamma}$:  we refer to   the classical  \cite{weinMod1985} and to the more recent \cite{CGN2007}. 
Let  \begin{align*}
  q(c) :=  2^{-1}\< \mathbf{Q}  _c , \sigma _1 \mathbf{Q}  _c  \>  .
\end{align*}
Then by an elementary computation
\begin{align}\nonumber
  q'(c) &=  \< \mathbf{Q}  _c , \sigma _1\partial _c \mathbf{Q}  _c  \>   =-\< Q_c, \partial _c (cQ_c)   \> -\< c Q_c, \partial _c Q_c    \>  =-\| Q _c \| _{L^2(\R )}^2 -2 c\<  Q_c, \partial _c Q_c    \>
\\& \nonumber  = -  \gamma ^{\frac{4}{p-1}}\|  Q ( \gamma \cdot ) \| _{L^2(\R )}^2   +  \frac{c^2}{\gamma}    \partial _\gamma       \(  \gamma ^{\frac{4}{p-1}}\|  Q ( \gamma \cdot ) \| _{L^2(\R )}^2 \)  \\& =   \gamma ^{\frac{4}{p-1}-3}  \|  Q   \| _{L^2(\R )}^2   \(   c^2 \(  \frac{4}{p-1} - 1 \) -   \gamma ^2      \) . \label{eq:qprc}
\end{align}
  Notice that $ q'(c)=0$  for $c ^2 = \frac{p-1}{4} $ and 
\begin{equation}\label{eq:q_prime}
    q'(c)>0 \quad \text{for }1> c ^2 > \frac{p-1}{4}.
\end{equation}  
For $c\in ((p-1)^{1/2}/2,1)$, set
\begin{align*}
  d(c):=( d_1(c)   ,  d_2(c) ) ^\intercal
  :=\int_{\R}\partial_c\mathbf{Q}_c(y)\,dy,
  \qquad s_c:=d_2(c)-d_1(c).
\end{align*}
We set 
\begin{align}
& \xi_1[c](x):=\mathbf{Q}_c'(x),\quad \xi_2[c](x):=\partial_c\mathbf{Q}_c(x)
  \label{eq:ker1}\\
& \eta_1[c](x):=
  \theta_1\sigma_1\int_{-\infty}^{x}
  \partial_c\mathbf{Q}_c(y)\,dy
  +\frac{s_c}{2q'(c)}\binom{1}{-1}
  +\theta_2\sigma_1\mathbf{Q}_c(x),
  \label{eq:ker2}\\
& \eta_2[c](x):=\theta_3\sigma_1\mathbf{Q}_c(x),
  \nonumber
\end{align}
where
\begin{align}
  -\theta_1=\theta_3=\frac{1}{q'(c)} \text{ and } \theta_2=\frac{d_1(c)^2+d_2(c)^2}{2(q'(c))^2}. \label{eq:deftheta}
\end{align}
 The functions in \eqref{eq:ker1}--\eqref{eq:ker2} satisfy, like in \cite{PegoWei1}, the following.

 \begin{lemma}
  \label{lem:kernel} We have the following:
\begin{align}
& \mathcal{L}_{[c]}\xi_1[c]=0,\quad
  \mathcal{L}_{[c]}\xi_2[c]=-\xi_1[c],
  \label{eq:M2.3}\\
& \mathcal{L}_{[c]}^*\eta_2[c]=0,\quad
  \mathcal{L}_{[c]}^*\eta_1[c]=-\eta_2[c],
  \label{eq:M2.4}\\
& \<\xi_i[c],\eta_j[c]\>=\delta_{ij}.
  \label{eq:M2.5}
\end{align}

\end{lemma}
  \begin{proof}
  Applying $\partial _x $ to \eqref{eq:soliton_eq} we obtain $ L _{+\gamma}Q _c '=0 $ and
     \begin{align}\label{eq:LcQprc} &  \mathbf{L} _{c} \mathbf{Q}  _c '=  \begin{pmatrix} L _{+\gamma}Q _c '+c^2 Q_c' -c^2Q_c'
    \\
   c  Q_c' -c Q_c'
    \end{pmatrix}  =0.
         \end{align}
         Hence, $\mathcal{L}_{[c]}\xi_1[c]=0$. 
     Applying $\partial _c $ to \eqref{eq:soliton_eq} we obtain $ L _{+\gamma}\partial _c Q _c = 2c Q _c  $. Then \small
   \begin{align}\nonumber
     \mathcal{L}_{[c]}\xi _2 [c]&  =\partial _x  \sigma _1   \begin{pmatrix}
       L _{+\gamma} +c^2 &
   c\\  c & 1
    \end{pmatrix}   \begin{pmatrix}
  \partial _c Q_c \\ -c \partial _c Q_c -Q_c
    \end{pmatrix}   = \partial _x  \sigma _1  \begin{pmatrix}  L _{+\gamma}  \partial _c Q_c +c^2 \partial _c Q_c -  c^2 \partial _c Q_c  -cQ_c
  \\   - Q_c
    \end{pmatrix}  \\& = \partial _x  \sigma _1   \begin{pmatrix}
 cQ_c \\       - Q_c
    \end{pmatrix}  = -  \xi _1 [c].\label{eq:Lxi2-xi1}
   \end{align} \normalsize
It is immediate from \eqref{eq:LcQprc} that
\begin{align*}
    \mathcal{L}_{[c]}^*\eta _2 [c] =   -       \mathbf{L} _{ c}  \partial _x  \sigma _1^2\theta _3 \mathbf{Q}  _c  =   - \theta _3       \mathbf{L} _{ c} \partial _x \mathbf{Q}  _c     =0
\end{align*}
and from \eqref{eq:LcQprc}--\eqref{eq:Lxi2-xi1}
\begin{align*}
    \mathcal{L}_{[c]}^*\eta _1 [c] =  -   \theta _1     \mathbf{L} _{ c}  \partial _x   \sigma _1^2  \int _{-\infty} ^{x } \partial _c \mathbf{Q}  _c (y) dy= -   \theta _1     \mathbf{L} _{c} \partial _c \mathbf{Q}  _c   =  -   \theta _1   \begin{pmatrix}
 cQ_c \\       - Q_c
    \end{pmatrix} =     \theta _1 \sigma _1 \mathbf{Q}  _c=-\eta _2 [c].
\end{align*}
We have
\begin{align*}
  \<  \xi _1, \eta _2\> = \theta _3  \< \mathbf{Q}  _c', \sigma _1  \mathbf{Q}  _c\> =0,
\end{align*}
\begin{align*}
  \<  \xi _2, \eta _2\> = \theta _3  \< \partial _c\mathbf{Q}  _c , \sigma _1  \mathbf{Q}  _c\> = \theta _3 q'(c) =1,
\end{align*}
\begin{align*}
  \<  \xi _1, \eta _1\> &= \theta _1  \< \mathbf{Q}  _c' , \sigma _1 \int   _{-\infty} ^x \partial_c \mathbf{Q}  _c(y) dy\>  +\theta _2  \underbrace{\< \mathbf{Q}  _c' , \sigma _1\mathbf{Q}  _c\> }_{0} = -\theta _1  \< \mathbf{Q}  _c  , \sigma _1   \partial_c \mathbf{Q}  _c\>  =  -\theta _1  q'(c)=1.
\end{align*}
Finally, integration by parts gives
\begin{align*}
  \<\partial_c\mathbf{Q}_c,
    \sigma_1\int_{-\infty}^{x}
    \partial_c\mathbf{Q}_c(y)\,dy\>
  &=\frac12\<d(c),\sigma_1d(c)\>_{\R^2}
    =d_1(c)d_2(c),\\
  \<\xi_2,(  {1},{-1}) ^\intercal\>
  &=d_1(c)-d_2(c)=-s_c.
\end{align*}
Hence
\begin{align*}
  \<\xi_2,\eta_1\>
  &=\theta_1d_1(c)d_2(c)
    -\frac{s_c^2}{2q'(c)}+\theta_2q'(c)\\
  &=\frac{-2d_1(c)d_2(c)-(d_2(c)-d_1(c))^2
          +d_1(c)^2+d_2(c)^2}{2q'(c)}
    =0.
\end{align*}

\end{proof}

\begin{proposition}\label{prop:noeigh} Consider $\mathcal{L}_{[c]}$ with domain $ D(\mathcal{L}_{[c]})=H^3( \R )\times  H^1( \R )$ as an operator in  $L^2(\R, \C ^2)$
and for $c\in  ( (  p-1) ^{1/2}/2 , 1   )$.
The following hold.
 \begin{description}
  	  \item[i]  The   spectrum is    $\sigma  \left (    \mathcal{L}_{[c]}   \right ) =\im \R$.
     \item[ii]  0 is an eigenvalue with   $\ker   \mathcal{L}_{[c]} =\Span \{     \xi _1 [c]\}$.

  \item[iii]    There are no eigenvalues $\lambda \in \im \R \backslash \{ 0  \} $  in $L^2(\R, \C ^2)$.

   \end{description}

\end{proposition}

\begin{proof}
The eigenvalue problem for $\mathcal{L}_{[c]}$   is
   \begin{align}
  \label{eq:eilinvs} \lambda \mathbf{v} =\mathcal{L}_{[c]} \mathbf{v} = \mathcal{L}_{[c]}^{(0)} \mathbf{v}+ T \mathbf{v}
       \text{ with }\mathcal{L}_{[c]}^{(0)} :=\partial _x   \begin{pmatrix}
 c &   1 \\  -\partial ^2 _x +1 &  c
    \end{pmatrix} \text{ and }T :=-\partial _x  \circ \begin{pmatrix}
0 &   0 \\  f'( Q_c)  &  0
    \end{pmatrix}         .
\end{align}
The  operator $ \mathcal{L}_{[c]}$ 
is a relatively compact perturbation of $\mathcal{L}_{[c]}^{(0)}$, Kato \cite[p. 194] {MR203473}, and so by  \cite[p. 244] {MR203473} their  essential spectra coincide. A point $\lambda$ belongs to the spectrum $\sigma \left (  \mathcal{L}_{[c]}^{(0)}  \right )$ if and  only if
\begin{align*}
  0= \det \begin{pmatrix}
 \lambda -\im ck  &   -\im k  \\  - \im k ^3 -\im k &   \lambda - \im ck
    \end{pmatrix}  =(  \lambda - \im ck )^2 +k ^4 +k ^2
\end{align*}
that is  if and only if the following holds, 
\begin{align*}
  \lambda = \im  k (  c\pm \sqrt{1+k ^2} ) \text{  for }k \in \R  .
\end{align*}
By the Sturm-Liouville theory, the number of negative $L^2$ eigenvalues of $L _{+\gamma}$ is 1. Combining this and \eqref{eq:q_prime}, we can apply the theory developed in Lin and Zeng \cite{LZ22Memoiers} (or by Corollary 11.1 of \cite{LZ22Memoiers}) to conclude that there are no eigenvalues of $\mathcal{L}_{[c]}$ with  $\Re \lambda  >0$ for $  \frac{\sqrt{p-1}}{2}<c  <1 $  for  \eqref{eq:eilinvs}. 

Since for $\mathbf{v}(x):= \mathbf{u}(-x)$ we have
$ (\mathcal{L}_{[c]} \mathbf{v})(x)= - (\mathcal{L}_{[c]} \mathbf{u})(-x)$, there are no eigenvalues with $\Re \lambda  <0$ and  we have Claim $\mathbf{i}$.

We know already by  \eqref{eq:M2.3}    that   $\ker   \mathcal{L}_{[c]} \supseteq  \Span \{     \xi _1 [c]\}$. Given $ \mathbf{v}  \in \ker \mathcal{L}_{[c]}$,  by $\mathbf{v}\in D(\mathcal{L}_{[c]})$  we conclude $ \mathbf{L} _{c}\mathbf{v}\in L^2(\R, \C ^2)$.
On the other hand, $\mathcal{L}_{[c]}\mathbf{v}=0 $  implies that  $ \mathbf{L} _{c}\mathbf{v} $  is a constant function, hence $ \mathbf{L} _{c}\mathbf{v}=0 $. By  \cite[Lemma 2.1]{maulen2025asymptoticstabilitystablegood} we have $\mathbf{v}\in\Span \{     \xi _1 [c]\}$, completing the proof of Claim $\mathbf{ii}$.

We prove   Claim \textbf{iii} at the end of this section.
\end{proof}

For  $c  _0 \in ( (  p-1) ^{1/2}/2 , 1   ) $ and $r >0$ and for   $D _{\mathcal{H} ^1   }(\mathbf{u},r):=\{\mathbf v\in \mathcal{H} ^1  \ |\ \|\mathbf{u}-\mathbf{v}\| _{\mathcal{H} ^1   }<r \}$,  we set     $$\mathcal{U} (c _0 ,r  ) = \bigcup _{  x_0\in \R }  e^{-x_0 \partial _x}  D_{\mathcal{H} ^1  }(\mathbf{Q }_{c _0},r ). $$
The following can be proved with a standard application of the implicit function theorem.  

\begin{lemma}[Modulation]\label{lem:mod1}
There exist $\delta_0>0$ and $B_0>0$ with $\delta_0\ll B_0^{-3/2}$ such that for $\delta_0^{-1}\gg B\ge B_0$ and for the function $\zeta_B$ introduced in \eqref{def:zetaC} there are $(c,D)\in C^1\( \mathcal U(c_0,\delta_0), ((p-1)^{1/2}/2,1)\times\R \)$ such that for any $\mathbf u\in\mathcal U(c_0,\delta_0)$
\begin{align}\label{61}
& \mathbf v(\mathbf u):= e^{D(\mathbf u)\partial_x}\mathbf u-\mathbf Q_{c(\mathbf u)}
\text{ satisfies }\\
& \<\mathbf v(\mathbf u),\zeta_B\eta_1[c(\mathbf u)]\>
=\<\mathbf v(\mathbf u),\eta_2[c(\mathbf u)]\>=0.
\nonumber
\end{align}
For any $c_1\in((p-1)^{1/2}/2,1)$ satisfying $\|\mathbf Q_{c_1}-\mathbf Q_{c_0}\|_{\mathcal H^1}<\delta_0$, and any $D_0\in\R$, we have the identities
\begin{align}
 c(\mathbf Q_{c_1})=c_1,\quad D(\mathbf Q_{c_1})=0,\quad
 D(e^{-D_0\partial_x}\mathbf u)=D(\mathbf u)+D_0,\quad
c(e^{-D_0\partial_x}\mathbf u)=c(\mathbf u).
\label{eq:equivariance}
\end{align}
Furthermore, there exists a constant $K>0$, independent of $B$,
such that
\begin{align}\label{eq:mod2}
B^{-1/2}|D(\mathbf u)-D_0|+|c(\mathbf u)-c_0|
\le K\|\mathbf u-e^{-D_0\partial_x}\mathbf Q_{c_0}\|_{\mathcal H^1} \text{,   $\quad \forall  \,   \mathbf u\in
e^{-D_0\partial_x}D_{\mathcal H^1}(\mathbf Q_{c_0},\delta_0)$.}
\end{align} 
\end{lemma}  
\begin{proof}
    The proof is  an almost verbatim replica of the corresponding proof in     \cite{CM25D1}, so we skip it.
\end{proof}

Choose $\delta>0$ within the range of
Theorem \ref{thm:orbstab} and sufficiently small so that
$C_0(c_0)\delta<\delta_0$.
Let $\mathbf u\in C^0([0,+\infty),\mathcal H^1)$ be the
corresponding global solution.
By \eqref{eq:gB4syst},
$\mathbf u\in C^1([0,+\infty),H^{-3}(\R,\R^2))$.
The test functions $\zeta_B\eta_1[c]$ and $\eta_2[c]$,
together with all their derivatives in $x$ and $c$, belong to $H^m$ for every nonnegative integer $m$.
Thus,   the orthogonality conditions in \eqref{61}
extends  
$\mathbf u$ in $H^{-3}(\R,\R^2)$.
 The implicit function theorem   gives
\begin{align*}
(c(t),D(t)):=(c(\mathbf u(t)),D(\mathbf u(t)))
\in C^1\(
[0,+\infty),((p-1)^{1/2}/2,1)\times\R
\).
\end{align*}
Applying \eqref{eq:mod2} with $D_0=x(t)$ from \eqref{OS}, we obtain
\begin{align*}
B^{-1/2}|D(t)-x(t)|+|c(t)-c_0|
\lesssim
\|\mathbf u(t)-e^{-x(t)\partial_x}\mathbf Q_{c_0}\|_{\mathcal H^1}\lesssim\delta.
\end{align*}
Consequently,
\begin{align}\label{eq:h1errorbounds}
\|\mathbf v(t)\|_{\mathcal H^1}
\lesssim \delta+|D(t)-x(t)|
\lesssim B^{1/2}\delta,
\quad
|c(t)-c_0|\lesssim\delta.
\end{align}
Insertion in \eqref{eq:gB2} of $\mathbf{u}=\mathbf{Q} _c(\cdot -D) + \mathbf{v}(\cdot -D)$,   after cancellations due to  \eqref{eq:soliton_eq}, yields
\begin{align}
  \label{eq:v1} \dot v_1 &= cv_1' +v_2' +(\dot D -c)  (  Q_c'  +v_1')-\dot c \partial _c  Q_c   \qquad   \text{  and} \\   \label{eq:v2} \dot v_2 &=  \partial _x (-\partial _x ^2 +\gamma ^2 -f'(Q_c) ) v_1 +c(cv_1' +v_2') \\& \nonumber -\( f(Q_c +v_1 ) - f(Q_c   ) -f'(Q_c)  v_1  \) '
 + (\dot D -c) (-cQ_c'  +v_2' ) +\dot c \partial _c (cQ_c)   \nonumber
\end{align} or also, equivalently, 
\begin{align}\label{eq:sist1}
   \dot {\mathbf{v} } &=  \partial _x \sigma _1   \diag \( -\partial _x^2 +1, 1     \) \mathbf{v} +c \mathbf{v}'- \partial _x(0, f  ( v_1 +Q_c) -f  (  Q_c)   )^\intercal \\& +(\dot D-c)  (  \mathbf{Q}_c   +\mathbf{v} )' -\dot c \partial _c\mathbf{Q}_c . \nonumber
\end{align}
By \eqref{eq:gB4syst},
$\mathbf u\in C^1([0,+\infty),
H^{-1}(\R)\times H^{-2}(\R))$.
Since $c,D\in C^1([0,+\infty))$, differentiation of
translations in these spaces gives
$
\mathbf v\in C^1([0,+\infty),
H^{-1}(\R)\times H^{-2}(\R)).
$
Thus \eqref{eq:v1}--\eqref{eq:sist1} hold for every $t\ge0$
as identities in $H^{-1}(\R)\times H^{-2}(\R)$, and hence
in the sense of distributions with respect to $x$.

We will repeatedly use the following reformulation of the eigenvalue problem associated with the linearized system of \eqref{eq:v1}--\eqref{eq:v2}.
\begin{lemma}\label{lem:transys}For  $\lambda = \gamma ^3 \Lambda $,  $\xi = \gamma x $, $v_1(x) = \widetilde{v}_1(  \xi)$    and  $v_2(x) = \widetilde{v}_4(  \xi)$   as in \eqref{itm:chvar1}--\eqref{eq:relkern2}, the   system
 \begin{align}
  \label{eq:eilinv1}   & cv_1' +v_2'  -  \lambda  v_1 =F_1    \text{  and} \\  &   \label{eq:eilinv2} \partial _x L _{+\gamma} v_1 +c(cv_1' +v_2') - \lambda  v_2  = F_2   
\end{align}
  is transformed into the equivalent  first order system
\small \begin{align}  \label{A_Decompose}&
  \frac{d}{d\xi}\widetilde{ V}  =( A _\infty^{(0)}  + \Lambda A _\infty^{(1)} (\gamma)  +q(\xi))\widetilde{ V} +\widetilde{\mathbf{F}},  \text{ where } q(\xi )= -p\begin{pmatrix}
0 & 0& 0 & 0 \\
0 & 0& 0 & 0\\
 (p-1)Q ^{p-2} Q' &    Q ^{p-1} & 0 & 0 \\
0 & 0 & 0& 0
\end{pmatrix} , \\&   A _\infty^{(0)}  = \begin{pmatrix}
0 & 1 & 0 &0 \\
0 & 0 & 1 & 0 \\
0 & 1 & 0 & 0 \\
0 & -c & 0 & 0
\end{pmatrix} , \  A _\infty^{(1)} (\gamma) = \begin{pmatrix}
0 & 0 & 0 &0 \\
0 & 0 & 0 & 0 \\
c & 0  & 0 & -1 \\
\gamma ^2  & 0 & 0 & 0
\end{pmatrix}  , \  \widetilde{V}:=   \begin{pmatrix}
\widetilde{v}_1 \\
\widetilde{v}_2\\
\widetilde{v}_3 \\
\widetilde{v}_4
\end{pmatrix}   ,  \   \widetilde{v} _2:=\widetilde{v} _1 ' ,\ 
   \widetilde{v} _3:=\widetilde{v} _1 ''.      \nonumber
 \end{align}\normalsize

\end{lemma}
\begin{proof} The substitutions      and elementary computations show  that  \eqref{eq:eilinv1}--\eqref{eq:eilinv2} is equivalent to 
\begin{equation}\label{eq:systres1}
  \begin{aligned}
     &  c\widetilde{v} _1'  +\widetilde{v}_4'    -  \gamma ^2 \Lambda  \widetilde{v}_1   = \gamma^{-1} {\widetilde{F}}_1   \text{  and} \\     &   \partial _\xi L _{+ } \widetilde{v}_1 +c \Lambda \widetilde{v}_1  -     \Lambda \widetilde{v}_4 = \gamma^{-3} \(   -c \widetilde{F} _1 +    \widetilde{F} _2   \)   \text{  where }L_+ = L _{+1}=-\partial _\xi ^2 +1 -p Q ^{p-1}(\xi )  ,
  \end{aligned}
\end{equation}
with  the differentiations in the $\xi $ variable.  
Writing  \eqref{eq:systres1} as  a first order system yields
\begin{equation}\label{eq:linsys00}
  \left\{
    \begin{array}{ll}
      \widetilde{v} _1 '& = \widetilde{v} _2  \hbox{ ,} \\
      \widetilde{v} _2 ' &= \widetilde{v} _3    \hbox{, } \\
      \widetilde{v} _3 ' &= c \Lambda  \widetilde{v} _1  -     \Lambda \widetilde{v}_4+  \widetilde{v} _2 -p (p-1) Q ^{p-2} Q' \widetilde{v} _1 -p   Q ^{p-1}  \widetilde{v} _2  + \gamma^{-3} \( c\widetilde{F}_1-\widetilde{F}_2 \)
             \hbox{ , } \\   \widetilde{v} _4 ' &= \gamma ^2 \Lambda \widetilde{v} _1- c \widetilde{v} _2 +  \gamma^{-1} {\widetilde{F}}_1
         \hbox{ .}
    \end{array}
  \right.
\end{equation}
so that for $\widetilde{V}= (\widetilde{v}_1,\widetilde{v}_2,\widetilde{v}_3,\widetilde{v}_4)^\intercal$ and for the $\widetilde{\mathbf{F}}$ defined in \eqref{eq:relkern2} we obtain \eqref{A_Decompose}.
    
\end{proof}

\textit{ Proof of Claim \textbf{iii}  in Proposition  \ref{prop:noeigh}.} We   assume  by contradiction that $\mathcal{L}_{[c]}\mathbf v=\im\tau\mathbf v$ for some $\tau \in \R \setminus\{0\}$ and for 
$0\neq \mathbf v\in D(\mathcal{L}_{[c]})$. Then there is a corresponding solution $\widetilde{V}$  of system \eqref{eq:linsys00} with $\widetilde{F}_1=\widetilde{F}_2=0$. By the theory of Jost functions in Sect. \ref{sec:jost}, there exist constants 
$\alpha , \beta \in \C$ such that $\widetilde{V} =\alpha  f_1 ( \cdot ,\Lambda  , \gamma  ) =\beta f_4 ( \cdot ,\Lambda  , \gamma  ) $. This,  by Lemmas  \ref{lem:roots1} and   \ref{lem:estimate_f},   implies that  $\widetilde{V}$ is smooth and $\widetilde{V}\xrightarrow {\xi  \to \infty    }0$ exponentially fast, with all its derivatives. Then the same is true of $\mathbf v$. 

Set  $\mathbf\Psi:=\mathbf L_c\mathbf v
=( 
\psi_1, 
\psi_2 )^\intercal.$ 
 Since  
$\partial_x\sigma_1\mathbf\Psi=\im\tau\mathbf v$,
applying $\mathbf L_c$ gives
$\mathbf L_c\partial_x\sigma_1\mathbf\Psi=\im\tau\mathbf\Psi $, that is 
\begin{align*}
(L_{+\gamma}+c^2)\psi_2'+c\psi_1' =\im\tau\psi_1 \text{ and } 
c\psi_2'+\psi_1' =\im\tau\psi_2. 
\end{align*}
We obtain 
\begin{align}\label{eq:noimag3}
L_{+\gamma}\psi_2'&=\im\tau(\psi_1-c\psi_2),
 \\ \label{eq:noimag4}
(\psi_1+c\psi_2)'&=\im\tau\psi_2.
\end{align}
Set  $b:=\psi_2$  and $h:= (\im\tau) ^{-1}   (\psi_1+cb ).$  Then  $h'=b$
 and
$\psi_1-cb=\im\tau h-2cb$.  Therefore \eqref{eq:noimag3} becomes
\begin{equation}\label{eq:noimag6}
L_{+\gamma}b'=-\tau^2h-2\im c\tau b.
\end{equation}
We know that  $b, h$ and all their  derivatives  decay
exponentially as $x\to\infty$.  
We introduce
 $\mu:=-\frac{Q_c'}{Q_c}$. 
 Notice we have $\mu'>0$.
Next,   by the exact same elementary computation in  \cite[Proposition 10.3]{CM25D1}, we have
\begin{equation*}
\operatorname{Re}\int_{\R}\mu\overline {b} L_{+\gamma}b'   dx
=-\frac{3(p-1)}{2(p+1)}
\int_{\R}Q_c^{p+1}
\left|\left(\frac b{Q_c}\right)'\right|^2dx \text{ for all $b\in \mathcal{S}(\R, \C)$,}
\end{equation*}
   in particular, for  our $b$. 
Alternatively, using \eqref{eq:noimag6} gives \small
\begin{align*}
\operatorname{Re}\int_{\R}\mu\overline b L_{+\gamma}b' dx
&=-\tau^2\operatorname{Re}\int_{\R}\mu\overline b h dx  = -\tau^2\operatorname{Re}\int_{\R}\mu\overline h ' h dx  =- \frac{\tau^2}{2}\int_{\R}\mu(|h|^2)'dx   = \frac{\tau^2} 2\int_{\R}\mu'|h|^2dx.
 \end{align*} \normalsize
Therefore together they give the following, which shows   $h=0$   and $b=h'=0$,
\begin{align*}
\frac{\tau^2}{2}\int_{\R}\mu'|h|^2dx
+\frac{3(p-1)}{2(p+1)}\int_{\R}Q_c^{p+1}
\left|\left(\frac b{Q_c}\right)'\right|^2dx=0 . 
\end{align*}
Hence $\mathbf v\in \Span \{ \xi _1[c] \}$.    Then 
$\im \tau \mathbf v =\mathcal{L}_{[c]}\mathbf v=0$, which  by $\tau \neq 0$ implies $\mathbf v=0$, thus a contradiction.
\qed

\section{Main estimates} \label{sec:mainest}

The proof of Theorem \ref{thm:main} is mainly  based on the following continuation argument. The material in this section is similar to the part of \S 3 in \cite{CM25D1} below  Proposition 3.4.

\begin{proposition}\label{prop:continuation}
For any small $ \epsilon >0$ and for large constants $A$ and $B$ as in formula \eqref{eq:relABg}
there exists  a    $\delta _0= \delta _0(\epsilon )   $ s.t.   if  for an interval   $I=[0,T]$ for some $T>0$ the following inequality holds,
\begin{align}
  \label{eq:main2} \|\mathbf{v} \| _{L^2(I,  { \Sigma }_{1A A_1} )} +
\|\mathbf{v} \| _{L^2(I,  { \Sigma }_{2A A_1} )} + \| v_1 \|_{ L^2( I , \widehat{\Sigma})}  +   \| \dot D - c \| _{L^2(I  )} + \| \dot c \| _{L^2(I  )}\le \epsilon ,
\end{align}
   and if  the constant in \eqref{OS} satisfies  $\delta  \in (0, \delta _0)$
then     \eqref{eq:main2} holds   for   $\epsilon$ replaced by $   o _{ B ^{-\frac{1}{2}} }(1) \epsilon $.
\end{proposition}

The proof splits  into three lemmas obtained assuming the hypotheses of Proposition  \ref{prop:continuation}.

\begin{lemma}\label{lem:lemdscrt} We have the estimates
 \begin{align} \label{eq:discrest1D}
   &     |\dot D -c |    \lesssim  B ^{-1/2}   (   \| \mathbf{v}  \|_{  { \Sigma }_{1A A_1}}   +  \| \mathbf{v}  \|_{  { \Sigma }_{2A A_1}}) \text{  and }\\& |\dot c   |     \lesssim B ^{-1} \delta ^{\frac{\min  \{ p ,2\}-1}{4}} (   \| \mathbf{v}  \|_{  { \Sigma }_{1A A_1}}   +  \| \mathbf{v}  \|_{  { \Sigma }_{2A A_1}})  . \label{eq:discrest1c}
\end{align}

\end{lemma}

\begin{lemma}[Virial Inequality]\label{prop:1virial}
 We have for a fixed $C _{c_0}$
 \begin{align}\label{eq:sec:1virial11}
   \|  \mathbf{v}  \|_{ L^2( I ,  \Sigma  _{1A A_1} )} &\le  C _{c_0}   \(   \sqrt[3]{ {\delta }} + \| v_1 \|_{ L^2( I , \widehat{\Sigma})}   \)    \text{ and}\\ \|  \mathbf{v}  \|_{ L^2( I ,  \Sigma  _{2A A_1}) } &\le C _{c_0} \(  \sqrt[3]{ {\delta }} +A_1 ^{-1} \| v_1 \|_{ L^2( I , \widehat{\Sigma})}   \) . \label{eq:sec:1virial12}
 \end{align}
 \end{lemma}

\begin{lemma}[Smoothing Inequality]\label{prop:smooth11}
 We have
 \begin{align}\label{eq:sec:smooth11}
   \| v_1 \|_{ L^2( I , \widehat{\Sigma})}  \lesssim       A _1  ^{3/2} B ^{-1}   \| \mathbf{v} \| _{L^2(I,  { \Sigma }_{1A A_1}   )} +     B ^{ 1/2}   \| \mathbf{v} \| _{L^2(I,  { \Sigma }_{2A A_1}   )} +  o _{B^{-1}} (1)  \epsilon     .
 \end{align}
 \end{lemma}
The proofs of lemmas  \ref{lem:lemdscrt} and \ref{prop:1virial} will be proved in \S \ref{sec:discr} and \S \ref{sec:virial}, respectively. Lemma \ref{prop:smooth11} will be proved in \S \ref{sec:smooth}.
\begin{proof}[Proof of Proposition \ref{prop:continuation}  assuming Lemma \ref{lem:lemdscrt}--\ref{prop:smooth11}.]
By $A_1 = B ^{3/5}$, see   \eqref{eq:relABg}, 
      \eqref{eq:sec:1virial11}--\eqref{eq:sec:smooth11} give \small
\begin{align*}
  & \| v_1 \|_{ L^2( I , \widehat{\Sigma})}   \lesssim       A _1  ^{3/2} B ^{-1}   \| \mathbf{v} \| _{L^2(I,  { \Sigma }_{1A A_1}   )} +     B ^{ 1/2}   \| \mathbf{v} \| _{L^2(I,  { \Sigma }_{2A A_1}   )} +  o _{B^{-1}} (1)  \epsilon
\\& \lesssim  o _{B^{-1}} (1)  \epsilon  +    \( A _1  ^{\frac{3}{2}  } B ^{-1}  + B ^{ 1/2} A _{1}  ^{-1}  \)  \| v_1 \|_{ L^2( I , \widehat{\Sigma})}  \lesssim  o _{B^{-1}} (1)  \epsilon  +   B ^{-\frac{1}{10}}      \| v_1 \|_{ L^2( I , \widehat{\Sigma})}    = o _{B^{-1}} (1)  \epsilon .
\end{align*}
\normalsize
Then $ \| v_1 \|_{ L^2( I , \widehat{\Sigma})}   =o _{B^{-1}} (1)  \epsilon . $ Inserting this in \eqref{eq:sec:1virial11} and \eqref{eq:sec:1virial12} shows that
\begin{align*}
    \| \mathbf{v} \| _{L^2(I,  { \Sigma }_{1A A_1}   )}  +       \| \mathbf{v} \| _{L^2(I,  { \Sigma }_{2A A_1}   )} =o _{B^{-1}} (1)  \epsilon .
\end{align*}
Finally using this information and Lemma  \ref{lem:lemdscrt} the proof of  Proposition \ref{prop:continuation} is completed. 
\end{proof}

\bigskip 

\begin{proof}[Proof of Theorem \ref{thm:main}]
By $\mathbf{u}\in C^0([0,+\infty),\mathcal{H}^1 )$ it is elementary that \eqref{eq:main2} can be extended to  $I=\R _+$,  in particular yielding \eqref{eq:asstab2}.
Consider  
\begin{equation}
  \begin{aligned}\label{eq:bolda}
   \mathbf{a}(t) &:= 2 ^{-1} \|  e^{- a\< x\>}   v _1(t)  \| _{L^2(\R )} ^2   .
  \end{aligned}
  \end{equation}
By \eqref{eq:v1}, $v_1\in C^0([0,+\infty),H^1)\cap C^1([0,+\infty),H^{-1})$. Hence $\mathbf a$ is locally absolutely continuous. The weak Leibniz rule, together with \eqref{eq:discrest1D}--\eqref{eq:discrest1c}, gives
\begin{align*}
|\dot{\mathbf a}(t)| = \left|\left\langle e^{-2a\langle x\rangle}v_1, cv_1'+v_2'+(\dot D-c)(Q_c'+v_1') -\dot c\partial_cQ_c \right\rangle\right| \lesssim\|\mathbf v(t)\|_{\mathcal H^1}^2.
\end{align*}
Here the pairing involving $v_2'$ is understood by integration by parts. 
Thus \eqref{eq:h1errorbounds} implies that $\mathbf a$ is uniformly Lipschitz. Moreover, $\mathbf a\in L^1(\R_+)$ by \eqref{eq:asstab2}. 
A nonnegative, uniformly continuous, integrable function tends to zero, so \eqref{eq:asstab20} follows.
Next we prove  \eqref{eq:asstab3}.  Since $M( \mathbf{Q}_{c}) $ is strictly increasing   in $c$ for $c\in ( (  p-1) ^{1/2}/2 , 1   )$ , it suffices to show that  $M(\mathbf{Q}_{c (t)})$ converges as $t\to +\infty$. 
By $\eta_2[c]=(q'(c))^{-1}\sigma_1\mathbf Q_c$, the second orthogonality condition in \eqref{61}   gives 
\begin{align*}
\<\mathbf v(t),\sigma_1\mathbf Q_{c(t)}\> =q'(c(t))\<\mathbf v(t),\eta_2[c(t)]\>=0.
\end{align*}
Consequently, conservation and translation invariance of
$M$ yield
\begin{align}\label{eq:equipart1}
 M(\mathbf u_0)
= M(\mathbf Q_{c(t)})+M(\mathbf v(t))
   \text{ for all }t\ge0.
\end{align}
Thus, our task is now to prove $\frac{d}{dt}M(\mathbf{v}) \in L^1(\R _+) $, which is sufficient to show the convergence of $M(\mathbf{v}(t))$ as $t\to +\infty$.  For $U(t)$ the operator in \eqref{eq:group}, 
let $\widetilde{\mathbf{v}}(t) =U^{-1}(t)\mathbf{v}(t) = U(-t)\mathbf{v}(t)$. Then  
\begin{align*} 
   \dot {\widetilde{\mathbf{v}} } &=   \dot D  \widetilde{\mathbf{v}}'- U(-t)  \partial _x(0, f  ( v_1 +Q_c) -f  (  Q_c)   )^\intercal  +(\dot D-c) U(-t)  \mathbf{Q}_c  ' -\dot c  U(-t)\partial _c\mathbf{Q}_c   \nonumber
\end{align*}
and,  for the second equality proceeding like for the differentiation of $\mathbf{a}$    and      using $  \sigma _1 U(-t)\sigma _1 =U^*(t)$,    \small
\begin{align*}
\frac{d}{dt}M(\mathbf{v})&= \frac{d}{dt}M(\widetilde{\mathbf{v}})  = \< \widetilde{\mathbf{v}} , \sigma _1 \dot { \widetilde{\mathbf{v}} } \>
= (\dot D -c ) \<\widetilde{\mathbf{v}}, \sigma _1 U(-t)\xi _1[c]\> -  \dot c \<\widetilde{\mathbf{v}}, \sigma _1  U(-t)\xi _2[c]\>  \\&   -\<\widetilde{\mathbf{v}} ,\sigma _1 U(-t)  \partial _x(0, f  ( v_1 +Q_c) -f  (  Q_c)   )^\intercal \>
 \\& =     \left [ (\dot D -c ) \<\mathbf{v}, \sigma _1  \xi _1[c]\> -  \dot c \<\mathbf{v},  \sigma _1\xi _2[c]\>  \right ]   +\<v _1',f ( v_1 +Q_c)   -f   (  Q_c)  \>
=:I+II .
\end{align*}  \normalsize
By the exponential decay of $\xi_1[c]$ and $\xi_2[c]$,
\begin{align*}
|I(t)|
\lesssim (|\dot D-c|+|\dot c|)
\bigl(
\|\mathbf v(t)\|_{\Sigma_{1AA_1}}
+\|\mathbf v(t)\|_{\Sigma_{2AA_1}}
\bigr).
\end{align*}
Thus \eqref{eq:main2} on $\R_+$ implies $I\in L^1(\R_+)$.

For $1<p\le2$,  we have $|f'(r)-f'(s)|\le C_p |r-s|^{p-1}$
 for $ r,s\in\R$, hence 
\begin{align*}
II(t)
=\int_{\R}v_1'
\bigl(f(Q_c+v_1)-f(Q_c)-f(v_1)\bigr)\,dx
=\int_{\R}v_1'v_1
\int_0^1\bigl(f'(Q_c+sv_1)-f'(sv_1)\bigr)\,ds\,dx.
\end{align*}
For $A$ sufficiently large,  
  $Q_c^{p-1}\lesssim\zeta_A^2$  uniformly for $c$ near $c_0$
 gives
\begin{align*}
|II(t)|
\lesssim\int_{\R}Q_c^{p-1}|v_1v_1'|\,dx
\lesssim\|\zeta_Av_1(t)\|_{H^1}^2
\lesssim
\|\mathbf v(t)\|_{\Sigma_{1AA_1}}^2
+\|\mathbf v(t)\|_{\Sigma_{2AA_1}}^2.
\end{align*}
By \eqref{eq:main2}, the last expression belongs to $L^1(\R_+)$.
This proves $II\in L^1(\R_+)$.
\end{proof}

\section{Proof of Lemma \ref{lem:lemdscrt} }\label{sec:discr}

 The proof is standard and is very similar   to \cite[\S 4]{CM25D1}.
 From \eqref{61}
\begin{align*}
    \< \dot { \mathbf{v} }  , \zeta_{B} \eta _1[c] \> =-\dot c   \<   \mathbf{v}  , \zeta_{B} \partial _c\eta _1[c] \>  ,
\end{align*}
  using \eqref{eq:ker1} and applying     $ \< \quad  , \zeta_{B} \eta _1[c] \>  $ to \eqref{eq:sist1} we
  obtain
\begin{align}&
  - ( \dot D -c )   \< \xi _1[c] +  \mathbf{v}' , \zeta_{B} \eta _1[c] \>      + \dot c  \left [ \< \xi _2[c] , \zeta_{B} \eta _1[c] \>    -    \<   \mathbf{v}  , \zeta_{B} \partial _c\eta _1[c] \> \right ] \label{eq:seceq}\\& = \<    \mathcal{L}_{[c]}    \mathbf{v}   ,  \zeta_{B}  \eta _1[c] \>  + \<  (0, f  ( Q_c + v_1   )-f'(Q_c)v_1  -f  (  Q_c)   )^\intercal   ,\( \zeta_{B} \eta _1[c]\) ' \>   =: \mathcal{A} _{1 }+\mathcal{A} _{2 }. \nonumber
\end{align}
We have
\begin{align*}
  \mathcal{A} _{1  }=  \<    \mathbf{v} ,  \zeta_{B}  \mathcal{L}_{[c]}^* \eta _1[c] \> + \<    \mathbf{v} , [  \mathcal{L}_{[c]}^*, \zeta_{B} ] \eta _1 [c] \> =:\mathcal{A} _{1 1 }+ \mathcal{A} _{1  2}.
\end{align*}
Now, by \eqref{eq:M2.4},  \eqref{61},    $1= \vartheta _{1A_1} + \vartheta _{2A_1} $ and the definition of $  \zeta_{A}$ in \eqref{def:zetaC}, we have
\begin{align*}
  \mathcal{A} _{1 1}&=    \<      \mathbf{v} ,  \zeta_{B}  \mathcal{L}_{[c]}^* \eta _1[c] \> = -\<     \mathbf{v} , \zeta_{B}  \eta _2[c] \>=\<     \mathbf{v} , (1-\zeta_{B})  \eta _2[c] \>   \\&=\< (\vartheta _{1A_1} + \vartheta _{2A_1} ) \zeta_{A}  \mathbf{v}  ,   \zeta_{A} ^{-1} \(1- e  ^{-\frac{|x|}{B}(1-\chi )}    \)  \eta _2[c] \>   .
\end{align*}
Then by the exponential decay of $\eta _2[c]$, consequence of the its definition \eqref{eq:ker2} and \eqref{eq:QQ}--\eqref{eq:QQ1},   
\begin{align}\label{eq:A11bound}
  |\mathcal{A} _{1 1 }| & \lesssim B ^{-1}   (   \| \mathbf{v}  \|_{  { \Sigma }_{1A A_1}}   +  \| \mathbf{v}  \|_{  { \Sigma }_{2A A_1}})
\end{align}
and, by $|\zeta_B^{(k)}|\lesssim B^{-1}\zeta_B$ for $1\le k\le3$, $\|\zeta_B/\zeta_A\|_{L^2(\R)}\lesssim B^{1/2}$, and $\|\eta_1[c]\|_{W^{n,\infty}(\R)} \leq C(n,c_0)$ for all $n$,
\begin{align}\label{eq:A12bound}
  |\mathcal{A} _{1 2 }| & \lesssim B ^{-1}  \left  \| \frac{ \zeta_{B}  }{\zeta_{A}}  \zeta_{A} (\vartheta _{1A_1} + \vartheta _{2A_1})  \mathbf{v}    \right \|_{   L^1( \R )   } \lesssim  B ^{-1/2}  (   \| \mathbf{v}  \|_{  { \Sigma }_{1A A_1}}   +  \| \mathbf{v}  \|_{  { \Sigma }_{2A A_1}}) .
\end{align}
   Next we claim
\begin{align}\label{eq:boundA2}
    | \mathcal{A}_2|\lesssim B ^{-1} \delta ^{\frac{\min  \{ p ,2\}-1}{4}} (\| v_1\|_{\widehat{\Sigma} }   +  \| \mathbf{v}  \|_{  { \Sigma }_{2A A_1}}) .
\end{align}  Like  in \cite {CM25D1}
   for $s\in (0,1)$ we partition the line    as
\begin{align}&
   \Omega_{1,t ,s } :=\{x\in \R\ |\  2| s v_1(t,x)| \le  Q   _{c (t)}(x) \}  \text{ and }\label{eq:partition}\\&  \Omega_{2,t,s } :=\R\setminus \Omega_{1,t ,s }  =\{x\in \R\ |\ 2| s v_1(t,x)|> Q   _{c (t)}(x)\} .\nonumber
\end{align}
For $\eta _{1 }[c] =( \eta _{11}[c] ,\eta _{12}[c])^\intercal $ and using   \eqref{eq:partition}  we have 
\begin{align*}
   \mathcal{A} _{ 2 }& = \sum _{i=1,2}\int _{0}^{1} ds \int _{\Omega_{i,t,s }} \(    f'( Q_c +s v_1 ) -  f'( Q  _{c} )   \) v_1 \( \zeta_{B} \eta _{12}[c]\) ' dx  =:  \sum _{i=1,2}\mathcal{A} _{2 i}  .
\end{align*}
 For $1<p\le2$, the range of $p$'s we consider in the proof,  recall the provisio in Remark \ref{rem:udermain},
 the concavity of  $f'$ along with $f'(0)=0$ gives
 \begin{align*}
 |f'(Q_c+sv_1)-f'(Q_c)|
 \lesssim_p s^{p-1}|v_1|^{p-1}.
 \end{align*}
 On $\Omega_{1,t,s}$, the inequality $2s|v_1|\le Q_c$ implies
 \begin{align*}
 s^{p-1}|v_1|^p
 \lesssim
 s^{(p-1)/2}Q_c^{(p-1)/2}|v_1|^{(p+1)/2}.
 \end{align*}
 Using this inequality, the boundedness of
 $(\zeta_B\eta_{12}[c])'$, and \eqref{eq:h1errorbounds}, we obtain
 \begin{align*}
 &|\mathcal A_{21}|
 \lesssim
 \int_0^1s^{p-1}\int_{\Omega_{1,t,s}}
 |v_1|^p|(\zeta_B\eta_{12}[c])'|\,dx\,ds\\
 &\lesssim
 \int_0^1s^{(p-1)/2}\int_{\Omega_{1,t,s}}
 Q_c^{(p-1)/2}|v_1|^{(p+1)/2}\,dx\,ds\lesssim
 \|v_1\|_{L^\infty}^{(p-1)/2}
 \|v_1\|_{\widehat\Sigma}  \lesssim
 B^{(p-1)/4}\delta^{(p-1)/2}
 \|v_1\|_{\widehat\Sigma}.
 \end{align*}
 Here we used the exponential decay of $Q_c$ and chose
 $\kappa$ in \eqref{eq:sigmatilde} sufficiently small. By the same concavity argument, $|f'(Q_c+sv_1)-f'(Q_c)| \lesssim_p  s^{p-1}|v_1|^{p-1} \le |v_1|^{p-1}$ on $\Omega_{2,t,s}$, we obtain
\begin{align*}
   |\mathcal{A} _{22}|&\lesssim  \int _{\R}|v_1| ^{p }   |\zeta_{B}' \eta _1[c] |   dx + \int _{\R}|v_1| ^{p }  |\zeta_{B}  \eta _1' [c] | dx
  =: \mathcal{A} _{221}+\mathcal{A} _{222}.
 \end{align*}
By choosing $\kappa$ sufficiently small, the profile formulas give
\begin{align*}
\|\cosh(\gamma\kappa x)\eta_1'[c]\|_{L^2}\lesssim1,
\end{align*}
uniformly for $c$ near $c_0$. Hence
\begin{align*}
\mathcal A_{222}\lesssim \|v_1\|_{L^\infty}^{p-1}\|v_1\|_{\widehat\Sigma}\lesssim
B^{(p-1)/2}\delta^{p-1}\|v_1\|_{\widehat\Sigma}\lesssim
B^{(p-1)/4}\delta^{(p-1)/2}\|v_1\|_{\widehat\Sigma},
\end{align*}
where the last inequality uses $B^{1/2}\delta\le1$.
 Using the partition of unity $1= \vartheta _{1A_1} + \vartheta _{2A_1} $    we write
\begin{align*}
   \mathcal{A} _{221} &=  \int _{\R}|v_1| ^{p }    \vartheta _{1A_1}  |\zeta_{B}' \eta _1[c] |   dx   +  \int _{\R}|v_1| ^{p }    \vartheta _{2A_1}  |\zeta_{B}' \eta _1[c] |   dx =:\mathcal{ A} _{2211}+\mathcal{A} _{2212}
\end{align*}
where, by $\| \zeta_{B}' /\zeta_{A}      \| _{L^2(\R )} \lesssim   B^{-1/2} $ and $   \|    v _1  \| _{L^\infty(\R )}^{p-1}\lesssim  B^{\frac{p-1}{2}} \delta  ^{\frac{p-1}{2}}$ by \eqref{eq:h1errorbounds} and Sobolev embedding, 
\begin{align*}
  |\mathcal{A} _{2212}|&\lesssim   \left \|    \frac{\zeta_{B}'}{\zeta_{A} }  \right \| _{L^2(\R )}      \left \|    \eta _1[c]  \right \| _{L^\infty (\R )}    \left \|    v _1 \right \| _{L^\infty(\R )}^{p-1}   \| \mathbf{v}  \|_{  { \Sigma }_{2A A_1}}\lesssim  B^{\frac{p-2}{2}} \delta  ^{\frac{p-1}{2}} \| \mathbf{v}  \|_{  { \Sigma }_{2A A_1}}.
\end{align*}
Finally, since $0<\gamma\le1$, we choose $\delta>0$ so small that $\delta^{\frac{p-1}{4}}\left(B^{\frac{p+3}{4}}+B^{\frac p2}\cosh(2\kappa A_1)\right)\ll1$. Then
\begin{align}
|\mathcal A_{2211}| & \lesssim \|\cosh(\gamma\kappa x)\|_{L^\infty(|x|\le2A_1)} \|\zeta_B'\|_{L^2} \|v_1\|_{L^\infty}^{p-1} \|\sech(\gamma\kappa x)v_1\|_{L^2}
\nonumber\\ 
&\lesssim
B^{\frac{p-2}{2}}\cosh(2\gamma\kappa A_1) \delta^{\frac{p-1}{2}}
\|v_1\|_{\widehat\Sigma}
\lesssim
B^{-1}\delta^{\frac{p-1}{4}}
\|v_1\|_{\widehat\Sigma}.
\label{eq:smallBdelta}
\end{align}
Together with the preceding estimates for
$\mathcal A_{21}$, $\mathcal A_{222}$ and $\mathcal A_{2212}$,
this proves \eqref{eq:boundA2}.
Notice also that this choice implies $B^2\delta\ll1$.
Finally  in \eqref{eq:seceq} for the coefficients of $\dot D -c $ and  $\dot  c $  we have
\begin{align}\label{eq:bdcoeff1}&
    \< \xi _1[c] +  \mathbf{v}' , \zeta_{B} \eta _1[c] \>  =1+ O(B^{-1})  \text{   and }\\& \< \xi _2[c] , \zeta_{B} \eta _1[c] \>    -    \<   \mathbf{v}  , \zeta_{B} \partial _c\eta _1[c] \> =  O(B^{-1}) .\nonumber
\end{align}
From the above and 
\begin{equation}\label{eq:reltidesigetal}
    \| v_1  \|_{   \widehat{\Sigma} } \lesssim  \| \mathbf{v}  \|_{  { \Sigma }_{1A A_1}}   +  \| \mathbf{v}  \|_{  { \Sigma }_{2A A_1}}.
\end{equation}
we derive
\begin{align} \label{eq:discrest1D1}
   |\dot D -c |   & \lesssim B^{-1} |\dot c | +B ^{-1/2 } (   \| \mathbf{v}  \|_{  { \Sigma }_{1A A_1}}   +  \| \mathbf{v}  \|_{  { \Sigma }_{2A A_1}}) .
\end{align}
Applying     $ \< \quad  ,  \eta _2[c] \>  $ to \eqref{eq:sist1}
  and using \eqref{eq:M2.4} we obtain
\begin{align*}&
  - ( \dot D -c )   \< \xi _1[c] +  \mathbf{v}' ,   \eta _2[c] \>      +\dot c  \left [1    -    \<   \mathbf{v}  ,   \partial _c\eta _2[c] \> \right ] 
  \\& =    \<  (0, f  ( Q_c + v_1   )-f'(Q_c)v_1  -f  (  Q_c)   )^\intercal   ,  \eta _2[c]  ' \>  \nonumber  .
\end{align*}
From 
\begin{align*} &
    \< \xi _1[c] +  \mathbf{v}' ,   \eta _2[c] \>  = -\<    \mathbf{v}  ,   \eta _2'[c] \> =O( \|  \mathbf{v}\|  _{L^2}     ) =O( \delta B^{1/2}    ) \text{   and }\\& \< \xi _2[c] ,   \eta _2[c] \>    -    \<   \mathbf{v}  ,   \partial _c\eta _2[c] \> =  1 + O( \delta B^{1/2}    )\approx 1 , 
\end{align*}
the identity $\mathcal L_{[c]}^*\eta_2[c]=0$,
the H\"older estimate for $f'$, the exponential decay of
$\eta_2'[c]$, and \eqref{eq:reltidesigetal}, we obtain
\begin{align*}
|\dot c|
\lesssim
\delta B^{1/2}|\dot D-c|
+B^{-1}\delta^{\frac{p-1}{4}}
\left(
\|\mathbf v\|_{\Sigma_{1AA_1}}
+\|\mathbf v\|_{\Sigma_{2AA_1}}
\right).
\end{align*}
Indeed, by \eqref{eq:h1errorbounds} and the choice of $\delta$ above,
the nonlinear term satisfies
\begin{align*}
\left|
\langle
(0,f(Q_c+v_1)-f(Q_c)-f'(Q_c)v_1)^\intercal,
\eta_2'[c]
\rangle
\right|
\lesssim
\|v_1\|_{L^\infty}^{p-1}\|v_1\|_{\widehat\Sigma}
\lesssim
B^{\frac{p-1}{2}}\delta^{p-1}
\|v_1\|_{\widehat\Sigma}
\lesssim
B^{-1}\delta^{\frac{p-1}{4}}
\|v_1\|_{\widehat\Sigma}.
\end{align*}
Substituting \eqref{eq:discrest1D1} gives
\begin{align*}
|\dot c|
\lesssim
\delta B^{-1/2}|\dot c|
+\left(\delta+B^{-1}\delta^{\frac{p-1}{4}}\right)
\left(
\|\mathbf v\|_{\Sigma_{1AA_1}}
+\|\mathbf v\|_{\Sigma_{2AA_1}}
\right).
\end{align*}
The choice of $\delta$ above implies $\delta B^{-1/2}\ll1$ and
$B\delta^{1-\frac{p-1}{4}}\ll1$.
Absorbing the first term therefore proves \eqref{eq:discrest1c}.
Substitution in \eqref{eq:discrest1D1} then proves \eqref{eq:discrest1D}, completing the proof of Lemma \ref{lem:lemdscrt} for $1<p\le 2$.

\section{The virial inequality: proof of Lemma \ref{prop:1virial} }\label{sec:virial}

Using the notation in \eqref{def:zetaC}  we set
\begin{align}\nonumber 
\mathcal{I}_{i} := 2^{-1}\<  \mathbf{v}  , \varphi _{iAA_1} \sigma _1\mathbf{v}  \>    \text{  for $i=1,2$ }.
\end{align}
This section is essentially a replica of  \cite[Sect. 5]{CM25D1}, although here the estimates are a refinement of   \cite[Lemma  5.1]{CM25D1}: the estimates below can be proved also in   \cite[Lemma  5.1]{CM25D1}.

\begin{lemma}\label{lem:1stV1}
There exist fixed constants $C_j=C_j(c_0)>0$ such that
\begin{align}
\label{eq:lem:1stV11}
\|\mathbf v\|_{\Sigma_{1AA_1}}^2
&\leq
-C_1\frac{d}{dt}\mathcal I_1
-C_2B^{-1}\frac{d}{dt}\mathcal I_2
+
C_3
\|v_1\|_{\widehat{\Sigma}}^2
,
\\
\label{eq:lem:1stV12}
\|\mathbf v\|_{\Sigma_{2AA_1}}^2
&\leq
-C_4A_1^{-2} \frac{d}{dt}\mathcal I_1-C_5\frac{d}{dt}\mathcal I_2
+
C_6  
A_1^{-2}\|v_1\|_{\widehat{\Sigma}}^2
.
\end{align}
\end{lemma}
\begin{proof}
  To justify the identities below, set
$J_\varepsilon=(1-\varepsilon\partial_x^2)^{-1}$
and $\mathbf v^\varepsilon=J_\varepsilon\mathbf v$.
Applying $J_\varepsilon$ to \eqref{eq:sist1}, we obtain
$\mathbf v^\varepsilon\in C([0,T],H^3\times H^2)
\cap C^1([0,T],L^2\times L^2)$ for every $T>0$.
We can therefore differentiate
$\frac12\<\mathbf v^\varepsilon,
\varphi_{iAA_1}\sigma_1\mathbf v^\varepsilon\>$.
The nonlinear term in the regularized equation is
$-\partial_x(0,J_\varepsilon(f(Q_c+v_1)-f(Q_c)))^\intercal$.
After integration by parts, the contribution of the constant
coefficient linear operator is
\begin{align*}
-\frac32\int_{\R}
\varphi_{iAA_1}'|(v_1^\varepsilon)'|^2\,dx
+\frac12\int_{\R}
\varphi_{iAA_1}'''|v_1^\varepsilon|^2\,dx
-\frac12\int_{\R}\varphi_{iAA_1}'
\left(|v_1^\varepsilon|^2+|v_2^\varepsilon|^2\right)\,dx.
\end{align*}
Since $\mathbf v^\varepsilon\to\mathbf v$ in
$C([0,T],\mathcal H^1)$ and
\begin{align*}
J_\varepsilon(f(Q_c+v_1)-f(Q_c))
\longrightarrow f(Q_c+v_1)-f(Q_c)
\quad\text{in }C([0,T],L^2),
\end{align*}
all terms converge uniformly on $[0,T]$ after integration by
parts. Passing to the limit in the time-integrated identity
shows that $\mathcal I_i$ is locally absolutely continuous
and satisfies the following identity. Derivative pairings
below are understood through these integrations by parts.
\begin{align}\label{eq:vir2}
\dot{\mathcal I}_i
&=
\<\partial_x\diag(-\partial_x^2+1,1)\mathbf v,
\varphi_{iAA_1}\mathbf v\>
+c\<\mathbf v',\varphi_{iAA_1}\sigma_1\mathbf v\>
 \\
&\quad
+\<f(v_1+Q_c)-f(Q_c),(\varphi_{iAA_1}v_1)'\>
+(\dot D-c)\<\xi_1[c],\varphi_{iAA_1}\sigma_1\mathbf v\>
\nonumber\\
&\quad
-\dot c\<\xi_2[c],\sigma_1\varphi_{iAA_1}\mathbf v\>
+(\dot D-c)\<\mathbf v',\varphi_{iAA_1}\sigma_1\mathbf v\>
=:\sum_{1\le k\le6}B_k^{(i)}.
\nonumber
\end{align}
 Only for this proof, we   set $\omega ^{(i)}:= \zeta_A  \vartheta _{iA_1}\mathbf{v} = (\zeta_A  \vartheta _{iA_1}v_1, \zeta_A  \vartheta _{iA_1}v_2)^\intercal  =:(\omega ^{(i)}_1,  \omega ^{(i)}_2)^\intercal       $.
 We split
 \begin{align*}
   B_1^{(i)}=  \<   -{v}_1 ''' +v_1 '  ,\varphi _{iAA_1}  {v} _1  \>+\<   {v}_2 '    ,\varphi _{iAA_1}  {v} _2  \> =: B _{11}^{(i)}+ B^{(i)}_{12}.
 \end{align*}
 We claim that
 \begin{align} \label{eq:B1}
    B_1^{(i)} &=  -  \frac{3}{2} \| (\omega ^{(i)} _1 )'\| _{L^2}^{2}  -\frac{1}{2}  \|\omega ^{(i)} \| _{L^2}^{2} +  A _1 ^{-2}  O\(   \|\omega ^{(1)} _1\| _{L^2}^{2} +\|\omega ^{(2)} _1 \| _{L^2}^{2}    \)
 \end{align}
 It is straightforward  that
 \begin{align*}
  B^{(i)}_{12} =- 2 ^{-1} \|\omega ^{(i)}_2 \| _{L^2}^{2}  .
 \end{align*}
 Next we claim   \small
 \begin{align} \label{eq:B11}
   B _{11}^{(i)} &=  -  \frac{3}{2} \| (\omega ^{(i)} _1 )'\| _{L^2}^{2}  -\frac{1}{2}  \|\omega ^{(i)}_1 \| _{L^2}^{2} +   B _{111}^{(i)} \text{ with }   B _{111}^{(i )}:=\langle v_1^2, ((\zeta_A \vartheta _{i A_1}  ) ')^2-\frac{1}2 (\zeta_A \vartheta _{i A_1}  ) '' \zeta_A \vartheta _{i A_1}\rangle .
 \end{align}\normalsize
  Before proving \eqref{eq:B11}, we write
\begin{align*}
B_{111}^{(i)}
&=
\langle
\left((\zeta_A')^2-\frac12\zeta_A\zeta_A''\right)
\vartheta_{iA_1}^2,v_1^2
\rangle
+\<\zeta_A'\vartheta_{iA_1}'v_1,
\zeta_A\vartheta_{iA_1}v_1\>
+\langle
(\vartheta_{iA_1}')^2
-\frac12\vartheta_{iA_1}''\vartheta_{iA_1},
\zeta_A^2v_1^2
\rangle\\&
=:\sum_{1\le k\le3}B_{111k}^{(i)}.
\end{align*}
By the definition of $\zeta_A$,
$|\zeta_A'|\lesssim A^{-1}\zeta_A$ and
$|\zeta_A''|\lesssim A^{-1}\zeta_A$. Hence
\begin{align*}
|B_{1111}^{(i)}|
&\lesssim
A^{-1}\|\omega_1^{(i)}\|_{L^2}^2   \text{ and}\\
|B_{1112}^{(i)}+B_{1113}^{(i)}|
&\lesssim
A^{-1}\left(
\|\omega_1^{(1)}\|_{L^2}^2
+\|\omega_1^{(2)}\|_{L^2}^2
\right)
+A_1^{-2}\|v_1\|_{L^2(A_1\le|x|\le2A_1)}^2.
\end{align*}
By \eqref{eq:relABg}, we have $A^{-1}\lesssim A_1^{-2}$.
Using the partition of unity $  \vartheta _{1A_1} + \vartheta _{2A_1}=1 $ and the fact that $\zeta _A \approx 1$ for  $  |x|\le 2A_1$ we have
\begin{align*}
  A _1 ^{-2}\| v_1 \| ^{2}_{L^2(A_1\le |x|\le 2A_1)}\lesssim    A _1 ^{-2} \(  \|\omega ^{(1)}_1 \| _{L^2}^{2} +\|\omega ^{(2)}_1 \| _{L^2}^{2} \).
\end{align*}
It is evident that \eqref{eq:B11} implies \eqref{eq:B1}. We prove now \eqref{eq:B11}. First we write
 \begin{align*}
   B _{11}^{(i)} = -\<   {v}_1 '''   ,\varphi _{iAA_1}  {v} _1  \>  -   \|\omega ^{(i)}_1 \| _{L^2}^2    /2 
 \end{align*}
 and then, by integration by parts,
\begin{align*}
-\<v_1''',\varphi_{iAA_1}v_1\>
&=
\<v_1'',\zeta_A^2\vartheta_{iA_1}^2v_1\>
-\frac12\|\zeta_A\vartheta_{iA_1}v_1'\|_{L^2}^2 =
-\frac32\|(\omega_1^{(i)})'\|_{L^2}^2
+B_{111}^{(i)}.
\end{align*}
This yields \eqref{eq:B11} and hence \eqref{eq:B1}.
Proceeding with the other terms in \eqref{eq:vir2},
 it is elementary that
\begin{align}
 \nonumber B_2^{(i)} &= -c/2  \ \langle \omega ^{(i)} , \sigma _1 \omega ^{(i)} \rangle  \le c/2 \| \omega ^{(i)} \| _{L^2}^2.
\end{align}
We leave aside  momentarily $ B_3^{(i)}$.  We write
\begin{align*}
  B_4^{(i)}&= (\dot D -c )\<  \zeta _A ^{-1}  \xi _1[c]    \varphi _{iAA_1} , \zeta _A (\vartheta _{1A_1} + \vartheta _{2A_1} )  \sigma _1        \mathbf{v}   \>
\end{align*}
and using  the exponential decay of $\xi _j[c]$ and   \begin{align}
  \label{eq:suppvarphi1}\supp  \varphi _{2AA_1}\subseteq \supp  \vartheta _{2A_1}\subseteq \R \backslash (-A_1,A_1 ), 
\end{align}
by \eqref {eq:discrest1D}    we obtain for $i=1,2$
\begin{align*} &
   |B_4^{(i)}| \lesssim     A _1 ^{ -100 (i-1)}      B ^{-1/2}   (   \| \mathbf{v}  \|_{  { \Sigma }_{1A A_1}}   +  \| \mathbf{v}  \|_{  { \Sigma }_{2A A_1}})        (   \| \mathbf{v}  \|_{  { \Sigma }_{1A A_1}}   +  A _1 ^{ -100  } \| \mathbf{v}  \|_{  { \Sigma }_{2A A_1}}).
\end{align*}
Then
\begin{align}
\nonumber   |B_4^{(2)}| \lesssim     A _1 ^{ -100 }      (   \| \mathbf{v}  \|_{  { \Sigma }_{1A A_1}}  ^2 +  \| \mathbf{v}  \|_{  { \Sigma }_{2A A_1}} ^2 )
\end{align}
and
\begin{align*}
  |B_4^{(1)}| &\lesssim      B ^{-1/2}  \| \mathbf{v}  \|_{  { \Sigma }_{1A A_1}}  ^2 +  A _1 ^{ -100 } \| \mathbf{v}  \|_{  { \Sigma }_{2A A_1}}  ^2  +  B ^{-1/2} \| \mathbf{v}  \|_{  { \Sigma }_{1A A_1}}   \| \mathbf{v}  \|_{  { \Sigma }_{2A A_1}}
\end{align*}
for which for any preassigned fixed small $\varepsilon _1>0$   we derive 
\begin{align}\label{eq:b41}&
   |B_4^{(1)}| \lesssim        \varepsilon _1  \| \mathbf{v}  \|_{  { \Sigma }_{1A A_1}}  ^2 + \varepsilon _1 ^{-1}  B ^{-1 } \| \mathbf{v}  \|_{  { \Sigma }_{2A A_1}} ^2 . 
\end{align}
With a similar argument,  by   \eqref{eq:discrest1c} and the fact that $\delta >0$ is very small,   \begin{align}  \label{eq:b52}  |B_5^{(i)}| \lesssim A _1 ^{ -100 }      (   \| \mathbf{v}  \|_{  { \Sigma }_{1A A_1}}  ^2 +  \| \mathbf{v}  \|_{  { \Sigma }_{2A A_1}} ^2 ) \text{ for }i=1,2.
\end{align}
 By integration by parts,
\begin{align*}
B_6^{(i)}
&=
-\frac{\dot D-c}{2}
\langle\omega^{(i)},\sigma_1\omega^{(i)}\rangle  \text{ and so }
|B_6^{(i)}|
 \le
\frac{|\dot D-c|}{2}\|\omega^{(i)}\|_{L^2}^2
\lesssim
\delta\|\omega^{(i)}\|_{\mathcal H^1}^2.
\end{align*}
Here we used \eqref{eq:h1errorbounds} and
\eqref{eq:discrest1D}, which give $|\dot D-c|\lesssim\delta$.
Turning to  $ B_3^{(i)}$,  we   are basically in the same situation of   \cite{CM25D1}.
We write
\begin{align*}
 B_3  ^{(i)}= \<   f( Q  _{c} + v_1  )  -  f( Q  _{c} )-f(v_1)   ,(\varphi _{iAA_1} v_1)'  \>  + \< f(v_1)   ,(\varphi _{iAA_1} v_1)' \>=: B_{31}^{(i)}+B_{32}^{(i)} .
\end{align*}
Then, by \eqref{eq:h1errorbounds}, 
\begin{align*}
 |B_{32}^{(i)}|&=  | \< f(v_1) v_1 -F(v_1)   ,\zeta_A^2  \vartheta _{iA_1} ^2  \> |\lesssim \| v_1\| _{L^\infty (\R )}^{p-1} \| \omega ^{(i)}  \|_{     L^2(\R )}^2  \lesssim B ^{\frac{p-1}{2}} \delta^{p-1} \| \omega ^{(i)}  \|_{     L^2(\R )}^2.
\end{align*}
 Using the splitting \eqref{eq:partition} for $s=1$
 we split
\begin{align*}
  B_{31}^{(i)}= \sum _{k=1,2}  \<   f( Q  _{c} + v_1  )  -  f( Q  _{c} )-f(v_1)   , ( \varphi _{iAA_1} v_1)'  \>_{L^2(\Omega_{k,t,1  })}=: \sum _{k=1,2} B_{31k}^{(i)} .
\end{align*}
We  further split  \small
\begin{align*}
  B_{312}^{(i)}&=- \<    f( Q  _{c })    , ( \varphi _{iAA_1} v_1 )'  \>_{L^2(\Omega_{2,t ,1 })}+ \int _0^1 \<    f'( v_1 +sQ  _{c} ) Q  _{c}   , ( \varphi _{iAA_1} v_1 )'  \>_{L^2(\Omega_{2,t ,1 })}   ds=: B_{3121}^{(i)}+B_{3122}^{(i)}.
\end{align*} \normalsize
The partition of unity $1= \vartheta _{1A_1} + \vartheta _{2A_1} $,
  $Q  _{c} ^{p}\lesssim Q  _{c } ^{\frac{p-1}{2} } |v_1 |^{\frac{p-1}{2}} |v_1|$  in $\Omega_{2,t,1  }$ and     \eqref{eq:h1errorbounds} yield
\begin{align*}
   |B_{3121}^{(i)}|&\lesssim  \int_{\R}Q  _{c } ^{\frac{p-1}{2} }|v_1 |^{\frac{p-1 }{2}+1}   \left |   \varphi _{iAA_1} \zeta_{A}^{-1} \( (\vartheta _{1A_1} + \vartheta _{2A_1})    \zeta_{A}v_1 \) '  +\zeta_{A}v_1 (\varphi _{iAA_1} \zeta_{A}^{-1})'   \right | dx  \\& \lesssim  \| v_1 \| _{L^\infty (\R )}  ^{\frac{p-1 }{2} }\( \left \|   \( \vartheta _{1A_1} \zeta_A v_1 \) '       \right\|_{L^2(\R )} ^2 +  \left \|   \( \vartheta _{2A_1} \zeta_A v_1 \) '       \right\|_{L^2(\R )} ^2 +  \left \| \sech \(     \gamma \kappa  x\)     v_1   \right \|_{L^2(\R )} ^2   \)
\\&\lesssim B ^{\frac{p-1}{4}}\delta ^{\frac{p-1 }{2} }   (  \| \mathbf{v}   \|_{  { \Sigma }_{1A A_1}} ^2+ \| \mathbf{v}   \|_{  { \Sigma }_{2A A_1}}^2   + \|  v_1 \|_{ \widehat{\Sigma}  }^2 ) .
\end{align*}
Similarly, since $1<p\le2$ and $Q_c<2|v_1|$ on
$\Omega_{2,t,1}$, we have
\begin{align*}
Q_c|v_1|^{p-1}
\lesssim Q_c^{\frac{p-1}{2}}|v_1|^{\frac{p+1}{2}}
\quad\text{on }\Omega_{2,t,1}.
\end{align*}
Using the same argument as for $B_{3121}^{(i)}$, we obtain
\begin{align*}&
|B_{3122}^{(i)}|
\lesssim
\int_{\Omega_{2,t,1}}
Q_c|v_1|^{p-1}|(\varphi_{iAA_1}v_1)'|\,dx\\&
 \lesssim
\int_{\R}
Q_c^{\frac{p-1}{2}}|v_1|^{\frac{p+1}{2}}
|(\varphi_{iAA_1}v_1)'|\,dx \lesssim
B^{\frac{p-1}{4}}\delta^{\frac{p-1}{2}}
\left(
\|\mathbf v\|_{\Sigma_{1AA_1}}^2
+\|\mathbf v\|_{\Sigma_{2AA_1}}^2
+\|v_1\|_{\widehat\Sigma}^2
\right).
\end{align*}
 Next we split
\begin{align*}
B_{311}^{(i)} &= \<f'(Q_c)v_1,(\varphi_{iAA_1}v_1)'\>_{L^2(\Omega_{1,t,1})} +\int_0^1 \<[f'(Q_c+sv_1)-f'(Q_c)]v_1, (\varphi_{iAA_1}v_1)'\>_{L^2(\Omega_{1,t,1})}\,ds\\
&\quad -\int_0^1 \<f'(sv_1)v_1, (\varphi_{iAA_1}v_1)'\>_{L^2(\Omega_{1,t,1})}\,ds =:\sum_{1\le k\le3}B_{311k}^{(i)}.
\end{align*}
We have  \small
\begin{align}\nonumber
  |B_{3113}^{(i)}|&\lesssim \<  |v_1|^p   , |( \varphi _{iAA_1} v_1 )'| \>_{L^2(\Omega_{1,t,1  })}
  \\&  \lesssim  \< Q _c ^{\frac{p-1}{2}} |v_1| ^{\frac{p-1}{2}}    |v_1|   , \left |   \varphi _{iAA_1} \zeta_{A}^{-1} \( (\vartheta _{1A_1} + \vartheta _{2A_1})    \zeta_{A}v_1 \) '  +\zeta_{A}v_1 (\varphi _{iAA_1} \zeta_{A}^{-1})'   \right | \>_{L^2(\Omega_{1,t,1  })} \nonumber \\&\lesssim B ^{\frac{p-1}{4}}\delta ^{\frac{p-1 }{2} }   (  \| \mathbf{v}   \|_{  { \Sigma }_{1A A_1}} ^2+ \| \mathbf{v}   \|_{  { \Sigma }_{2A A_1}}^2   +\|  v_1 \|_{ \widehat{\Sigma}  }^2 ) .\label{eq:B2113}
\end{align}
\normalsize
By concavity  $|f'( Q  _{c} + sv _1  ) - f'( Q  _{c}   )|\le |f'(   sv _1 )|$    and as in \eqref{eq:B2113}
\begin{align*}
  |B_{3112}^{(i)}|&\lesssim  B ^{\frac{p-1}{4}}\delta ^{\frac{p-1 }{2} }   (  \| \mathbf{v}   \|_{  { \Sigma }_{1A A_1}} ^2+ \| \mathbf{v}   \|_{  { \Sigma }_{2A A_1}}^2   + \|  v_1 \|_{ \widehat{\Sigma}  }^2 ) .
\end{align*}
With $B_{3111}^{(i)}$  we make a clear distinction between the two cases $i=1,2$. For $i=1$, for a preassigned  fixed small constant $\varepsilon _1>0$, we have the following,  \small
\begin{align}
 \nonumber | B_{3111}^{(1)}|&= \left |\<    f'( Q  _{c}   )    v _1  ,  \varphi _{1AA_1} \zeta_{A}^{-1} \( (\vartheta _{1A_1} + \vartheta _{2A_1})    \zeta_{A}v_1 \) '  +\zeta_{A}v_1 (\varphi _{1AA_1} \zeta_{A}^{-1})' \>_{L^2(\Omega_{1,t,1  })}\right |\\& \lesssim   \varepsilon _1 \| \mathbf{v}    \|_{  { \Sigma }_{1A A_1}}^2  +A_1 ^{-100} \| \mathbf{v }  \|_{  { \Sigma }_{2A A_1}}^2 + \varepsilon _1^{-1} \|  v_1 \|_{ \widehat{\Sigma}  }^2      \nonumber
\end{align}\normalsize
    where we exploited  the exponential decay at infinity of $f'( Q  _{c}   )$ and    \eqref{eq:suppvarphi1}.
 In the $i=2$ case once again  by \eqref{eq:suppvarphi1}  we obtain the following,  \small
\begin{align}
\nonumber  | B_{3111}^{(2)}|&=\left  |\<    f'( Q  _{c}   )    v _1   ,   \varphi _{2AA_1} \zeta_{A}^{-1} \( (\vartheta _{1A_1} + \vartheta _{2A_1})    \zeta_{A}v_1 \) '  +\zeta_{A}v_1 (\varphi _{2AA_1} \zeta_{A}^{-1})'\>_{L^2(\Omega_{1,t,1  } \cap \{ |x|\ge A_1    \}     )}\right | \\& \lesssim  A _1 ^{-100}  (  \|  v_1 \|_{ \widehat{\Sigma}  }^2  +  \| \mathbf{ v}  \|_{  { \Sigma }_{1A A_1}}^2  +\| \mathbf{v}  \|_{  { \Sigma }_{2A A_1}}^2)  .\nonumber
\end{align}\normalsize
Hence we conclude  that
for any preassigned small constant $\varepsilon _1>0$  we have
\begin{align}
 \nonumber  &   | B_3^{(1)}  | \lesssim  \varepsilon _1 \| \mathbf{v}    \|_{  { \Sigma }_{1A A_1}}^2  +A_1 ^{-100} \| \mathbf{v }  \|_{  { \Sigma }_{2A A_1}}^2 + \varepsilon _1^{-1} \|  v_1 \|_{ \widehat{\Sigma}  }^2  \text{  and  }\\&  
\nonumber    | B_3^{(2)}  | \lesssim A _1 ^{-100}  ( \|  v_1 \|_{ \widehat{\Sigma}  }^2  +  \| \mathbf{ v}  \|_{  { \Sigma }_{1A A_1}}^2  +\| \mathbf{v}  \|_{  { \Sigma }_{2A A_1}}^2 )  .
\end{align}

By \eqref{eq:h1errorbounds}, we may assume that
$c(t)\le(1+c_0)/2$.
We choose $\varepsilon_1$ sufficiently small,
$B$ sufficiently large, and then $\delta$ sufficiently small.
In particular, we require
\begin{align*}
B^{\frac{p-1}{4}}\delta^{\frac{p-1}{2}}
\ll\min\{\varepsilon_1,A_1^{-100}\}.
\end{align*}
Since $A_1^{-2}=B^{-6/5}\le B^{-1}$, the preceding estimates, after absorbing the terms involving $\|\mathbf v\|_{\Sigma_{iAA_1}}^2$ into the corresponding
left-hand side, give
\begin{align}
\frac{1-c_0}{8}\|\mathbf v\|_{\Sigma_{1AA_1}}^2
&\le
-\dot{\mathcal I}_1
+C\left(
\|v_1\|_{\widehat\Sigma}^2
+B^{-1}\|\mathbf v\|_{\Sigma_{2AA_1}}^2
\right),
\label{eq:alvir1}\\
\frac{1-c_0}{8}\|\mathbf v\|_{\Sigma_{2AA_1}}^2
&\le
-\dot{\mathcal I}_2
+CA_1^{-2}\left(
\|\mathbf v\|_{\Sigma_{1AA_1}}^2
+\|v_1\|_{\widehat\Sigma}^2
\right),
\label{eq:alvir2}
\end{align}
where $C=C(c_0)>0$ is fixed.  Substituting \eqref{eq:alvir2} in \eqref{eq:alvir1} and bootstrapping, and likewise substituting \eqref{eq:alvir1} in \eqref{eq:alvir2} and bootstrapping,     gives \eqref{eq:lem:1stV11}--\eqref{eq:lem:1stV12}.

\end{proof}

\begin{proof}[Proof of Lemma \ref {prop:1virial}]
    We integrate in $t$ in the   interval $I$ the inequalities in \eqref{eq:lem:1stV11}--\eqref{eq:lem:1stV12} and   use
\begin{align*}&
  |\mathcal{I}_{i} (t)|\le \| \varphi _{iAA_1}\| _{L^\infty (\R ) }\| \mathbf  v(t) \| _{L^2 (\R ) }^2 \lesssim A  B \delta ^2 \text{  for $i=1,2$}.
\end{align*}
By Lemmas \ref{lem:lemdscrt} and \ref{lem:1stV1}, we obtain
\begin{align*}&
   \| \mathbf{v} \| _{L^2(I, { \Sigma }_{1A A_1} )} \lesssim  \sqrt{A B} \delta  +      \| v_1 \| _{L^2(I, \widehat{\Sigma}   )}       \text{ and} \\& \| \mathbf{v} \| _{L^2(I, { \Sigma }_{2A A_1} )} \lesssim  \sqrt{A B} \delta  +    A_1 ^{-1}   \|   v_1    \| _{L^2(I, \widehat{\Sigma}    )}    .
\end{align*}
After fixing $A$ and $B$, choose $\delta$ sufficiently small that
$\delta\le(AB)^{-3/4}$, consistently with \eqref{eq:relABg}. Then
\begin{align*}
 \sqrt{AB}\,\delta
 =\sqrt{AB}\,\delta^{2/3}\delta^{1/3}
 \le\delta^{1/3}.
\end{align*}
Substituting this bound in the preceding two inequalities proves
Lemma \ref{prop:1virial}.
 \end{proof}

\section{Some  spectral theory for $\mathcal{L}_{[c]}$ }\label{sec:prelspec}

We start by  seeking a nonzero  solution $\widetilde{\mathbf{v}}= \mathbf{a} e ^{\mu \xi} $  with $\mathbf{a}=(a_1,a_2)^\intercal \in \C^2$ of system 
\begin{align}
  \label{eq:noeilinv10} \gamma ^2 \Lambda  \widetilde{v}_1 &= c\partial _\xi \widetilde{v} _1  +\partial _\xi \widetilde{v}_2     \text{  and} \\   \label{eq:noeilinv20} \Lambda \widetilde{ v}_2  - c  \Lambda  \widetilde{v}_1&=  \partial _\xi (  - \partial ^2 _\xi +1) \widetilde{v}_1 .
\end{align} This  exists   only in correspondence to roots of the characteristic equation, 
\begin{align}\label{eq:mus1} \mathcal{F}(\Lambda , \mu , \gamma ):=
 \det  \begin{pmatrix}
    \gamma ^2 \Lambda  -c\mu  &   -\mu  \\   \mu  ^3 - \mu -c \Lambda &   \Lambda
    \end{pmatrix} = \gamma ^2\Lambda ^2 -\mu (  2c \Lambda +\mu -  \mu ^3  )=0.
\end{align}
The last equation reduces to $\Lambda =\Lambda _\pm $ with
\begin{align}\label{eq:lambdapm}
  \Lambda _\pm =\Lambda _\pm (\mu ) :   =      \gamma ^{-2}  \mu   \(   c \pm     \sqrt{ 1   - \gamma ^2 \mu ^2} \) .
\end{align}
Setting   $\Lambda _\pm (\im k  )  = \im \tau _\pm (  k  ) $ we obtain
  two odd strictly monotonic   diffeomorphisms
\begin{align}\label{eq:mus221}
 \R \ni k\to \tau _\pm (  k  )=  \gamma ^{-2}  k \left (   c \pm     \sqrt{ 1 + \gamma ^2 k ^2 } \right ) \in   \R .
\end{align}

\begin{lemma}\label{lem:roots1}
\begin{description}
  \item[a]   For any $\Re \Lambda >0$ and any $ 0<c< 1$  equation \eqref{eq:mus1} has only solutions $\mu \not \in \im \R$.  Counted with their multiplicity, there are two  in $\C_+$ and two in $\C _-$, see \ref {itm:Cpm}  Notation \ref{not:notation}. 
  \item[b] For any $\gamma_0\in(0,1)$, there exist $r>0$ and
functions
$\mu_j\in C^\omega(D_\C(0,r)\times(0,\gamma_0),\C)$,
$j=1,2,3,4$, holomorphic in $\Lambda$, such that
$\mathcal F(\Lambda,\mu_j(\Lambda,\gamma),\gamma)\equiv0$
and
 \begin{align}
 \label{lem:roots12}&  \mu _j(\Lambda , \gamma ) =  -((-1)^j1+ c)\Lambda  +O(\Lambda ^2)  \text{ for } j=2,3 \text{ and } \\&  \mu _j(\Lambda , \gamma ) =   (-1) ^j   +O(\Lambda  )  \text{ for } j=1 \text{ and }4.  \label{lem:roots13}
\end{align}

\item[c] For $\Lambda\in\im\R\setminus\{0\}$ and $0<c<1$,
the four roots are distinct and satisfy
\begin{align}\label{eq:mus21}
\Re\mu_1<0=\Re\mu_2=\Re\mu_3<\Re\mu_4.
\end{align}

 \item[d] For any $c_0\in(0,1)$, there exist
$\varepsilon_0>0$ with
$[c_0-\varepsilon_0,c_0+\varepsilon_0]\subset(0,1)$
and a conjugation-invariant neighborhood $\Upsilon$
as in \ref{itm:neigh1} of Notation \ref{not:notation},
such that the maps
\begin{align*}
(\Lambda,c)\longmapsto
\mu_j(\Lambda,\sqrt{1-c^2}),\qquad j=1,2,3,4,
\end{align*}
extend analytically to
$\Upsilon\times(c_0-\varepsilon_0,c_0+\varepsilon_0)$.
These extensions are holomorphic in $\Lambda$ and real
analytic in $c$; at $\Lambda\in\im\R$, they extend across
the imaginary axis.
After shrinking $\Upsilon$, the roots are distinct for
$\Lambda\in\Upsilon\setminus\{0\}$ and
\begin{align*}
\Re\mu_1<0,\quad \Re\mu_2<0,\quad
\Re\mu_3>0,\quad \Re\mu_4>0
   \text{ for }  \Lambda\in\Upsilon\cap\C_+ .
\end{align*}

\item[e]
 For  all $ \Lambda \in \im \R $   we  have \begin{align}  \label{mu1_mu4}
      \mu_1  (-\Lambda ,   \gamma )  & = -  \mu_4  (\Lambda ,   \gamma )  \text{   and }    \\  \label{mu22_mu33}
      \mu_j  (-\Lambda ,   \gamma )  & = -  \mu_j  (\Lambda ,   \gamma )   \text{   for $j=2,3$}   .
\end{align}
\item[f]  For $ \Lambda =\im \Lambda _I\in  \im \R $ and  $|\Lambda _I|\to\infty$ we have
\begin{equation}\label{eq:assroots} \begin{aligned}
 \mu _1&= -\sqrt{\gamma |\Lambda _I|}  +O(1)  , \quad \mu _4=  \sqrt{\gamma |\Lambda _I|}  +O(1) \text{ and } \\  \mu _j&=  (-1)^{j-1} \im     \sqrt{\gamma |\Lambda _I|} \sign \left (  \Lambda _I\right ) +O(1) \text{ for }j=2,3.
\end{aligned}
\end{equation}

\end{description}

\end{lemma}
 \begin{proof}
    Assume by contradiction that  for $\Lambda \in \C_+$ there is a root of the form $
\mu=\im k  $  with $ k\in\mathbb R.$ Substituting in   equation \eqref{eq:mus1}  and using Notation \ref{not:notation}    we obtain
$$
(\im  k)^4-(\im   k)^2-2c(\Lambda _R  +\im\Lambda _I ) \im k+(1-c^2)(\Lambda _R +\im\Lambda _I)^2=0.
$$
Separating real and imaginary parts yields
$$
\begin{cases}
k^4+k^2+2\Lambda _Ic\,k+(1-c^2)(\Lambda _R  ^2-\Lambda _I^2)=0,\\[1mm]
-2\Lambda _R  c\,k+2(1-c^2)\Lambda _R  \Lambda _I=0
\end{cases}
$$
 yielding
$
k=\frac{\gamma ^2}{c}\,\Lambda _I.
$
Substituting this into the first equation gives the following contradiction
$$0=
k^4+k^2+2\Lambda _I c\,k+ \gamma ^2 (\Lambda _R ^2-\Lambda _I^2)
= c^{-4}
 {\gamma ^2}  (\Lambda _R  ^2c^4+\Lambda _I ^2c^2+\Lambda _I^4\gamma ^6 ) >0.
$$
   At $\Lambda=0$, the roots of
$\mathcal F(0,\mu,\gamma)=\mu^2(\mu-1)(\mu+1)$
are $\mu=\pm1$ and the double root $\mu=0$.
The implicit function theorem at $\mu=\pm1$ gives
the analytic roots $\mu_1$ and $\mu_4$.
For the remaining roots, set $\mu=\Lambda\nu$. Then
\begin{align}\label{eq:impftmu23}
\mathcal F(\Lambda,\Lambda\nu,\gamma)
&=-\Lambda^2\mathcal F_1(\Lambda,\nu,\gamma), \text{ where }
\mathcal F_1(\Lambda,\nu,\gamma)
:=\nu^2+2c\nu-\gamma^2-\Lambda^2\nu^4.
\end{align}
At $\Lambda=0$, the roots of $\mathcal F_1=0$ are
$\nu_\pm=-c\pm1$, and
\begin{align*}
\partial_\nu\mathcal F_1(0,-c\pm1,\gamma)=\pm2.
\end{align*}
The implicit function theorem therefore gives analytic
functions $\nu_\pm(\Lambda,\gamma)$ for $\Lambda \in D_\C (0,r)$.
The applications of this theorem can be made uniform for
$\gamma\in[0,\gamma_0]$, giving a common $r>0$.
Setting
$\mu_2=\Lambda\nu_-$ and $\mu_3=\Lambda\nu_+$
proves Claim \textbf{b}.

For each fixed $0<c<1$, these expansions show that, for
sufficiently small positive $\Lambda$,
$\mu_1,\mu_2<0<\mu_3,\mu_4$.
Since no root lies on $\im\R$ when $\Lambda\in\C_+$,
as proved above, the numbers of roots in $\C_+$ and
$\C_-$, counted with multiplicity, are constant on the
connected set $\C_+$.
Thus there are two roots in each half-plane, completing
the proof of Claim \textbf{a}.

Now let $\Lambda=\im\tau$, where $\tau\in\R\setminus\{0\}$.
For real $k$, we have
\begin{align*}
\mathcal F(\im\tau,\im k,\gamma)
&=k^4+k^2+2c\tau k-\gamma^2\tau^2,\\
\frac{d^2}{dk^2}\mathcal F(\im\tau,\im k,\gamma)
&=12k^2+2>0.
\end{align*}
This strictly convex polynomial is negative at $k=0$
and tends to $+\infty$ as $k\to\pm\infty$.
It therefore has exactly two real roots, both simple.
By \eqref{eq:mus221}, they are
$\tau_-^{-1}(\tau)$ and $\tau_+^{-1}(\tau)$.
Thus precisely two roots of \eqref{eq:mus1} are imaginary.
The remaining two roots are off the imaginary axis and,
by
\begin{align}\label{eq:mus22}
\mu_1+\mu_2+\mu_3+\mu_4=0,
\end{align}
have opposite nonzero real parts.
In particular, they are distinct, proving Claim \textbf{c}.

Away from $\Lambda=0$, the analytic implicit function
theorem applies to all four roots on the imaginary axis.
Near $\Lambda=0$, we use the analytic branches in
Claim \textbf{b}.
Gluing these branches gives the extensions in
Claim \textbf{d}.
By restricting $c$ to a compact interval about $c_0$
contained in $(0,1)$, the neighborhoods can be chosen
in common for these values of $c$.
Their width may depend on the point of the imaginary axis.
We may also choose $\Upsilon$ invariant under complex conjugation.
For the two imaginary roots, we have 
\begin{align*}
\partial_\Lambda\mu_2(\im\tau,\gamma)
&=\frac{1}{\tau_-'(\tau_-^{-1}(\tau))}<0 \text{ and }
\partial_\Lambda\mu_3(\im\tau,\gamma)
=\frac{1}{\tau_+'(\tau_+^{-1}(\tau))}>0.
\end{align*}
These inequalities and continuity of the other two roots
give the sign conditions in Claim \textbf{d} after
shrinking $\Upsilon$.
Claim  \textbf{e}   follows from the identity
$  \mathcal{F}(-\Lambda , -\mu , \gamma ) = \mathcal{F}(\Lambda , \mu , \gamma )$.   

To prove Claim \textbf{f}, set $s=|\Lambda_I|$ and
$\mu=\sqrt{\gamma s}\,\nu$. Dividing
$\mathcal F(\im\Lambda_I,\mu,\gamma)=0$ by $\gamma^2s^2$
gives
\begin{align*}
\psi (\nu ):= \nu^4-1-\frac{\nu^2}{\gamma s}
 -\frac{2\im c\,\operatorname{sign}(\Lambda_I)}
        {\gamma^{3/2}\sqrt{s}}\,\nu=0.
\end{align*}
Fix $c$, set $\varphi (\nu )=\nu ^4 -1$ and choose
\begin{align*}
 K>\frac{25}{16\gamma}+\frac{5c}{2\gamma^{3/2}}.
\end{align*}
For $s$ sufficiently large that $s\ge1$ and $K/\sqrt{s}\le1/4$,
consider the four circles
$|\nu-\zeta|=K/\sqrt{s}$, where
$\zeta\in\{1,-1,\im,-\im\}$.
On each circle, $|\nu|\le5/4$, and hence
\begin{align*}
|\psi (\nu ) -\varphi (\nu ) | =      \left|\frac{\nu^2}{\gamma s}
 +\frac{2\im c\,\operatorname{sign}(\Lambda_I)}
        {\gamma^{3/2}\sqrt{s}}\,\nu\right|
 \le\left(\frac{25}{16\gamma}
          +\frac{5c}{2\gamma^{3/2}}\right)s^{-1/2}
 <Ks^{-1/2}.
\end{align*}
For every other fourth root of unity $\zeta'\ne\zeta$,
$|\nu-\zeta'|\ge\sqrt2-1/4>1$  on the circle  $|\nu-\zeta|=K/\sqrt{s}$,     so
\begin{align*}
| \varphi (\nu ) | = |\nu^4-1|
 =|\nu-\zeta|
   \prod_{\substack{(\zeta')^4=1\\\zeta'\ne\zeta}}
      |\nu-\zeta'|
 \ge Ks^{-1/2}.
\end{align*}
Rouch\'e's theorem, see \cite{MR924157}, therefore gives exactly
one root of $\psi (\nu )$ in each of the four disjoint disks. Each such root
satisfies
\begin{align*}
 \nu=\zeta+O(s^{-1/2})  \text{  and }
 \mu=\sqrt{\gamma s}\,\zeta+O(1).
\end{align*}
Claim \textbf{c} identifies the roots near $-1$ and $1$
with $\mu_1$ and $\mu_4$. By \eqref{eq:mus221}, the imaginary
part of $\mu_2$ has sign opposite to $\Lambda_I$, whereas
that of $\mu_3$ has the same sign. Thus
\begin{align*}
 \mu_1&=-\sqrt{\gamma|\Lambda_I|}+O(1),&
 \mu_4&=\sqrt{\gamma|\Lambda_I|}+O(1),\\
 \mu_2&=-\im\sqrt{\gamma|\Lambda_I|}
            \operatorname{sign}(\Lambda_I)+O(1),&
 \mu_3&=\im\sqrt{\gamma|\Lambda_I|}
            \operatorname{sign}(\Lambda_I)+O(1),
\end{align*}
which proves \eqref{eq:assroots}.

For later use notice that  using the invertible functions $\tau _\pm :\R \to \R$ in \eqref{eq:mus221} we have
\begin{equation}\label{eq:mus3pm}
  \mu _2(\im \Lambda _I, \gamma ) =\im  \tau _- ^{-1}(\Lambda _I)  \text{ and }  \mu _3(\im \Lambda _I, \gamma ) =\im  \tau _+ ^{-1}(\Lambda _I) .
\end{equation}
  
 \end{proof}

  Besides system \eqref{A_Decompose}   and its homogeneous version
\begin{align}  \label{A_Decomposehom}&
\left ( d   /d\xi  -  ( A _\infty^{(0)}  + \Lambda A _\infty^{(1)} (\gamma)  +q(\xi))  \right ) \widetilde{V}   =  0,
\end{align} we  will be also interested to the adjoint   problem
\begin{align}  \label{A_djoint}&
 \left ( d   /d\xi +  ( A _\infty^{(0)}  + \Lambda A _\infty^{(1)} (\gamma)  +q(\xi))  ^\intercal \right )    W     =0 ,
\end{align}
which in the components $W=\left ( w_1, w_2, w_3, w_4    \right )^\intercal $ is
\begin{equation} \label{A_djointcoo1} \begin{aligned}
 w_1' &=-c \Lambda w_3-\gamma ^2 \Lambda w_4+p (p-1)Q ^{p-2} Q'w_3\\   w_2' &=-  w_1-w_3+c w_4+p  Q ^{p-1}  w_3\\ w_3' &=-  w_2 \text{ and }w_4'=\Lambda w_3.
\end{aligned}
\end{equation}
For $\Lambda\in\Upsilon$, the system associated
asymptotically to \eqref{A_Decomposehom} as
$|\xi|\to\infty$ is
\begin{align}\label{A_infinite}
\frac{d}{d\xi}\widetilde V
=(A_\infty^{(0)}+\Lambda A_\infty^{(1)}(\gamma))
\widetilde V.
\end{align}
Its space of solutions is generated by
\begin{align}\label{eq:vj0}
\widetilde V_j^{(0)}
=e^{\mu_j\xi}\mathbf a_j,\qquad
\mathbf a_j
=\left(1,\mu_j,\mu_j^2,
\frac{\gamma^2\Lambda}{\mu_j}-c\right)^\intercal,
\quad j=1,2,3,4.
\end{align}
The quotients at $\Lambda=0$ are understood by analytic
continuation. In particular,
\begin{align*}
\left.\frac{\Lambda}{\mu_2}\right|_{\Lambda=0}
=-\frac{1}{1+c},
\qquad
\left.\frac{\Lambda}{\mu_3}\right|_{\Lambda=0}
=\frac{1}{1-c}.
\end{align*}
Hence
$\mathbf a_2(0,\gamma)=(1,0,0,-1)^\intercal$ and
$\mathbf a_3(0,\gamma)=(1,0,0,1)^\intercal$,
and the four vectors remain a basis at zero.
The same convention is used in \eqref{eq:bj} below.
The denominators there have limits
$-2/(1+c)$ and $-2/(1-c)$ for $j=2,3$, respectively,
and $2$ for $j=1,4$.
Thus the dual vectors $\mathbf b_j$, defined in \eqref{eq:bj} below, also extend
analytically to zero.
We will also set $\mathbf{a}_j  =:  \(a _{j1}, a _{j2}, a _{j3} ,  a _{j4}  \)   ^\intercal $.     System
\begin{align}  \label{A_djointinf}&
 d  W /d\xi =- ( A _\infty^{(0)}  + \Lambda A _\infty^{(1)} (\gamma)    )  ^\intercal W,
\end{align}
or equivalently
\begin{equation} \label{A_djointinf1} \begin{aligned}
 w_1' &=-c \Lambda w_3-\gamma ^2 \Lambda w_4\\   w_2' &=-  w_1-w_3+c w_4\\ w_3' &=-  w_2 \text{ and }w_4'=\Lambda w_3,
\end{aligned}
\end{equation}
has space of solutions generated by
\begin{align}\label{eq:zj0}
   Z _j^{(0)}= e ^{-\mu _j \xi} \mathbf{b}_j \text{ with }  \< \mathbf{b}_j  , \mathbf{a} _k \> _{\C ^4}=\delta _{jk} ,
\end{align}
where   we will write $\mathbf{b}_j  =:  \(b _{j1}, b _{j2}, b _{j3} ,  b _{j4}  \)   ^\intercal  $ with
\begin{align} \nonumber
   &\mathbf{b}_j  =  \frac{\widetilde{\mathbf{b}}_j}{ \< \widetilde{\mathbf{b}}_j  , \mathbf{a} _j \> _{\C ^4}} \text{ with  } \widetilde{\mathbf{b}}_j = \(     \frac{\Lambda}{\mu _j} \( c- \gamma ^2  \frac{\Lambda}{\mu _j} \)  , \mu _j,1, -  \frac{\Lambda}{\mu _j}  \) ^\intercal   \text{ and}  \\& \< \widetilde{\mathbf{b}}_j  , \mathbf{a} _j \> _{\C ^4}= 2  \frac{\Lambda}{\mu _j} \( c- \gamma ^2  \frac{\Lambda}{\mu _j} \) +2\mu_j^2. \label{eq:bj}
\end{align}
By Lemma \ref{lem:transys} for $q(\xi )\equiv 0$   equation $ (\mathcal{L} _{[c]}^{(0)}-\lambda ) (v_1, v_2)^\intercal =  (F_1, F_2)^\intercal $ writes equivalently 
  as
\begin{align}  \label{A_Decomposenhcc}&
  d   \widetilde{V} /d\xi  =( A _\infty^{(0)}  + \Lambda A _\infty^{(1)} (\gamma)  )\widetilde{ V} +\widetilde{\mathbf{F}}  
\end{align}
Here $\widetilde V$ and $\widetilde{\mathbf F}$ are as in
\eqref{A_Decompose} and \eqref{eq:relkern2}, respectively.
For $\Lambda\in\Upsilon\cap\C_+$, the unique $L^2$ solution is \small
\begin{align}\label{eq:A_Decomposenhcc1}
\widetilde V(\xi)
=L(\Lambda)\widetilde{\mathbf F}(\xi)
&:=
\sum_{j=1}^2\mathbf a_j
\int_{-\infty}^{\xi}
e^{\mu_j(\xi-y)}
\<\widetilde{\mathbf F}(y),\mathbf b_j\>_{\C^4}\,dy
-\sum_{j=3}^4\mathbf a_j
\int_{\xi}^{+\infty}
e^{\mu_j(\xi-y)}
\<\widetilde{\mathbf F}(y),\mathbf b_j\>_{\C^4}\,dy.
\end{align}\normalsize
The sign conditions in Lemma \ref{lem:roots1} and
Young's inequality show that, for each fixed
$\Lambda\in\Upsilon\cap\C_+$,
$L(\Lambda)$ is bounded on $L^2(\R,\C^4)$.
For $\Lambda\in\im\R$ and
$\widetilde{\mathbf F}\in L^1(\R,\C^4)$,
the same formula yields  the pointwise $L(z)\widetilde{\mathbf F}(\xi)   \xrightarrow{z\to \Lambda } L(\Lambda)\widetilde{\mathbf F}(\xi) $       limit  with $z\in \C_+$
and defines the outgoing boundary value.
In particular we have
\small
\begin{align} \label{eq:tildev1v4}
  \( \widetilde{v}_1 , \widetilde{v}_4\)  ^\intercal =  \(   ( c \gamma^{-3}  L _{13}   +\gamma^{-1}  L _{14}) \widetilde{ F}_1  -     \gamma^{-3}  L _{13} \widetilde{ F}_2     , 
 ( c \gamma^{-3}  L _{43}   +\gamma^{-1}  L _{44}) \widetilde{ F}_1  -     \gamma^{-3}  L _{43} \widetilde{ F}_2 \)  ^\intercal
\end{align} \normalsize
where, for $\mathbf{c}_l  =   \(c _{l1}, c _{l2}, c _{l3} ,  c _{l4}  \)   ^\intercal $   for $\mathbf{c}_l=\mathbf{a}_l,\mathbf{b}_l$, 
 \small
\begin{align}\label{deflij}&
  L_{ij}g  = K _{ij}( \cdot ,  \Lambda , \gamma ) *g  \text{ with } \\&  K _{ij}(\xi , \Lambda , \gamma )=\sum_{\ell =1,2} a_{\ell i}(  \Lambda , \gamma )b_{\ell j}(  \Lambda , \gamma)  e^{\mu _\ell (  \Lambda , \gamma ) \xi} 1_{\R _+}(\xi )  - \sum_{\ell =3,4} a_{\ell i}(  \Lambda , \gamma ) b_{\ell j}(  \Lambda , \gamma ) e^{\mu _\ell (  \Lambda , \gamma ) \xi} 1_{\R _-}(\xi ). \nonumber
\end{align}
\normalsize
Since $\{\mathbf b_j\}_{j=1}^4$ is the dual basis of
$\{\mathbf a_j\}_{j=1}^4$ with respect to the bilinear pairing, $\sum_{j=1}^4\mathbf a_j\mathbf b_j^\intercal=I_4.$
Taking the $(\ell,3)$ entry gives
\begin{align}\label{eq:varpid}
\sum_{j=1}^4a_{j\ell}b_{j3}=\delta_{\ell3},
\qquad \ell=1,\dots,4.
\end{align}
When differentiating $K_{i3}$ in \eqref{deflij}
distributionally, the coefficient of the Dirac mass at
zero is therefore $\delta_{i3}$.
Consequently, for $i\ne3$ and initially for
$\widetilde F_2\in C_c^\infty(\R)$, we obtain
\begin{align} \label{eq:varpid1}
      L _{i 3} \widetilde{ F}_2 '=  \widetilde{ L} _{i 3} \widetilde{ F}_2: =  (  \sum_{\ell =1,2} a_{\ell i} b_{\ell 3}  \mu _\ell  e^{\mu _\ell   \xi} 1_{\R _+}   - \sum_{\ell =3,4} a_{\ell i}      b_{\ell 3}  \mu _\ell e^{\mu _\ell  \xi} 1_{\R _-}   )  * \widetilde{ F}_2 \text{ for  }i\neq 3.
\end{align}

\section{Dispersion and Kato smoothing  for the linearized equation }\label{sec:disp}

For $\mathbf f\in L^1(\R,\C^2)$, we define
\begin{align}\label{eq:projections}
\mathcal P_{[c]}\mathbf f
&:=
\sum_{i=1}^2\xi_i[c]\<\eta_i[c],\mathbf f\>,
\qquad
\mathcal Q_{[c]}\mathbf f
:=\mathbf f-\mathcal P_{[c]}\mathbf f.
\end{align}
These inner products are    absolutely convergent integrals because
$\eta_i[c]\in L^\infty(\R,\C^2 )$.
By \eqref{eq:M2.5},
$\mathcal P_{[c]}^2=\mathcal P_{[c]}$ and
$\mathcal Q_{[c]}^2=\mathcal Q_{[c]}$.
The exponential decay of $\xi_i[c]$ implies that these
are bounded projections on $L^{1,1}(\R,\C^2)$ and on
$\mathcal H^1\cap L^{1,1}(\R,\C^2)$.
For $g_2\in L^1(\R)$, the inner products  with derivatives
are defined by integration by parts:
\begin{align*}
\<\eta_i[c],(0,\partial_xg_2)^\intercal\>
:=
-\<\eta_i'[c],(0,g_2)^\intercal\>,
\qquad i=1,2.
\end{align*}
We consider the projection $\pi_1$ in \ref{itm:proj} of Notation \ref{not:notation}.
We first justify the Laplace transforms used below.

\begin{lemma}\label{lem:linear-group-growth}
For $c\in((p-1)^{1/2}/2,1)$, the operator  
$\mathcal L_{[c]}$ on $\mathcal H^1$, with domain
$H^3(\R)\times H^2(\R)$, generates a strongly continuous group and
\begin{align}\label{eq:linear-group-growth}
\|e^{t\mathcal L_{[c]}}\mathbf v_0\|_{\mathcal H^1}
\le C(c)(1+|t|)\|\mathbf v_0\|_{\mathcal H^1},\qquad t \in\mathbb{R}, \quad \text{for all }  \mathbf v_0 \in \mathcal H^1.
\end{align}
Consequently, the  $R_{\mathcal L_{[c]}}(z)=(\mathcal L_{[c]}-z)^{-1}$ is defined by
\begin{align}\label{eq:linear-laplace}
R_{\mathcal L_{[c]}}(z)\mathbf v_0
=-\int_0^\infty e^{-zt}e^{t\mathcal L_{[c]}}\mathbf v_0\,dt \quad
 \text{for } \Re z>0 \text{ and } \mathbf v_0 \in \mathcal H^1,
\end{align}
where the integral converges in $\mathcal H^1$.
\end{lemma}
\begin{proof}
The free operator $\mathcal L_{[c]}^{(0)}$ generates a strongly continuous group $e^{t\mathcal L_{[c]}^{(0)}}$ on $\mathcal H^1$.  Since $f'(Q_c)\in W^{1,\infty}(\R)$, we have
\begin{align*}
\|(\mathcal L_{[c]}-\mathcal L_{[c]}^{(0)})\mathbf v\|_{\mathcal H^1}
=\|\partial_x(f'(Q_c)v_1)\|_{L^2}
\le C(c)\|\mathbf v\|_{\mathcal H^1}.
\end{align*}
By the bounded perturbation theorem, $\mathcal L_{[c]}$ generates a strongly continuous  group $e^{t\mathcal L_{[c]} }$ on $\mathcal H^1$.

From Lemma 2.1 of \cite{maulen2025asymptoticstabilitystablegood}, there exists $\delta(c)>0$ such that 
\begin{align}\label{eq:linear-constrained-coercivity}
\langle\mathbf L_c\mathbf z, \mathbf z \rangle
\ge\delta(c)\|\mathbf z\|_{\mathcal H^1}^2
\quad\text{if}\quad
\langle\mathbf z,\sigma_1\mathbf Q_c\rangle
=\langle\mathbf z,\xi_1[c]\rangle=0.
\end{align}
For $\mathbf v_0 \in H^3(\R)\times H^2(\R)$,  $\mathbf v = e^{t\mathcal L_{[c]}}\mathbf v_0 \in H^3(\R)\times H^2(\R)$  is the classical solution to $\partial_t \mathbf v = \mathcal L_{[c]}\mathbf v$. Set 
\begin{align}\label{eq:linear-growth-decomposition}
b=\langle\eta_2[c],\mathbf v_0\rangle,
\quad \mathbf w(t)&=e^{t\mathcal L_{[c]}}\mathbf v_0-b\xi_2[c],\quad
a(t)=\frac{\langle\mathbf w(t),\xi_1[c]\rangle}{\|\xi_1[c]\|_{L^2}^2},\quad
\mathbf z(t)=\mathbf w(t)-a(t)\xi_1[c].
\end{align}
Then, $\langle\mathbf z,\xi_1[c]\rangle\equiv 0$. Since $\mathbf L_c$ is self-adjoint and $\mathbf L_c\xi_1 = 0$, we have 
\[
\begin{split}
\langle \mathbf L_c \mathbf{z}, \mathbf{z} \rangle 
& =   \langle \mathbf L_c  \mathbf{w}, \mathbf{w} \rangle + \langle \mathbf L_c  \mathbf{w}, -a\xi_1 \rangle + \langle -a \mathbf L_c  \xi_1, \mathbf{w} \rangle + \langle -a \mathbf L_c  \xi_1,-a \xi_1 \rangle   =   \langle \mathbf L_c  \mathbf{w}, \mathbf{w} \rangle.
\end{split}
\] 
By \eqref{eq:M2.3}, we have
$\partial_t\mathbf w=\mathcal L_{[c]}\mathbf w-b\xi_1[c]$. Since $\mathbf L_c\xi_1[c]=0$ and $\partial_x\sigma_1$ is skew-adjoint,
\[
\partial_t\langle \mathbf L_c \mathbf w, \mathbf w\rangle = 2\Re \langle \partial_t \mathbf w, \mathbf L_c \mathbf w \rangle = 2\Re \langle \mathcal L_{[c]}\mathbf w-b\xi_1[c], \mathbf L_c \mathbf w \rangle = 2\Re \langle \partial_x \sigma_1\mathbf L_c\mathbf w, \mathbf L_c \mathbf w \rangle = 0,
\]
and hence, $\langle \mathbf L_c \mathbf{w}(t), \mathbf{w}(t) \rangle$ is conserved. Combining all, we see that
\[
\begin{split}
\|\mathbf z\|_{\mathcal H^1}^2 
& \leq \langle \mathbf L_c \mathbf{z}, \mathbf{z} \rangle = \langle \mathbf L_c \mathbf{w}, \mathbf{w} \rangle = \langle \mathbf L_c \mathbf{w}(0), \mathbf{w}(0) \rangle =   \langle \mathbf L_c \mathbf v_0, \mathbf v_0 \rangle - 2\Re \langle \mathbf L_c \mathbf v_0, b\xi_2 \rangle + |b|^2 \langle \mathbf L_c \xi_2, \xi_2 \rangle \\
& \leq C(c) \|\mathbf v_0\|_{\mathcal H^1}^2.
\end{split}
\]
Hence, 
\begin{align}\label{eq:linear-growth-transverse}
\sup_{t\in\R}\|\mathbf z(t)\|_{\mathcal H^1}
\le C(c)\|\mathbf v_0\|_{\mathcal H^1}.
\end{align}
On the other hand, we have
\begin{align*}
\dot a(t)=\frac{\langle\mathbf z(t),\mathcal L_{[c]}^*\xi_1[c]\rangle}
{\|\xi_1[c]\|_{L^2}^2}-b.
\end{align*}
Integrating over $[0,t]$ and  using \eqref{eq:linear-growth-transverse} give $|a(t)|\le C(c)(1+|t|)\|\mathbf v_0\|_{\mathcal H^1}$. With \eqref{eq:linear-growth-decomposition}, this proves
\eqref{eq:linear-group-growth}.
For $\mathbf v_0 \in \mathcal H^1$, the above calculations can be justified by approximation by smooth data in
$\mathcal H^1$. 

Finally,
\eqref{eq:linear-group-growth} makes the integral in
\eqref{eq:linear-laplace} convergent for every $\Re z>0$. Hence, $\Re z>0$ is included in the resolvent set of $\mathcal L_{[c]}$ and, in particular,   we have
\[
(\mathcal L_{[c]}-z)^{-1}\mathbf v_0
=-\int_0^\infty e^{-zt}e^{t\mathcal L_{[c]}}\mathbf v_0\,dt
\]
for all $\mathbf v_0 \in \mathcal H^1$ (see Theorem 1.10 in Chapter II of \cite{EN2000}).

\end{proof}

We prove the following proposition later in this section.
\begin{proposition}
\label{thm:pwdisp}
For every $c\in((p-1)^{1/2}/2,1)$ there exists $C(c)>0$ such that
\begin{align}\label{eq:pwdisp1}
\|\pi_1e^{t\mathcal L_{[c]}}\mathcal Q_{[c]}\mathbf v_0\|_
{L^2((0,+\infty),\widehat\Sigma)}
\le C(c)\|\mathbf v_0\|_{\mathcal H^1\cap L^{1,1}(\R,\C^2)}.
\end{align}
\end{proposition}

\begin{proposition}\label{prop:smooth}
For every $c\in((p-1)^{1/2}/2,1)$ there exists $C(c)>0$ such that
\begin{align}\label{eq:smooth1}
\left\|\pi_1\int_0^t e^{(t-s)\mathcal L_{[c]}}
\mathcal Q_{[c]}f(s)\,ds\right\|_{L^2(\R_+,\widehat\Sigma)}
\le C(c)\|f\|_{L^{1,1}(\R,L^2(\R_+,\C^2))}.
\end{align}
\end{proposition}
\begin{proof}
  From  \cite[Lemma 3.5]{katoMA1966}   for any $\varepsilon >0$ and \eqref{eq:pwdisp1}
\begin{align}\label{eq:smooth--}
  \int _{\R} \|  \pi _1 R_{\mathcal{L} _{[c]}} (  \im \lambda + \varepsilon  ) \mathcal Q_{[c]}  \mathbf{w}_0 \| _{ \widehat{\Sigma }}^2 \,d\lambda=2\pi \int _{\R _+} \| \pi _1 e^{  t  \mathcal{L}_{[c]} }\mathcal Q_{[c]}     e^{-\varepsilon t} \mathbf{w}_0 \| _{ \widehat{\Sigma }}^2 dt  \le C(c)    \|    \mathbf{w}_0 \| _{ \mathcal{H}^1 \bigcap  L^{1,1}  (\R ) } ^2.
\end{align}
 Then $F(z ):=\pi _1 R_{\mathcal{L} _{[c]}} ( z  ) \mathcal Q_{[c]}  \mathbf{w}_0$ is in the Hardy class $H^2(\C _+, \widehat{\Sigma} )$  with     nontangential limit
$  F_*(\im \lambda   ):=\pi _1 R_{\mathcal{L} _{[c]}} ( \im \lambda +0  ) \mathcal Q_{[c]}  \mathbf{w}_0\in L^2( \R_\lambda , \widehat{\Sigma} )$, see \cite[Remark 1.4]{katoMA1966},  where by standard arguments
\begin{align}\label{eq:smooth--+}
   \|  F_* \| _{L^2( \R_\lambda , \widehat{\Sigma} )}=\sup _{\varepsilon >0} \|  F (\im \lambda +\varepsilon ) \| _{L^2( \R_\lambda , \widehat{\Sigma} )}.
\end{align}
For, 
    $f\in C_c^{ \infty}([0, +\infty )\times \R , \C ^2) $, with $\mathcal{ F} \left [  g \right ] =\widehat{g}$ the Fourier transform in the time variable,   \small
 \begin{align*} &
  \pi _1 \int _\R  dt e^{-(\im \lambda  + \varepsilon ) t} \int _0 ^t    e^{  (t-s) \mathcal{L} _{[c]} }\mathcal Q  _{[c]} f(s) ds =  \pi _1 \int _{0}^{+\infty} ds  e^{- \im \lambda   s}   \int _{0}^{+\infty}  e^{-(\im \lambda  + \varepsilon ) t}   e^{  t\mathcal{L} _{[c]} }   \mathcal    Q  _{[c]}  e^{-  \varepsilon   s} f(s) dt\\& =-  \pi _1\int _{0}^{+\infty}    e^{-\im \lambda s}  R_{\mathcal{L} _{[c]} }(\im \lambda  + \varepsilon )\mathcal Q  _{[c]} f(s) ds =- \pi _1 R_{\mathcal{L} _{[c]} }     ( \im \lambda  +\varepsilon  ) \mathcal Q  _{[c]}  \mathcal{ F} \left [  e^{-  \varepsilon   s}   f \right ](\lambda) \text{ for any }\varepsilon >0,
 \end{align*} \normalsize
  with the following,  where the first inequality follows from \eqref{eq:smooth--+},
\begin{equation}\nonumber
     \begin{aligned}
        &
      \left \| \pi _1 \int _0 ^t    e^{  (t-s) \mathcal{L}_{[c]} }\mathcal Q_{[c]}   f(s) ds  \right \| _{L^2 \( \R _+ ,    \widehat{\Sigma}\) } \le  \liminf _{\varepsilon\to 0^+} \left \|  \pi _1 e^{-  \varepsilon   t} \int _0 ^t    e^{  (t-s) \mathcal{L}_{[c]} }\mathcal Q_{[c]}   f(s) ds  \right \| _{L^2 \( \R _+ ,    \widehat{\Sigma}\) }
\\& =(2\pi ) ^{-1} \liminf _{\varepsilon\to 0^+}
 \left \|   \pi _1  R_{\mathcal{L}  _{[c]}}  (\im \lambda   +\varepsilon )   \mathcal Q_{[c]}   \mathcal{ F} \left [  e^{-  \varepsilon   s}   f \right ] (\lambda)   \right \| _{L^2 \( \R  ,    \widehat{\Sigma}\) }
 \\& \le (2\pi ) ^{-1} \liminf _{\varepsilon\to 0^+}
 \left \|   \pi _1  R_{\mathcal{L}  _{[c]}}  (\im \lambda   +0)   \mathcal Q_{[c]}   \mathcal{ F} \left [  e^{-  \varepsilon   s}   f \right ] (\lambda)   \right \| _{L^2 \( \R  ,    \widehat{\Sigma}\) }
 \\&\le   (2\pi ) ^{-1}  \sup _{\substack { \lambda \in \R }} \left \|  \pi _1    R_{\mathcal{L}_{[c]} }  (\im \lambda + 0  )  \mathcal  Q_{[c]}   \right \| _{    \mathcal{L} \(  L^{1,1}  \( \R\),   \widehat{\Sigma} \)    } \sup   _{\varepsilon>0}  \left \|       \mathcal{ F} \left [  e^{-  \varepsilon   s}   f \right ] (\lambda)    \right \| _{L^2 \( \R   ,     L^{1,1}  \( \R\)\) }\\&\le   (2\pi ) ^{-1}  \sup _{\substack { \lambda \in \R }} \left \| \pi _1    R_{\mathcal{L}_{[c]} }  (\im \lambda + 0  )  \mathcal  Q_{[c]}   \right \| _{    \mathcal{L} \( L^{ 1, 1}\( \R\),   \widehat{\Sigma} \)    }   \|       f \| _{  L^{1,1}  \( \R   ,    L^{ 2}\( \R _+\)\) }  .
     \end{aligned}
  \end{equation}
Then  \eqref{eq:smooth1}  follows from    {Proposition} \ref{prop:smooth2} below. 
\end{proof}

\begin{proposition}\label{prop:smooth2}
For every $c\in((p-1)^{1/2}/2,1)$ there exist   $C(c)>0$   and  $\varepsilon_0(c)>0$  
\begin{align}\label{eq:smooth21strip}
\sup_{\substack{0\le \varepsilon<\varepsilon_0(c)\\\lambda\in \R\setminus\{0\}}}
\|\pi_1R_{\mathcal L_{[c]}}(\varepsilon+\im\lambda)
\mathcal Q_{[c]}\|_
{\mathcal L(L^{1,1}(\R,\C^2),\widehat\Sigma)}\le C(c).
\end{align}
\end{proposition}
The proof of the estimates in {Proposition} \ref{prop:smooth2} is distinguished between near zero, considered in  \S \ref{sec:proofsmooth2} and away
from zero, which we consider now. 

\begin{lemma}\label{lem:largterm}
For every $c\in((p-1)^{1/2}/2,1)$ and sufficiently large
$\mathbf c_1>0$,
\begin{align}\label{eq:largterm1}
\sup_{\substack {|\lambda|\ge\gamma^3\mathbf c_1\\ 0\le \varepsilon\le\gamma^3}}
\|\pi_1R_{\mathcal L_{[c]}}(\im\lambda+\varepsilon)\|_
{\mathcal L(L^1(\R,\C^2),\widehat\Sigma)}\le C(c).
\end{align}
\end{lemma}
\begin{proof} For $z=\im\lambda+\varepsilon$
write $z=\gamma^3\Lambda$, $\xi=\gamma x$,
$F(x)=\widetilde F(\xi)$ and
$v_1(x)=\widetilde v_1(\xi)$, $v_2(x)=\widetilde v_4(\xi)$.
By \eqref{A_Decompose} and \eqref{deflij}, the outgoing solution satisfies
\begin{align}\label{eq:smooth12}
\widetilde v_i
=L_{i3}(q_{31}\widetilde v_1+q_{32}\widetilde v_2)
+\gamma^{-3}L_{i3}(c\widetilde F_1-\widetilde F_2)
+\gamma^{-1}L_{i4}\widetilde F_1.
\end{align}
The identities $q_{31}=q_{32}'$ and
$\widetilde v_2=\widetilde v_1'$ in \eqref{A_Decompose} give
$q_{31}\widetilde v_1+q_{32}\widetilde v_2
=(q_{32}\widetilde v_1)'$.
By \eqref{eq:varpid1} and $a_{j2}=\mu_j a_{j1}$,
$L_{13}\partial_\xi=L_{23}$; there is no Dirac term because
$\sum_j a_{j1}b_{j3}=0$ in \eqref{eq:varpid}.
Thus \eqref{eq:smooth12} for $i=1$ becomes
\begin{align}\label{eq:high-scalar-resolvent}
\widetilde v_1
=L_{23}(q_{32}\widetilde v_1)
+\gamma^{-3}L_{13}(c\widetilde F_1-\widetilde F_2)
+\gamma^{-1}L_{14}\widetilde F_1.
\end{align}
By \eqref{eq:assroots}, $|\mu_j|\asymp_c|\Lambda|^{1/2}$ at
high energy. The denominator in \eqref{eq:bj} also equals
$$\partial_\mu\mathcal F(\Lambda,\mu_j,\gamma)/\mu_j
=4\mu_j^2-2-2c\Lambda/\mu_j$$ and has modulus comparable to
$|\Lambda|$. These estimates hold uniformly for
$0\le\Re\Lambda\le1$ and sufficiently large $|\Im\Lambda|$. 
The exponentials in the   kernels \eqref{deflij} have
modulus at most one on their integration half-lines.
Since $a_{j1}=1$, $a_{j2}=\mu_j$, $b_{j3}$ is the reciprocal of
that denominator and $b_{j4}=-(\Lambda/\mu_j)b_{j3}$,
\begin{align}\label{eq:high-kernel-bounds}
\|\sech(\kappa\xi)L_{13}h\|_{L^2_\xi}
&\le C(c)\langle\Lambda\rangle^{-1}\|h\|_{L^1_\xi},\\
\|\sech(\kappa\xi)L_{14}h\|_{L^2_\xi}
+\|\sech(\kappa\xi)L_{23}h\|_{L^2_\xi}
&\le C(c)\langle\Lambda\rangle^{-1/2}\|h\|_{L^1_\xi}.
\nonumber
\end{align}
Choose $\kappa>0$ sufficiently small that
$q_{32}\cosh(\kappa\xi)\in L^2_\xi$.
Cauchy--Schwarz and \eqref{eq:high-kernel-bounds} yield
\begin{align}\label{eq:high-scalar-contraction}
\|\sech(\kappa\xi)L_{23}(q_{32}h)\|_{L^2_\xi}
\le C(c)\langle\Lambda\rangle^{-1/2}
\|\sech(\kappa\xi)h\|_{L^2_\xi}.
\end{align}
For sufficiently large $\mathbf c_1$, the factor in
\eqref{eq:high-scalar-contraction} is at most $1/2$.
Hence 
\begin{align}\label{eq:largterm2}
\|\sech(\kappa\xi)\widetilde v_1\|_{L^2_\xi}
\le C(c)\bigl(
\langle\Lambda\rangle^{-1/2}\|\widetilde F_1\|_{L^1_\xi}
+\langle\Lambda\rangle^{-1}\|\widetilde F_2\|_{L^1_\xi}\bigr).
\end{align}
Finally,
\begin{align}\label{eq:resolvent-weight-scaling}
\|v_1\|_{\widehat\Sigma}
=\gamma^{-1/2}\|\sech(\kappa\xi)\widetilde v_1\|_{L^2_\xi},
\qquad
\|\widetilde F\|_{L^1_\xi}=\gamma\|F\|_{L^1_x}.
\end{align}
Equations \eqref{eq:largterm2} and \eqref{eq:resolvent-weight-scaling}
prove \eqref{eq:largterm1} and the strip estimate.
\end{proof}

\begin{proof}[Proof of Proposition \ref{thm:pwdisp}]
We first record the free estimate
\begin{align}\label{eq:free-local-smoothing}
\|\pi_1e^{t\mathcal L_{[c]}^{(0)}}F\|_
{L^2(\R,\widehat\Sigma)}\le C(c)\|F\|_{\mathcal H^1}.
\end{align}
For Schwartz data, the first component is the sum of the two Fourier integrals
\begin{align}\label{eq:free-branches}
I_\pm (t,x):=\frac1{4\pi}\int_\R e^{\im kx+\im t\omega_\pm(k)}
\left(\widehat F_1(k)\pm
\frac{\widehat F_2(k)}{\sqrt{1+k^2}}\right)\,dk,
\qquad
\omega_\pm(k)=ck\pm k\sqrt{1+k^2}.
\end{align}
By  
$|\omega_\pm'(k)|\ge1-c>0$ and Plancherel, for $x\in \R $ fixed
\begin{align*}
      \| I_\pm (t,x)\| _{L^2(\R _t)}  \approx    \| \left(\widehat F_1(k_\pm (\omega ))\pm
\frac{\widehat F_2(k_\pm (\omega ))}{\sqrt{1+k_\pm (\omega )^2}}\right)  k_\pm '(\omega )   \| _{L^2(\R _\omega )}\lesssim   \|F\|_{\mathcal H^1}
\end{align*}
so that the following yields \eqref{eq:free-local-smoothing},
\begin{align*}
      \| \sech  (\gamma\kappa x)  I_\pm (t,x)\| _{L^2(\R ^2)}   \lesssim   \|F\|_{\mathcal H^1}.
\end{align*}
Equivalently this yields  the following, for all $F \in \mathcal H^1$ and all $0\le \varepsilon \le \varepsilon _0(c)\ll 1$
\begin{align}\label{eq:free-resolvent-l2}
\|\pi_1R_{\mathcal L_{[c]}^{(0)}}( \im\lambda   +\varepsilon )F\|_
{L^2_\lambda(\R,\widehat\Sigma)}\le C(c)\|F\|_{\mathcal H^1}.
\end{align} 
We now set
 $F=\mathcal Q_{[c]}\mathbf v_0.$
 For $|\lambda|\geq\gamma^3\mathbf c_1$, we use
\eqref{eq:high-scalar-resolvent}. By
\eqref{eq:high-scalar-contraction}, after increasing $\mathbf c_1$
if necessary, the term containing $\widetilde v_1$ on the right hand
side of \eqref{eq:high-scalar-resolvent} can be absorbed into the
left hand side. The remaining two terms are precisely, after the
scaling in \eqref{itm:chvar1}--\eqref{eq:relkern2}, the first
component of the outgoing free resolvent. Therefore,
using also \eqref{eq:resolvent-weight-scaling} and
\eqref{eq:free-resolvent-l2}, we obtain for all $0\le \varepsilon \le \varepsilon _0(c) $
\begin{align}\label{eq:homogeneous-high-bound}
\|\pi_1R_{\mathcal L_{[c]}}(\im\lambda+\varepsilon)F\|_
{L^2(\{|\lambda|\geq\gamma^3\mathbf c_1\},\widehat\Sigma)}
\le C(c)\|F\|_{\mathcal H^1}.
\end{align}
 On  $|\lambda|\leq\gamma^3\mathbf c_1$,
\eqref{eq:smooth21strip} gives for all $0\le \varepsilon \le \varepsilon _0(c) $
\begin{align*}
\|\pi_1R_{\mathcal L_{[c]}}(\im\lambda+\varepsilon)
\mathcal Q_{[c]}\mathbf v_0\|_
{L^2(\{|\lambda|\leq\gamma^3\mathbf c_1\},\widehat\Sigma)}
\lesssim
\|\mathbf v_0\|_{L^{1,1}}.
\end{align*}
Moreover, by \eqref{eq:projections} and the exponential decay of
$\xi_1[c]$ and $\xi_2[c]$,
\begin{align}\label{eq:projected-energy-bound}
\|\mathcal Q_{[c]}\mathbf v_0\|_{\mathcal H^1}
\lesssim
\|\mathbf v_0\|_{\mathcal H^1}
+
\|\mathbf v_0\|_{L^{1,1}}.
\end{align}
Combining the last two estimates with
\eqref{eq:homogeneous-high-bound}, we conclude that for all $0\le \varepsilon \le \varepsilon _0(c) $
\begin{align}\label{eq:boundary-resolvent-L2}
\|\pi_1R_{\mathcal L_{[c]}}(\im\lambda+\varepsilon)
\mathcal Q_{[c]}\mathbf v_0\|_
{L^2(\R_\lambda,\widehat\Sigma)}
\lesssim
\|\mathbf v_0\|_{\mathcal H^1\cap L^{1,1}}.
\end{align}
This  and  the equality in \eqref{eq:smooth--}  yield the inequality in  \eqref{eq:smooth--} and \eqref{eq:pwdisp1}  by Fathou's lemma for $\varepsilon\to 0^+$.

\end{proof}

Finally we need the following addition to Proposition \ref{prop:smooth};
its proof is given in Sect. \ref{sec:proofsmoothreg}.
\begin{proposition}\label{prop:smoothreg}
For every $c\in((p-1)^{1/2}/2,1)$ there exists $C(c)>0$ such that
\begin{align}\label{eq:smoothreg1}
\left\|\pi_1\int_0^t e^{(t-s)\mathcal L_{[c]}}
\mathcal Q_{[c]}(0,\partial_xg_2(s))^\intercal\,ds\right\|_
{L^2(\R_+,\widehat\Sigma)}
\le C(c)\|g_2\|_{L^1(\R,L^2(\R_+))}.
\end{align}
\end{proposition}

\section{Bounding $\| v_1 \|_{   {\widehat{\Sigma}}     }$  by Kato smoothing: Proof of Lemma \ref{prop:smooth11}}\label{sec:smooth}

In this section the weight in $\widehat\Sigma$ is fixed at
$c=c_0$. All constants are uniform for $c(t)$ in a sufficiently
small compact interval about $c_0$.
Multiplying the   equation in \eqref{eq:sist1}  by $\zeta_{B}$  and using the linearized operator in \eqref{eq:Lin}

\begin{align} \label{eq:prekato1}
   \partial _t \( \zeta_{B} {\mathbf{v} }\)  &=   \mathcal{L}_{[c_0]} \( \zeta_{B} {\mathbf{v} }\) + [ \zeta_{B} , \mathcal{L}_{[c _0 ]} ] \mathbf{v}        - \zeta_{B}\partial _x(0, f  ( v_1 +Q_c) -f  (  Q_c)-  f  '(  Q_c)  v_1 )^\intercal \\& +      \zeta_{B} \(   \mathcal{L}_{[c ]} - \mathcal{L}_{[c_0]}  \)   {\mathbf{v} }  +(\dot D-c) \zeta_{B} (  \mathbf{Q}_c   +\mathbf{v} )' -\dot c  \zeta_{B} \partial _c\mathbf{Q}_c . \nonumber
\end{align}
Next, for $\mathcal P=\mathcal P_{[c_0]}$ and
$\mathcal Q=\mathcal Q_{[c_0]}$, we consider the decomposition
 \begin{align}
  \label{eq:deftildw}
    \zeta_{B}\mathbf{v}=\mathcal P \zeta_{B}\mathbf{v} +  \mathbf{z} \text { where } \mathbf{z}:=\mathcal  Q \zeta_{B}\mathbf{v} .
 \end{align}

\begin{lemma}\label{lem:disczbv} We have
\begin{align}\label{eq:disczbv}
\|\mathcal P\zeta_B\mathbf v\|_{L^2}
\lesssim B^{-1}\left(
\|\mathbf v\|_{\Sigma_{1AA_1}}+
\|\mathbf v\|_{\Sigma_{2AA_1}}\right).
\end{align}
\end{lemma}
\begin{proof}
By \eqref{61} and \eqref{eq:projections}, the coefficients of
$\xi_1[c_0]$ and $\xi_2[c_0]$ in $\mathcal P\zeta_B\mathbf v$ are
\begin{align}\label{eq:disczbv-coeff}
\langle\zeta_B\mathbf v,\eta_1[c_0]\rangle
&=\langle\zeta_B\mathbf v,\eta_1[c_0]-\eta_1[c]\rangle,\\
\langle\zeta_B\mathbf v,\eta_2[c_0]\rangle
&=\langle\mathbf v,(\zeta_B-1)\eta_2[c]\rangle
+\langle\zeta_B\mathbf v,\eta_2[c_0]-\eta_2[c]\rangle.
\nonumber
\end{align}
The partition of unity and $A\ge2B$ give
\begin{align}\label{eq:localized-energy-control}
\|\zeta_A\mathbf v\|_{\mathcal H^1}
&\le\|\mathbf v\|_{\Sigma_{1AA_1}}
+\|\mathbf v\|_{\Sigma_{2AA_1}},\\
\|\zeta_B/\zeta_A\|_{L^2}&\lesssim B^{1/2}.
\nonumber
\end{align}
Lemma \ref{lem:kernel}, the decay of $\eta_2[c]$, and
$|1-\zeta_B(x)|\lesssim B^{-1}|x|$ therefore imply
\begin{align*}
|\langle\zeta_B\mathbf v,\eta_1[c_0]\rangle|
&\lesssim |c-c_0|B^{1/2}\|\zeta_A\mathbf v\|_{L^2},\\
|\langle\zeta_B\mathbf v,\eta_2[c_0]\rangle|
&\lesssim(B^{-1}+|c-c_0|)\|\zeta_A\mathbf v\|_{L^2}.
\end{align*}
Now \eqref{eq:h1errorbounds} and $\delta B^{3/2}\ll1$ give
\eqref{eq:disczbv}, since $\xi_i[c_0]\in L^2$.
\end{proof}

For exponentially localized energy vectors $\mathbf w$, integration
by parts and \eqref{eq:M2.3}--\eqref{eq:M2.5} give
\begin{align}\label{eq:localized-projection-commutation}
\mathcal P\mathcal L_{[c_0]}\mathbf w
=-\xi_1[c_0]\langle\eta_2[c_0],\mathbf w\rangle
=\mathcal L_{[c_0]}\mathcal P\mathbf w.
\end{align}
Differentiated inputs in this identity are paired by integration
by parts. Applying $\mathcal Q$ to \eqref{eq:prekato1} and using
\eqref{eq:localized-projection-commutation}, we obtain
\begin{align}
    \partial _t \mathbf{z}  &= \mathcal{L}_{[c_0]}  \mathbf{z}  +\mathcal Q [ \zeta  _B,\mathcal{L}_{[c_0]} ]\mathbf{v}  \nonumber
\\&   + (\dot D-c) \mathcal  Q\zeta  _B        \xi _1[c]  -   \dot c \  \mathcal Q\zeta  _B   \xi _2[c]    \label{eq:gKdV22}
\\&  + \mathcal Q\zeta  _B    (\dot D-c )\mathbf{v} '
+  \mathcal Q  \zeta_{B} \left (  \mathcal{L}_{[c ]}- \mathcal{L}_{[c_0]}   \right )\mathbf{v}    \label{eq:gKdV23}  \\&    \label{eq:gKdV24}    - \mathcal Q \zeta  _B \partial _x(0, f  ( v_1 +Q_c) -f  (  Q_c)-  f  '(  Q_c)  v_1 )^\intercal  .  
\end{align}
 We have
\begin{align}\label{eq:expv1}
  \mathbf{z} &=e^{  t \mathcal{L} _{[c_0]}}  \mathbf{z}(0) +   \int _{0}^t e^{ (t-s) \mathcal{L} _{[c_0]}}\mathcal Q [\zeta_{B}, \mathcal{L} _{[c_0]} ]\mathbf{v} ds \\& \quad+     \int _{0}^t e^{  (t-s)\mathcal{L}_{[c_0]} } \text{lines \eqref{eq:gKdV22}--\eqref{eq:gKdV24}} ds.\label{eq:expv2}
\end{align}
\begin{lemma}\label{lem:w0} We have
\begin{align}\label{eq:w0}
\|\pi_1e^{t\mathcal L_{[c_0]}}\mathbf z(0)
\|_{L^2(\R_+,\widehat\Sigma)}\le\sqrt\delta.
\end{align}
\end{lemma}
\begin{proof}
By \eqref{eq:h1errorbounds} and Cauchy--Schwarz,
\begin{align*}
\|\zeta_B\mathbf v(0)\|_{\mathcal H^1\cap L^{1,1}}
\lesssim (1+B^{3/2})\|\mathbf v(0)\|_{\mathcal H^1}
\lesssim B^2\delta.
\end{align*}
Apply \eqref{eq:pwdisp1} and choose $\delta$ sufficiently small
after $B$ as in \eqref{eq:relABg}.
\end{proof}

\begin{lemma}\label{lem:gKdV22} We have
\begin{align}\label{eq:gKdV22-estimate}
\left\|\pi_1\int_0^t e^{(t-s)\mathcal L_{[c_0]}}
\text{line \eqref{eq:gKdV22}}\,ds\right\|_{L^2(I,\widehat\Sigma)}
\lesssim\|\dot D-c\|_{L^2(I)}+\|\dot c\|_{L^2(I)}.
\end{align}
\end{lemma}
\begin{proof}
The vectors $\xi_i[c]$ decay exponentially, uniformly for $c$
in the compact modulation interval. Consequently
\begin{align*}
\|\zeta_B((\dot D-c)\xi_1[c]-\dot c\xi_2[c])
\|_{L^{1,1}(\R,L^2(I))}
\lesssim\|\dot D-c\|_{L^2(I)}+\|\dot c\|_{L^2(I)}.
\end{align*}
Extend the source by zero outside $I$ and apply \eqref{eq:smooth1}.
\end{proof}

\begin{lemma}\label{lem:gKdV24} We have
\begin{align}\label{eq:gKdV24-estimate}
\left\|\pi_1\int_0^t e^{(t-s)\mathcal L_{[c_0]}}
\text{line \eqref{eq:gKdV24}}\,ds\right\|_{L^2(I,\widehat\Sigma)}
\lesssim\delta^{\frac{\min\{p,2\}-1}{4}}
\left(\|\mathbf v\|_{L^2(I,\Sigma_{1AA_1})}
+\|\mathbf v\|_{L^2(I,\Sigma_{2AA_1})}\right).
\end{align}
\end{lemma}
\begin{proof}
The product rule gives
\begin{align}\label{eq:nonlinear-product-rule}
\zeta_B\partial_x\bigl(f(Q_c+v_1)-f(Q_c)-f'(Q_c)v_1\bigr)
={}&\partial_x\bigl(\zeta_B(f(Q_c+v_1)-f(Q_c)-f'(Q_c)v_1)\bigr)
\nonumber\\
&-\zeta_B'\bigl(f(Q_c+v_1)-f(Q_c)-f'(Q_c)v_1\bigr).
\end{align}
For $1<p\le2$, the H\"older estimate for $f'$ gives the
following bound; for $p>2$ it follows from the local Lipschitz
bound for $f'$ on the bounded range of $Q_c+sv_1$:
\begin{align}\label{eq:nonlinear-remainder-bound}
|f(Q_c+v_1)-f(Q_c)-f'(Q_c)v_1|
\lesssim \|v_1\|_{L^\infty(I\times\R)}^{\min\{p,2\}-1}|v_1|.
\end{align}
Also, directly from \eqref{def:zetaC},
\begin{align}\label{eq:nonlinear-cutoff-norms}
\|\zeta_B/\zeta_A\|_{L^2}
+\|\zeta_B'/\zeta_A\|_{L^{2,1}}\lesssim B^{1/2}.
\end{align}
Apply \eqref{eq:smoothreg1} to the first term of
\eqref{eq:nonlinear-product-rule} and \eqref{eq:smooth1} to
the second. Cauchy--Schwarz in $x$, \eqref{eq:nonlinear-remainder-bound},
\eqref{eq:nonlinear-cutoff-norms}, and \eqref{eq:h1errorbounds}
bound the resulting expression by
\begin{align*}
C B^{1/2}(B^{1/2}\delta)^{\min\{p,2\}-1}
\|\zeta_Av_1\|_{L^2(I\times\R)}.
\end{align*}
By \eqref{eq:localized-energy-control} and the choice of $\delta$
after $B$, this proves \eqref{eq:gKdV24-estimate}.
The argument applies to energy solutions by the continuous
extensions in Propositions \ref{prop:smooth} and \ref{prop:smoothreg}.
\end{proof}

\begin{lemma}\label{lem:gKdV231} We have
\begin{align}\label{eq:gKdV231-estimate}
\left\|\pi_1\int_0^t e^{(t-s)\mathcal L_{[c_0]}}
\mathcal Q\zeta_B(\dot D-c)\mathbf v'\,ds
\right\|_{L^2(I,\widehat\Sigma)}
\lesssim\sqrt\delta\left(
\|\mathbf v\|_{L^2(I,\Sigma_{1AA_1})}
+\|\mathbf v\|_{L^2(I,\Sigma_{2AA_1})}\right).
\end{align}
\end{lemma}
\begin{proof}
Use \eqref{eq:smooth1} for the first component of the source.
For the second component use
\begin{align}\label{eq:transport-product-rule}
\zeta_B(\dot D-c)v_2'
=\partial_x\bigl((\dot D-c)\zeta_Bv_2\bigr)
-(\dot D-c)\zeta_B'v_2,
\end{align}
and apply \eqref{eq:smoothreg1} and \eqref{eq:smooth1}, respectively.
The resulting three source norms are bounded by
\begin{align*}
&\|\zeta_B(\dot D-c)v_1'\|_{L^{1,1}(\R,L^2(I))}
+\|(\dot D-c)\zeta_Bv_2\|_{L^1(\R,L^2(I))}
+\|(\dot D-c)\zeta_B'v_2\|_{L^{1,1}(\R,L^2(I))}\\
&\qquad\lesssim B^{3/2}
\|\dot D-c\|_{L^2(I)}\|\mathbf v\|_{L^\infty(I,\mathcal H^1)}\\
&\qquad\lesssim B^{3/2}\delta
\left(\|\mathbf v\|_{L^2(I,\Sigma_{1AA_1})}
+\|\mathbf v\|_{L^2(I,\Sigma_{2AA_1})}\right).
\end{align*}
Here Cauchy--Schwarz in $x$ is applied before integrating
$|\dot D-c|^2\|\mathbf v(t)\|_{\mathcal H^1}^2$ in $t$.
The last bound uses \eqref{eq:discrest1D} and \eqref{eq:h1errorbounds}.
Choosing $\delta$ sufficiently small proves \eqref{eq:gKdV231-estimate}.
\end{proof}

\begin{lemma}\label{lem:gKdV232} We have
   \begin{align}\label{eq:gKdV242}
     \left  \|     \pi _1 \int _{0}^t e^{  (t-s)\mathcal{L}_{[c_0]} }  \mathcal Q  \zeta_{B} \(  \mathcal{L}_{[c ]}- \mathcal{L}_{[c_0]}   \)\mathbf{v} ds\right \| _{L^2 (I, \widehat{\Sigma})}\lesssim   \sqrt{\delta}  (\| \mathbf{v} \| _{L^2(I,  { \Sigma }_{1A A_1}    )}+\| \mathbf{v} \| _{L^2(I,  { \Sigma }_{2A A_1}   )}) .
   \end{align}
\end{lemma}
\begin{proof}
The exact operator difference is
\begin{align}\label{eq:linearization-difference}
(\mathcal L_{[c]}-\mathcal L_{[c_0]})\mathbf v
=(c-c_0)\mathbf v'
+\bigl(0,\partial_x((f'(Q_{c_0})-f'(Q_c))v_1)\bigr)^\intercal.
\end{align}
By smooth dependence of the positive profiles on $c$, for some
$\nu>0$,
\begin{align}\label{eq:potential-difference}
|f'(Q_{c_0}(x))-f'(Q_c(x))|
\lesssim |c-c_0|e^{-\nu|x|},
\end{align}
uniformly on the modulation interval.
For the first term in \eqref{eq:linearization-difference}, use
\eqref{eq:transport-product-rule} with $\dot D-c$ replaced by
$c-c_0$, and use its $L^\infty(I)$ norm instead of its $L^2(I)$
norm. The source estimates now follow from
\eqref{eq:localized-energy-control} and
\begin{align*}
\|\zeta_B/\zeta_A\|_{L^{2,1}}\lesssim B^{3/2},
\qquad
\|c-c_0\|_{L^\infty(I)}\lesssim\delta.
\end{align*}
For the second term of \eqref{eq:linearization-difference},
move $\zeta_B$ inside the derivative and retain the remainder
with $\zeta_B'$. Equations \eqref{eq:smoothreg1},
\eqref{eq:smooth1}, and \eqref{eq:potential-difference} apply.
Altogether the left-hand side of \eqref{eq:gKdV242} is bounded by
\begin{align*}
C B^{3/2}\delta\left(
\|\mathbf v\|_{L^2(I,\Sigma_{1AA_1})}
+\|\mathbf v\|_{L^2(I,\Sigma_{2AA_1})}\right),
\end{align*}
which proves \eqref{eq:gKdV242} by decreasing $\delta$.
\end{proof}

We now examine the commutator term in \eqref{eq:expv1}.
Direct computation gives
\begin{align}\label{eq:cutoff-commutator}
[\zeta_B,\mathcal L_{[c_0]}]
={}&\left(3\zeta_B'\partial_x^2+3\zeta_B''\partial_x
+\zeta_B'''+\zeta_B'f'(Q_{c_0})\right)
\begin{pmatrix}0&0\\1&0\end{pmatrix}
-\zeta_B'\begin{pmatrix}c_0&1\\1&c_0\end{pmatrix}.
\end{align}
\begin{lemma}\label{lem:estcomm11} We have
\begin{align}\label{eq:comm-high-estimate}
&\left\|\pi_1\int_0^t e^{(t-s)\mathcal L_{[c_0]}}\mathcal Q
\left(3\zeta_B'\partial_x^2+3\zeta_B''\partial_x
+\zeta_B'''+\zeta_B'f'(Q_{c_0})\right)(0,v_1)^\intercal\,ds
\right\|_{L^2(I,\widehat\Sigma)}\nonumber\\
&\qquad\lesssim B^{-1/2}\left(
\|\mathbf v\|_{L^2(I,\Sigma_{1AA_1})}
+\|\mathbf v\|_{L^2(I,\Sigma_{2AA_1})}\right).
\end{align}
\end{lemma}
\begin{proof}
The first two differentiated terms satisfy the exact identity
\begin{align}\label{eq:commutator-product-rule}
3\zeta_B'v_1''+3\zeta_B''v_1'=3\partial_x(\zeta_B'v_1').
\end{align}
By \eqref{def:zetaC}, including its fixed transition region,
\begin{align}\label{eq:commutator-cutoff-bounds}
\|\zeta_B'/\zeta_A\|_{L^2}&\lesssim B^{-1/2},\\
\|\zeta_B'''/\zeta_A\|_{L^{2,1}}
+\|\zeta_B'f'(Q_{c_0})/\zeta_A\|_{L^{2,1}}
&\lesssim B^{-1}.
\nonumber
\end{align}
Indeed, outside $[-2,2]$ the $k$th derivative of $\zeta_B$
has size at most $B^{-k}\zeta_B$; on $[-2,2]$ each positive
derivative is $O(B^{-1})$.
Apply \eqref{eq:smoothreg1} to \eqref{eq:commutator-product-rule}
and \eqref{eq:smooth1} to the other two terms. The bounds
\eqref{eq:commutator-cutoff-bounds} and Cauchy--Schwarz reduce
the result to
\begin{align*}
C B^{-1/2}\left(
\|\zeta_Av_1'\|_{L^2(I\times\R)}
+\|\zeta_Av_1\|_{L^2(I\times\R)}\right).
\end{align*}
Since $|\zeta_A'/\zeta_A|\lesssim A^{-1}$,
\eqref{eq:localized-energy-control} proves \eqref{eq:comm-high-estimate}.
\end{proof}

\begin{lemma}\label{lem:estcomm21} We have
\begin{align}\label{eq:comm-low-estimate}
&\left\|\pi_1\int_0^t e^{(t-s)\mathcal L_{[c_0]}}
\mathcal Q\zeta_B'\begin{pmatrix}c_0&1\\1&c_0\end{pmatrix}
\mathbf v\,ds\right\|_{L^2(I,\widehat\Sigma)}\lesssim A_1^{3/2}B^{-1}
\|\mathbf v\|_{L^2(I,\Sigma_{1AA_1})}
+B^{1/2}\|\mathbf v\|_{L^2(I,\Sigma_{2AA_1})}.
\end{align}
\end{lemma}
\begin{proof}
By \eqref{eq:smooth1} and the partition of unity, the
left-hand side of \eqref{eq:comm-low-estimate} is bounded by
\begin{align*}
C\|\zeta_B'(\vartheta_{1A_1}+\vartheta_{2A_1})\mathbf v
\|_{L^{1,1}(\R,L^2(I))}.
\end{align*}
Use
\begin{align*}
\|\zeta_B'/\zeta_A\|_{L^\infty}&\lesssim B^{-1},&
\|1_{[-2A_1,2A_1]}\|_{L^{2,1}}&\lesssim A_1^{3/2},&
\|\zeta_B'/\zeta_A\|_{L^{2,1}}&\lesssim B^{1/2}.
\end{align*}
Cauchy--Schwarz in $x$ gives \eqref{eq:comm-low-estimate}.
\end{proof}

\textit{Proof of Lemma \ref{prop:smooth11}.}
The cutoff definition gives
\begin{align}\label{eq:local-cutoff-remainder}
\|\sech(\gamma\kappa x)(1-\zeta_B)v_1\|_{L^2}
\lesssim B^{-1}\|\zeta_Av_1\|_{L^2},
\qquad \gamma=\sqrt{1-c_0^2},
\end{align}
because $|1-\zeta_B(x)|\lesssim |x|/B$ and $A$ is sufficiently
large compared with $(\gamma\kappa)^{-1}$.
Combining \eqref{eq:deftildw}, \eqref{eq:disczbv}, and
\eqref{eq:local-cutoff-remainder}, we obtain
\begin{align}\label{eq:reconstruct-local-first}
\|v_1\|_{L^2(I,\widehat\Sigma)}
\le\|\pi_1\mathbf z\|_{L^2(I,\widehat\Sigma)}
+C B^{-1}\left(
\|\mathbf v\|_{L^2(I,\Sigma_{1AA_1})}
+\|\mathbf v\|_{L^2(I,\Sigma_{2AA_1})}\right).
\end{align}
Apply \eqref{eq:w0}, \eqref{eq:gKdV22-estimate},
\eqref{eq:gKdV24-estimate}, \eqref{eq:gKdV231-estimate},
\eqref{eq:gKdV242}, \eqref{eq:comm-high-estimate}, and
\eqref{eq:comm-low-estimate} to \eqref{eq:expv1}--\eqref{eq:expv2}.
The modulation contribution is bounded by
$C B^{-1/2}(\|\mathbf v\|_{L^2(I,\Sigma_{1AA_1})}
+\|\mathbf v\|_{L^2(I,\Sigma_{2AA_1})})$
by \eqref{eq:discrest1D}--\eqref{eq:discrest1c}.
Since $A_1=B^{3/5}$, the terms with coefficient $B^{-1/2}$
or $B^{-1}$ are bounded by the two explicit terms in
\eqref{eq:sec:smooth11}. Finally, \eqref{eq:main2} and a
sufficiently small choice of $\delta$ after $B$ and $\epsilon$
make the initial and nonlinear contributions
$o_{B^{-1}}(1)\epsilon$. Substitution in
\eqref{eq:reconstruct-local-first} proves \eqref{eq:sec:smooth11}.
\qed

In \S \ref{sec:jost}--\ref{sec:proofsmooth2} below we prove the small energy part of Proposition \ref{prop:smooth2}. Finally, the proof of   Proposition \ref{prop:smoothreg}    is in \S \ref{sec:proofsmoothreg}.

\section{The Jost functions}\label{sec:jost}
Throughout this section, $c$ is fixed in $((p-1)^{1/2}/2,1)$ and
$\gamma=\sqrt{1-c^2}$. Constants may depend on $p$ and $c$.
We use the first-order system \eqref{A_Decomposehom}, its adjoint
\eqref{A_djoint}, and the analytic free bases from
\eqref{eq:vj0}--\eqref{eq:bj}. In particular, the matrix
\begin{align}
    \label{eq:matrixbfA}\text{$\mathbf A=(\mathbf a_1,\mathbf a_2,\mathbf a_3,\mathbf a_4)$ is
invertible near zero and $\det\mathbf A(0)=-4$.}
\end{align}
The roots satisfy, on a sufficiently small closed right half-disk,
\begin{align}\label{eq:low-root-signs}
\Re\mu_1\le-\nu<0<\nu\le\Re\mu_4,\qquad
\Re\mu_2\le0\le\Re\mu_3,\qquad
|\mu_2|+|\mu_3|\le C|\Lambda|.
\end{align}
Indeed these follow from Lemma \ref{lem:roots1}, with the inequalities
for the central roots obtained by integrating their horizontal real
derivatives from the imaginary axis. We also have
$|\partial_\xi^n q(\xi)|\le C_ne^{-(p-1)|\xi|}$.
The rank-one matrices $\mathbf a_j\mathbf b_j^\intercal$ are analytic
at zero, even though the two central eigenvalues coincide there.

\begin{lemma}\label{lem:estimate_f}
Fix $T>0$ and $0<\rho<\min\{1,p-1\}$.
There exist a bounded open neighborhood $\mathcal O=\mathcal O_T\subset\C$
of $\im[-T,T]$ and solutions
$f_1,\widehat f_2,\widehat f_3,f_4,F_1,F_2,F_3,F_4$
of \eqref{A_Decomposehom}, defined on $\R\times\mathcal O$,
smooth in $\xi$ and holomorphic in $\Lambda$, with the following
representations and asymptotics:
\begin{align}\label{eq:limm1}
\begin{array}{c|c|c}
j & +\infty & -\infty\\ \hline
1 & f_1=e^{\mu_1\xi}m_1 \text{ with }m_1\xrightarrow{\xi\to +\infty}\mathbf a_1
  & F_1=e^{\mu_1\xi}M_1\text{ with }M_1\xrightarrow{\xi\to -\infty}\mathbf a_1\\
2 & \widehat f_2=e^{\mu_2\xi}\widehat m_2\text{ with }\widehat m_2\xrightarrow{\xi\to +\infty}\mathbf a_2
  & F_2=e^{\mu_2\xi}M_2\text{ with }M_2\xrightarrow{\xi\to -\infty}\mathbf a_2\\
3 & F_3=e^{\mu_3\xi}M_3\text{ with }M_3\xrightarrow{\xi\to +\infty}\mathbf a_3
  & \widehat f_3=e^{\mu_3\xi}\widehat m_3\text{ with }\widehat m_3\xrightarrow{\xi\to -\infty}\mathbf a_3\\
4 & F_4=e^{\mu_4\xi}M_4\text{ with }M_4\xrightarrow{\xi\to +\infty}\mathbf a_4
  & f_4=e^{\mu_4\xi}m_4\text{ with }m_4\xrightarrow{\xi\to -\infty}\mathbf a_4
\end{array}
\end{align}
The roots $\mu_j$ and the free bases $\mathbf a_j,\mathbf b_j$
are holomorphic on $\mathcal O$, which may be chosen invariant
under $\Lambda\mapsto-\Lambda$ and complex conjugation.

Furthermore, for every $n,m\ge0$ there is $C_{nm}>0$ such that
\begin{align}\label{eq:m1_estimate_partial}
\left|
\partial_\xi^n\partial_\Lambda^m
\bigl(\mathfrak m(\xi,\Lambda)-\mathbf a_j(\Lambda)\bigr)
\right|
\le C_{nm}e^{-\rho|\xi|},
\qquad \Lambda\in\mathcal O,
\end{align}
where $\mathfrak m$ ranges over all eight factors
$m_1,\widehat m_2,\widehat m_3,m_4,M_1,M_2,M_3,M_4$,
and $j$ is the index of its row in \eqref{eq:limm1}.
The estimate holds for $\xi\ge0$ for factors in the $+\infty$
column, and for $\xi\le0$ for factors in the $-\infty$ column.

There is also a solution $g_1$ of \eqref{A_djoint}, defined on
$\R\times\mathcal O$, smooth in $\xi$ and holomorphic in $\Lambda$,
such that
\begin{align}\label{eq:g1}
e^{\mu_1(\Lambda)\xi}g_1(\xi,\Lambda)
\xrightarrow{\xi\to-\infty} \mathbf b_1(\Lambda).
\end{align}
For every $n,m\ge0$, it satisfies
\begin{align}\label{eq:g1-normalized-estimate}
\left|
\partial_\xi^n\partial_\Lambda^m
\left(
e^{\mu_1(\Lambda)\xi}g_1(\xi,\Lambda)-\mathbf b_1(\Lambda)
\right)
\right|
\le C_{nm}e^{\rho\xi},
\qquad \xi\le0,\quad \Lambda\in\mathcal O.
\end{align}
For each $\Lambda\in\mathcal O$, the solutions $f_1,f_4,g_1$
are uniquely determined by their respective normalized limits.
No uniqueness is asserted for the other six solutions.

The constants $C_{nm}$ may depend on $p,c,\rho,T$,
but are independent of $\xi$ and $\Lambda$ in the stated ranges.
For each fixed $T$, the solutions can be chosen locally real analytic in $c$,
with a common neighborhood $\mathcal O$ and locally uniform
constants when $c$ varies near its fixed value.
\end{lemma}
\begin{proof}
We skip the proof which is standard, referring in general to references \cite{Bealsbook1988, DT1979} and, for a setup almost congruent to our problem here, to 
 \cite{CM25D1}. We only recall that 
\begin{align}
m_1(\xi)&=\mathbf a_1-\sum_{k=1}^4\mathbf a_k\int_\xi^\infty
 e^{(\mu_k-\mu_1)(\xi-y)}
 \langle q(y)m_1(y),\mathbf b_k\rangle_{\C^4}\,dy,
\label{eq:integral_eq_m1}\\
m_4(\xi)&=\mathbf a_4+\sum_{k=1}^4\mathbf a_k\int_{-\infty}^\xi
 e^{(\mu_k-\mu_4)(\xi-y)}
 \langle q(y)m_4(y),\mathbf b_k\rangle_{\C^4}\,dy.
\label{eq:integral_eq_m4}
\end{align}

\end{proof}

We have the following elementary result
\begin{lemma}\label{lem:symjost1} Consider the diagonal matrix $D_{3,1}=\diag (1,-1,1, 1)$.  For $\Lambda \in \im \R$  we have
    \begin{equation}\label{eq:symjost11}
        f_1(\xi , -\Lambda   )=D_{3,1} f_4(-\xi , \Lambda  ).
    \end{equation}

    \end{lemma}
 \begin{proof}
 Using  $ D_{3,1}^{-1}  =D_{3,1}$ and, by \eqref {eq:QQ},    $Q(-x) =Q(x)$ and $Q'(-x) =- Q'( x)$   we compute
 \begin{align*}
    D_{3,1} A _\infty^{(0)} D_{3,1} = - A _\infty^{(0)}, \quad  D_{3,1}A _\infty^{(1)} D_{3,1} =   A _\infty^{(1)} \text{ and }D_{3,1} q(-x) D_{3,1}=-q( x) .
 \end{align*}
 Then
 \begin{align*}
     \frac{d}{d\xi}\(  D_{3,1}f_4(-\xi , \Lambda    )  \) &= -D_{3,1}\(  A _\infty^{(0)}  + \Lambda A _\infty^{(1)} (\gamma)  +q(-\xi)   \)  f_4(-\xi , \Lambda   ) \\& =\(  A _\infty^{(0)}  - \Lambda A _\infty^{(1)} (\gamma)  +q(\xi)   \)  D_{3,1}f_4(-\xi , \Lambda   ).
 \end{align*}
  By \eqref{mu1_mu4} and   $D_{3,1} \mathbf a_4(\Lambda    )=  \mathbf a _1(-\Lambda    )  $  we obtain \eqref{eq:symjost11}
 from the asymptotic behavior  for $\xi \to +\infty $
 \begin{align*}
   D_{3,1}f_4(-\xi , \Lambda     ) \sim   e ^{-\mu _4(\Lambda   )\xi }D_{3,1} \mathbf a _4(\Lambda     )= e ^{\mu _1(-\Lambda   )\xi }  \mathbf a _1(-\Lambda     )\sim   f_1( \xi , -\Lambda     ) .
 \end{align*}
 \end{proof}

The matrices
\begin{align}\label{eq:complete-bases}
 Y_-=(F_1,F_2,\widehat f_3,f_4),\qquad
 Y_+=(f_1,\widehat f_2,F_3,F_4)
\end{align}
are fundamental matrices: Abel's identity and the own-end limits give
$\det Y_\pm=\det\mathbf A\ne0$. Consequently the  following $c_{jk}$'s   are analytic,
\begin{align}\label{eq:connections}
 \begin{aligned}
 f_1&=c_{11}F_1+c_{12}F_2+c_{13}\widehat f_3+c_{14}f_4,\\
 \widehat f_2&=c_{21}F_1+c_{22}F_2+c_{23}\widehat f_3+c_{24}f_4,\\
 \widehat f_3&=c_{31}f_1+c_{32}\widehat f_2+c_{33}F_3+c_{34}F_4,\\
 f_4&=c_{41}f_1+c_{42}\widehat f_2+c_{43}F_3+c_{44}F_4.
 \end{aligned}
\end{align}
With different values of $T$, the uniqueness assertion in Lemma \ref{lem:estimate_f}
identifies $f_1,f_4,g_1$ on overlapping imaginary intervals.
The  $ \widehat f_j$ are not uniquely defined, but expanding them in the bases \eqref{eq:complete-bases}
and comparing their normalized limits, shows that any two distinct  $ \widehat f_j $ and $ \widehat f_j^{\,\mathrm{new}}$ are related by
\begin{align}\label{eq:jost-choice-independence}
 \widehat f_2^{\,\mathrm{new}}=\widehat f_2+a(\Lambda)f_1,\qquad
 \widehat f_3^{\,\mathrm{new}}=\widehat f_3+b(\Lambda)f_4.
\end{align}
Consequently $\det(f_1,\widehat f_2,\widehat f_3,f_4)$ is uniquely defined  and, in particular, independent of the choice of $T$  
(cf.~\cite[Lemma~7.6]{GLM07ARMA} for an analogous invariance result).
Give any solution $\widetilde V$, resp. $W$, of   \eqref {A_Decomposehom}, resp.  \eqref{A_djoint}, it is elementary to show that
$\langle \widetilde V (\xi ) ,W(\xi )\rangle _{\C^4} $ is constant in the variable $\xi$. By the inner products in  \eqref{eq:zj0} and the asymptotical values in Lemma \ref{lem:estimate_f},  by Pego and Weinstein  \cite {PegoWei1} we can  define the Evans function  as 
\begin{align}
    \label{eq:evansfunct} D(\Lambda):=\langle f_1(\xi,\Lambda) , g_1(\xi,\Lambda)\rangle .
\end{align}
Notice that by  Pego and Weinstein  \cite {PegoWei1} we have 
\begin{align}\label{eq:limm1-infty}
\lim_{\xi\to-\infty}m_1(\xi,\Lambda)=D(\Lambda)\mathbf a_1(\Lambda).
\end{align}

\begin{lemma}\label{lem:strong-global}
For every $T>0$, there exists an open neighborhood
$\mathcal U_T\Subset\mathcal O_T$ of $\im[-T,T]$,
invariant under negation and complex conjugation, such that,
for all nonnegative integers $n,m$   
\begin{align}\label{eq:m4_estimate_partial}
 |\partial_\xi^n\partial_\Lambda^m m_1|
 +|\partial_\xi^n\partial_\Lambda^m m_4|\le C_{nm}   \text{ uniformly for
$\Lambda\in\mathcal U_T$ and $\xi\in\R$}.
\end{align}
The constants $C_{nm}$ may depend on $p,c,\rho,T$,
but are independent of $\xi$ and $\Lambda$ in the stated ranges.
Moreover
\begin{align}\label{eq:evans}
D(\Lambda)  
 =c_{11}(\Lambda) \text{ and } c_{44}(\Lambda)=D(-\Lambda).
\end{align}
\end{lemma}
\begin{proof}
The bounds \eqref{eq:m4_estimate_partial} are standard and well known consequences of bounds on the Riccati equations \eqref {eq:integral_eq_m1} and 
\eqref {eq:integral_eq_m4}, like in  \cite{DT1979}. 
By Lemma \ref{lem:estimate_f}, \eqref{eq:connections} and \eqref{eq:limm1-infty},  the following gives $D(\Lambda)=c_{11}(\Lambda)$, 
\begin{align*}
D(\Lambda)\mathbf a_1(\Lambda)  = \lim_{\xi\to-\infty}m_1(\xi,\Lambda)
 =c_{11}(\Lambda)\mathbf a_1(\Lambda).
\end{align*}
By the symmetry \eqref{eq:symjost11}, \eqref{mu1_mu4}  and by \eqref{eq:limm1-infty}  for $\Lambda \in \im \R$  we similarly have
\begin{align*}
  c_{44}(\Lambda   )\mathbf{a}_4( \Lambda    )&=\lim  _{\xi\to +\infty}  m_4(\xi,\Lambda    )  =    \lim  _{\xi \to -\infty}  D_{3,1} m_1(\xi ,-\Lambda   )  =D(-\Lambda    )D_{3,1} \mathbf{a}_1( - \Lambda  ) =  D(-\Lambda    ) \mathbf{a}_4(  \Lambda    ) 
\end{align*}
with the identity extending by analytic continuation outside the imaginary axis. 

\end{proof}

\subsection{Zero-energy identities}
Set
\begin{align*}
 \alpha=[2(p+1)]^{-1/(p-1)},\qquad
 \beta=-\frac{2c\alpha}{p-1}.
\end{align*}
Thus $Q(\xi)=\alpha^{-1}e^{-|\xi|}
(1+O(e^{-(p-1)|\xi|}))$ for $\xi \to \infty$, with the corresponding
differentiated asymptotics. Differentiating the rescaled profile equation
gives $L_+\Lambda_pQ=-2Q$, where $\Lambda_p$ is defined in \eqref{eq:partcphic}. Set
\begin{align}\label{eq:bigxi}
 \Xi_1[c]=\begin{pmatrix}Q'\\Q''\\Q'''\\-cQ'\end{pmatrix},\qquad
 \Xi_2[c]=\begin{pmatrix}
 -c\Lambda_pQ\\-c(\Lambda_pQ)'\\-c(\Lambda_pQ)''
 \\c^2\Lambda_pQ-\gamma^2Q
 \end{pmatrix}.
\end{align}
They satisfy $\Xi_1[c]'=(A_\infty^{(0)}+q)\Xi_1[c]$ and
$\Xi_2[c]'=(A_\infty^{(0)}+q)\Xi_2[c]-A_\infty^{(1)}(\gamma)\Xi_1[c]$.

\begin{lemma}\label{lem:jostlambda0}
With a dot denoting $\partial_\Lambda$ at zero,
\begin{align}
 f_1(\xi,0)&=-\alpha\Xi_1[c],&
 f_4(\xi,0)&=\alpha\Xi_1[c],\label{eq:jostlambda01}\\
 \dot f_1(\xi)&=\alpha\Xi_2[c]+\beta\Xi_1[c],&
 \dot f_4(\xi)&=-\alpha\Xi_2[c]+\beta\Xi_1[c].\label{eq:jostlambda02}
\end{align}
In particular all four vectors decay exponentially at both ends.
\end{lemma}
\begin{proof}
The identities without dots follow from the equation, the asymptotic
normalizations in \eqref{eq:limm1}, and the uniqueness assertion for
$f_1$ and $f_4$ in Lemma \ref{lem:estimate_f}. Differentiation gives
$(\partial_\xi-A_\infty^{(0)}-q)\dot f_1=A_\infty^{(1)}(\gamma)f_1(\xi,0)$.
The proposed vector solves this equation. Since
$\dot\mu_1=c$ and
$\dot{\mathbf a}_1=(0,c,-2c,-\gamma^2)^\intercal$,
the Jost normalization gives
\begin{align*}
 \dot f_1(\xi)=e^{-\xi}
 \left(c\xi\mathbf a_1(0)+\dot{\mathbf a}_1(0)
       +O((1+\xi)e^{-\rho\xi})\right),\qquad \xi\to+\infty.
\end{align*}
The proposed vector has exactly this expansion; in its first component
the constant $-2c/(p-1)$ from $\alpha\Xi_2[c]$ is canceled by
$\beta\Xi_1[c]$. Subtracting gives a homogeneous solution which is
$o(e^{-\xi})$ at $+\infty$, hence zero by the fundamental basis $Y_+$.
Differentiating \eqref{eq:symjost11} at $\Lambda=0$ and
replacing $\xi$ by $-\xi$ gives
\begin{align*}
 \dot f_4(\xi)=-D_{3,1}\dot f_1(-\xi).
\end{align*}
Since $Q$ is even, the definitions in \eqref{eq:bigxi} imply
\begin{align*}
 D_{3,1}\Xi_1[c](-\xi)&=-\Xi_1[c](\xi),&
 D_{3,1}\Xi_2[c](-\xi)&=\Xi_2[c](\xi).
\end{align*}
Consequently, as asserted in \eqref{eq:jostlambda02},
\begin{align*}
 \dot f_4(\xi)
 &=-D_{3,1}\bigl(\alpha\Xi_2[c](-\xi)
                       +\beta\Xi_1[c](-\xi)\bigr) =-\alpha\Xi_2[c](\xi)+\beta\Xi_1[c](\xi).
\end{align*}

\end{proof}

Equations \eqref{eq:jostlambda01}--\eqref{eq:jostlambda02} give
\begin{align}\label{eq:strong-wedge-jets}
 (f_1\wedge f_4)(\xi,0)&=0,\nonumber\\
 \left.\partial_\Lambda(f_1\wedge f_4)\right|_{\Lambda=0}
 &=(\alpha\Xi_2[c]+\beta\Xi_1[c])\wedge\alpha\Xi_1[c]\nonumber\\
 &\quad-\alpha\Xi_1[c]\wedge(-\alpha\Xi_2[c]+\beta\Xi_1[c])
 =0.
\end{align}
We use the following generalized adjoint vector, with the outgoing
constants fixed explicitly. Set
\begin{align}\label{eq:defE1c}
h(c)&:=\left(\frac{2c}{p-1}-1\right)\int_\R Q(y)\,dy,\nonumber\\
E_{13}[c](\xi)&:=-c\int_{-\infty}^{\xi}\Lambda_pQ(y)\,dy
                  +\frac{1+c}{2}h(c),\nonumber\\
E_{14}[c](\xi)&:=\int_{-\infty}^{\xi}Q(y)\,dy+\frac12h(c),\nonumber\\
E_1[c]&:=\bigl(-L_+E_{13}[c]+cE_{14}[c],
             -E_{13}'[c],E_{13}[c],E_{14}[c]\bigr)^\intercal,\\
E_2&:=\bigl((p-1)Q^p,-Q',Q,0\bigr)^\intercal.
\label{eq:bigeta12-def}
\end{align}
Here $L_+=-\partial_\xi^2+1-pQ^{p-1}$, as in \eqref{eq:systres1}.
Direct differentiation, using $L_+\Lambda_pQ=-2Q$, gives
\begin{align}\label{eq:defE1c0}
E_1'[c]&=-(A_\infty^{(0)}+q)^\intercal E_1[c]
 -(A_\infty^{(1)}(\gamma))^\intercal E_2,\\
E_2'&=-(A_\infty^{(0)}+q)^\intercal E_2.
\label{eq:bigeta12}
\end{align}
We consider the following solution to \eqref{A_djoint} for $\Lambda =0$
\begin{align}\label{eq:outgoing-Psi2}
\Psi_2(\xi):=
\left(\frac12-\frac{p(1+c)}2Q^{p-1},0,
             -\frac{1+c}2,-\frac12\right)^\intercal ,
\end{align}
  with
$\Psi_2(\xi )  \xrightarrow{\xi \to  \infty}   \mathbf b_2(0)$.
We have $E_1[c]+h(c)\Psi_2\xrightarrow{\xi \to -\infty}0$  and
\begin{align}\label{eq:E1-limit}
\lim_{\xi\to+\infty}(E_1[c](\xi)+h(c)\Psi_2(\xi))
=\left(\int_\R Q\right)
\left(\frac{2c}{p-1},0,-c\left(\frac2{p-1}-1\right),1\right)^\intercal.
\end{align}

\begin{lemma}\label{lem:adjoint-classification}
Every solution of $w'=-(A_\infty^{(0)}+q)^\intercal w$ which is polynomially bounded
at either end has the form
\begin{align}\label{eq:adjoint-classification}
 w=bE_2+a_3
 \begin{pmatrix}pQ^{p-1}-1\\0\\1\\0\end{pmatrix}
 +a_4\begin{pmatrix}c\\0\\0\\1\end{pmatrix}.
\end{align}
In particular it is bounded, and its spatial derivatives of positive
order decay exponentially. If it tends to zero at $-\infty$, it is a
multiple of $E_2$.
\end{lemma}
\begin{proof}
The component equations imply
$w_4'=0$, $w_2=-w_3'$, and $L_+w_3'=0$.
The solutions of $L_+v=0$ which are polynomially bounded at either
specified end are multiples of $Q'$: on a right tail, a second solution is
$Q'(\xi)\int_{\xi_0}^\xi Q'(y)^{-2}\,d y$ and grows like $e^\xi$.
The left tail follows by reflection.
Uniqueness of the scalar initial-value problem extends the conclusion
to the whole line. Hence $w_3=bQ+a_3$ and
$w_1=-L_+w_3+cw_4$, proving the formula.
\end{proof}

\begin{lemma}\label{lem:evelem}\label{lem:tildew}
The following hold:
\begin{align}
 &D(0)=D'(0)=0,\qquad
 D''(0)=\alpha^2\gamma^{\,3-\frac4{p-1}}q'(c)>0,
 \label{eq:ddDexplicit}\\
 &\widehat W(\Lambda):=\det(f_1,\widehat f_2,\widehat f_3,f_4)
 =-2\alpha^2\gamma^{\,3-\frac4{p-1}}q'(c)\Lambda^2
   +O(\Lambda^3) \text{ for } \Lambda \text{ near }0\label{eq:deftildew}
\end{align}
In particular   $D(\Lambda )\neq 0$ and $\widehat W(\Lambda)\neq 0$        for $\Lambda \in D_\C (0, r)\setminus \{ 0 \}$ for a sufficiently small
$r>0.$
\end{lemma}
\begin{proof}
We claim that for a scalar $b$ the following hold, 
\begin{align}\label{eq:g1E20}
   g_1(\xi,0)=\frac\alpha2E_2  \text{ and }
 \dot g_1(\xi)-\frac\alpha2(E_1[c]+h(c)\Psi_2)=bE_2  .
\end{align}
The first equality holds by comparing  asymptotic behaviors of two solutions that both decay to 0 as $\xi \to -\infty$
and thus are  proportional. Likewise, the function on the left of the 2nd equation in \eqref {eq:g1E20}, is a solution to \eqref{A_djoint} for $\Lambda =0$
with limit 0 for $\xi \to -\infty$. Then it is a multiple of $E_2$ and \eqref{eq:g1E20} is proved.
The decay of $f_1(\cdot,0)$, $\dot f_1$ and of all the derivatives of $g_1$  to 0 for $\xi \to -\infty$ gives
$D(0)=D'(0)=0$.  
Differentiating twice with respect to $\Lambda$    the equation of  $f_1$  and evaluating at $\Lambda=0$
  we obtain
\begin{align}
\partial_\xi\ddot f_1  
={}&
\left(A_\infty^{(0)}+q\right)
\ddot f_1 
 +
2A_\infty^{(1)}(\gamma)
\dot f_1 .
\label{eq:ddf1zero}
\end{align}

We will be using here $u\cdot v$  for  $\<u,v\>_{\C^4}$.
Using  \eqref{eq:bigeta12}  and the ensuing cancellation of the $\left(A_\infty^{(0)}+q(\xi)\right)$ terms, 
we have
\begin{align} 
\frac{d}{d\xi}
\ddot  f_1(\xi,0 ) \cdot E_2(\xi) 
={}&
\partial_\xi\ddot f_1(\xi,0 ) \cdot E_2(\xi)
 +
\ddot f_1(\xi,0)\cdot E_2'(\xi)
  =
2
(A_\infty^{(1)}(\gamma)
\dot f_1(\xi,0)) \cdot 
E_2(\xi) .
\label{eq:pairddf1}
\end{align}
We next relate  line \eqref{eq:pairddf1} and
$ D''(0 )$. Differentiating twice
\eqref{eq:evansfunct}   gives at $\Lambda =0$ \small
\begin{align}
D''(0 )
={}&
\ddot f_1 \cdot 
g_1 
 +
2
\dot  f_1\cdot 
\dot  g_1 
 + 
f_1 \cdot \ddot
 g_1 .
\nonumber
\end{align}
\normalsize
Taking the limit  $\xi\to-\infty$, by   Lemma \ref{lem:jostlambda0} and by
$\partial _\Lambda ^n g_1(\xi,\Lambda ) \xrightarrow{\xi \to -\infty}0$ for all $n$, it follows
\begin{align}
D''(0 )
=
\frac{\alpha}{2}
\lim_{\xi\to-\infty}
\ddot f_1(\xi,0 )\cdot 
E_2(\xi) ,
\label{eq:ddDlimit}
\end{align}
where we used also   \eqref{eq:g1E20}.
When  we integrate \eqref{eq:pairddf1}, for $R>0$, \small
\begin{align*}
&
\ddot f_1(R,0)\cdot E_2(R)
-
\ddot f_1(-R,0)\cdot E_2(-R)
 =
2\int_{-R}^{R}
(A_\infty^{(1)}(\gamma)
\dot f_1(\xi,0) )\cdot
E_2(\xi) \,d\xi .
\end{align*} \normalsize
Since $ 
\ddot f_1(R,0)\cdot E_2(R)
\xrightarrow{R \to +\infty}0$ by the exponential decay of both functions,
 we obtain
\begin{align}
\lim_{\xi\to-\infty}
\ddot f_1(\xi ,0)\cdot E_2(\xi )
=
-2\int_{\R}
(A_\infty^{(1)}(\gamma)
\dot f_1(\xi,0) )\cdot
E_2(\xi) \,d\xi .
\label{eq:ddf1integral}
\end{align}
Combining \eqref{eq:ddDlimit} and \eqref{eq:ddf1integral} yields
\begin{align}
D''(0 )
=
-\alpha 
\int_{\R}
(A_\infty^{(1)}(\gamma)
\dot f_1(\xi,0) )\cdot
E_2(\xi) \,d\xi .
\label{eq:ddDevans}
\end{align}
Notice that by
\[
A_\infty^{(1)}(\gamma)\Xi_1[c]
=
(0,0,2cQ',\gamma^2Q')^\intercal,
\]   we have
\begin{align}
\int_{\R}
 (
A_\infty^{(1)}(\gamma)\Xi_1[c]) \cdot E_2 \,d\xi
=
2c\int_{\R}QQ'\,d\xi
=0
\label{eq:Xi1termzero}
\end{align}
and by
\[
A_\infty^{(1)}(\gamma)\Xi_2[c]
=
\left(
0,0,
\gamma^2Q-2c^2\Lambda_pQ,
-c\gamma^2\Lambda_pQ
\right)^\intercal,
\]
we have
\[ (A_\infty^{(1)}(\gamma)
\Xi_2[c] )\cdot
E_2  
=
\gamma^2Q^2-2c^2Q\Lambda_pQ.
\]
Then by  Lemma \ref{lem:jostlambda0} and using  \eqref{eq:qprc},   we obtain the following, which gives \eqref{eq:ddDexplicit},
\begin{align}  \nonumber
& \int_{\R}
(A_\infty^{(1)}(\gamma)
\dot f_1(\xi,0) )\cdot
E_2(\xi) \,d\xi
  =  \alpha 
\int_{\R}
(A_\infty^{(1)}(\gamma)
\Xi_2[c] )\cdot
E_2 
\\&   =   \alpha   \(  \gamma^2  \|Q\|_{L^2}^2 -2c^2 \langle Q,\Lambda_pQ\rangle \)
= \alpha  \left[
\gamma^2
-
c^2\left(\frac4{p-1}-1\right)
\right]\|Q\|_{L^2}^2
\nonumber\\
&=
-\alpha (\gamma) \gamma^{3-\frac4{p-1}}q'(c)  .
\nonumber
\end{align}
where for the last line we used \eqref{eq:qprc}. This completes the proof of \eqref{eq:ddDexplicit}.

We next prove \eqref{eq:deftildew}.  
For $\widetilde v_1$ the first component of $\widehat f_2(\xi,0)$, then
$(L_+\widetilde v_1)'=0$ by \eqref{eq:systres1}.
By $\widetilde v_1(\xi)-1\xrightarrow{\xi\to+\infty}0$
and $\widetilde v_1''(\xi)\xrightarrow{\xi\to+\infty}0$,
we have $L_+\widetilde v_1=1$.
Consider $w:=Q'\widetilde v_1'-Q''\widetilde v_1$.
Then an elementary computation gives $w'=-Q'$, so that
$w=-Q+C$. By $w(\xi)\xrightarrow{\xi\to+\infty}0$,
it follows $w=-Q$.
For $\xi\to-\infty$, $\widetilde v_1$ could have a growing
term of the form $\sim\kappa e^{-\xi}$.
For $\kappa\neq0$, this would give
$w+2\alpha^{-1}\kappa\xrightarrow{\xi\to-\infty}0$.
But from $w=-Q$ we conclude $\kappa=0$.
In particular this implies $c_{21}(0)=0$.
Then
\begin{align*}
\widehat f_2(-\infty,0)
:=\lim_{\xi\to-\infty}\widehat f_2(\xi,0)
=a\mathbf a_2(0)+b\mathbf a_3(0) \text{ where } \mathbf a_2(0)=(1,0,0,-1)^\intercal,
\quad
\mathbf a_3(0)=(1,0,0,1)^\intercal.
\end{align*}
In particular $\widetilde v_1''(\xi)\xrightarrow{\xi\to -\infty}0$,  
so that $L_+\widetilde v_1=1$ gives
$\widetilde v_1(-\infty)=1=a+b$, whereas
$\widetilde v_4(-\infty)=-a+b$.
Since $c\widetilde v_1'+\widetilde v_4'=0$
and
$c\widetilde v_1+\widetilde v_4\xrightarrow{\xi\to+\infty}c-1$,
we have
\begin{align*}
c\widetilde v_1+\widetilde v_4=c-1.
\end{align*}
Taking the limit as $\xi\to-\infty$ and using
$\widetilde v_1(-\infty)=1$, we obtain
\begin{align*}
\widetilde v_4(-\infty)=-1,
\qquad -a+b=-1.
\end{align*}
Together with $a+b=1$, this yields $a=1$ and $b=0$.
Hence 
$ \widehat f_2(\pm \infty,0)=\mathbf a_2(0).$  
The same argument, with the ends interchanged, gives
$\widehat f_3(\pm\infty,0)=\mathbf a_3(0)$. Consequently
\begin{align}\label{eq:central-connections-zero}
 c_{21}(0)=c_{23}(0)=c_{34}(0)=0,\qquad c_{22}(0)=1.
\end{align}
Also $c_{21}=\widehat f_2\cdot g_1$, so evaluation of its derivative at
$+\infty$, using \eqref{eq:E1-limit}, gives
\begin{align}\label{eq:c21-prime}
 c_{21}'(0)=\frac\alpha2 h(c).
\end{align}
The term $\dot{\widehat f}_2\cdot E_2$ vanishes in that limit.
Expanding the second column by \eqref{eq:connections} gives
\begin{align*}
 \widehat W=c_{21}H+c_{22}\det(f_1,F_2,\widehat f_3,f_4),\qquad
 H=\det(f_1,F_1,\widehat f_3,f_4).
\end{align*}
Equation \eqref{eq:strong-wedge-jets} implies $H(0)=H'(0)=0$.
The last determinant is $\det\mathbf A\,D$, either by the first
connection column or by its adjoint cofactor. Since
$\det\mathbf A(0)=-4$, \eqref{eq:central-connections-zero} and
$c_{21}=O(\Lambda)$ give  \eqref{eq:deftildew}.

\end{proof}

\subsection{Adjoint cofactors and the outgoing orthogonality conditions}

For $i=1,4$ we consider
$\widehat Z_i:=*(\widehat f_2\wedge\widehat f_3\wedge f_i)$. Since the trace of the matrices in \eqref{A_Decompose} is zero, they solve  \eqref{A_djoint}.

\begin{lemma}\label{lem:cofactor-jets}
There are constants $b_1,b_4$ such that
\begin{align}\label{eq:partlambdaz}
 \begin{aligned}
 \widehat Z_4(0)&=2\alpha E_2,&
 \widehat Z_1(0)&=-2\alpha E_2,\\
 \dot{\widehat Z}_4&=2\alpha E_1[c]+b_4E_2,&
 \dot{\widehat Z}_1&=-2\alpha E_1[c]+b_1E_2,\qquad
 b_1+b_4=4\beta.
 \end{aligned}
\end{align}
\end{lemma}
\begin{proof}
At zero, the second identity in \eqref{eq:connections} and
\eqref{eq:central-connections-zero} give
$\widehat Z_4=*(F_2\wedge\widehat f_3\wedge f_4)
 =-\det\mathbf A(0)g_1=2\alpha E_2$.
The assertion for $\widehat Z_1$ follows from $f_1=-f_4$ at zero.
The differentiated adjoint equations have inhomogeneous terms
$\mp2\alpha A_\infty^{(1)}(\gamma)^\intercal E_2$.

We compute the left limit of the derivatives before identifying them.
The only nondecaying contribution from differentiating the second
identity in \eqref{eq:connections} is
$c_{21}'(0)F_1\wedge\widehat f_3\wedge f_i$.
Indeed $f_i(0)$ and $\dot f_i(0)$ decay at the left end, whereas
the derivatives of the left central columns grow at most linearly.
By \eqref{eq:hodgestar}, \eqref{eq:jostlambda01}, and the left normalization,
with all vectors on the right evaluated at $\Lambda=0$,
\begin{align*}
 \lim_{\xi\to-\infty}*(F_1\wedge\widehat f_3\wedge f_i)(\xi,0)
 &=(-1)^i*(\mathbf a_1\wedge\mathbf a_3\wedge\mathbf a_4)\\
 &=(-1)^i\det\mathbf A(0)\mathbf b_2(0),\qquad i=1,4.
\end{align*}
Together with \eqref{eq:c21-prime}, this gives
\begin{align*}
 \dot{\widehat Z}_i(-\infty)
 =(-1)^i\det\mathbf A(0)c_{21}'(0)\mathbf b_2(0)
 =(-1)^{i+1}2\alpha h(c)\,\mathbf b_2(0).
\end{align*}
These are exactly the left limits of $2\alpha E_1[c]$ for $i=4$ and
$-2\alpha E_1[c]$ for $i=1$. Their differences solve the homogeneous
adjoint equation and decay at the left end. By
Lemma \ref{lem:adjoint-classification} they are multiples of $E_2$.
Finally differentiating
$\widehat Z_1+\widehat Z_4=*(\widehat f_2\wedge\widehat f_3\wedge(f_1+f_4))$
and using $\dot f_1+\dot f_4=2\beta\Xi_1[c]$ proves
$b_1+b_4=4\beta$.
\end{proof}

\begin{lemma}\label{lem:relkern}
For $F\in L^{1,1}(\R,\C^2)$ and $\widetilde{\mathbf F}$ as in
\eqref{eq:relkern2}, we have
\begin{align}
\<F,\eta_2[c]\>
 &=-\frac{\gamma^{\frac2{p-1}+2}}{q'(c)}
       \<\widetilde{\mathbf F},E_2\>,\label{eq:relkern3}\\
\<F,\eta_1[c]\>
 &=\frac{\gamma^{\frac2{p-1}-1}}{q'(c)}
       \<\widetilde{\mathbf F},E_1[c]\>
   +\frac{d_1(c)^2+d_2(c)^2}{2q'(c)}\<F,\eta_2[c]\>.
\label{eq:relkern4}
\end{align}
Thus $\mathcal P_{[c]}F=0$ is equivalent to
$\<\widetilde{\mathbf F},E_2\>=\<\widetilde{\mathbf F},E_1[c]\>=0$.
\end{lemma}
\begin{proof}
The first identity follows from the third component of
\eqref{eq:relkern2} and $\xi=\gamma x$. For the second, direct
substitution of the primitive terms in \eqref{eq:ker2}, using
$\partial_cQ_c(x)=-c\gamma^{\frac2{p-1}-2}(\Lambda_pQ)(\gamma x)$
from \eqref{eq:partcphic},
gives their pairing as
$\gamma^{\frac2{p-1}-1}\<\widetilde{\mathbf F},
E_1[c]+h(c)\Psi_2\>/q'(c)$.
The constants are accounted for by
\begin{align*}
\<\widetilde{\mathbf F},\Psi_2\>
 &=\frac{1+c}{2\gamma^2}\int_\R(F_2-F_1)\,dx,\\
s_c&=\gamma^{\frac2{p-1}-3}(1+c)h(c).
\end{align*}
Thus the constant term $s_c(1,-1)^\intercal/(2q'(c))$ in
\eqref{eq:ker2} cancels the $h(c)\Psi_2$ term. Finally,
$\theta_2\sigma_1\mathbf Q_c=\theta_2q'(c)\eta_2[c]$,
which gives the last term in \eqref{eq:relkern4}.
\end{proof}

We also define for $i=2,3$
\begin{align}\label{eq:defhatZ12}
  \widehat{Z}_i(\xi  ,  \Lambda     )&:= *(  \widehat{f}_i(\xi  ,  \Lambda    )\wedge  {f} _1(\xi  ,  \Lambda    )\wedge f_4(\xi  ,  \Lambda   ))   .
\end{align}
Then  we have \begin{align}
    \label{eqd01hat23} \text{ $\partial ^{a}_\Lambda  \widehat{Z}_i(\xi, 0    )=0$ for $a=0,1 $   and for both $i=2,3$.}
\end{align}
For $a=0$ this follows from   $ f_1(\xi,0)=- f_4(\xi,0)$, while for $a=1$
\begin{align*}
  & \partial _\lambda \(   \widehat{f}_i \wedge  {f} _1 \wedge f_4    \)   =    \widehat{f}_i \wedge  \partial _\lambda ({f} _1 \wedge f_4)
\end{align*}
where   at $\lambda =0$,  by Lemma \ref{lem:jostlambda0} and by the antisymmetry of the wedge product,
\begin{align*}
  \partial _\lambda ({f} _1 \wedge f_4) & =   \partial _\lambda {f} _1 \wedge f_4  +   {f} _1 \wedge      \partial _\lambda f_4 =  \(\alpha\Xi_2[c]+\beta\Xi_1[c] \)   \wedge \alpha  \Xi _1  - \alpha  \Xi _1    \wedge    \( -\alpha\Xi_2[c]+\beta\Xi_1[c]   \) =0.
\end{align*}

\subsection{Modified Jost solutions}
The analytic functions
\begin{align}\label{eq:defc0}
 c_2(\Lambda)=\frac{\Lambda c_{21}(\Lambda)}{D(\Lambda)},
 \qquad
 c_3(\Lambda)=\frac{\Lambda c_{34}(\Lambda)}{D(-\Lambda)}
\end{align}
have removable singularities at zero: by
\eqref{eq:ddDexplicit} and \eqref{eq:central-connections-zero},
\begin{align*}
 D(\Lambda)&=\tfrac12D''(0)\Lambda^2+O(\Lambda^3),\qquad D''(0)>0,\\
 c_{21}(\Lambda)&=O(\Lambda),\qquad c_{34}(\Lambda)=O(\Lambda).
\end{align*}
Define
\begin{align}
 f_2&=\Lambda\widehat f_2-c_2 f_1,\label{eq:deff2}\\
 f_3&=\Lambda\widehat f_3-c_3 f_4,\label{eq:deff23}\\
 m_j&=e^{-\mu_j\xi}f_j\quad(1\le j\le4).\nonumber
\end{align}
For $r>0$ small write
$\Omega_r=\{\Lambda:0<|\Lambda|<r,\ \Re\Lambda\ge0\}$ and
$\xi_+=\max(\xi,0)$, $\xi_-=\max(-\xi,0)$.

\begin{lemma}\label{lem:bdm2}
For all nonnegative integers $n,m$, uniformly on $\Omega_r$,
\begin{align}\label{eq:estm21}
 \begin{aligned}
 |\partial_\xi^n\partial_\Lambda^m m_1|
 +|\partial_\xi^n\partial_\Lambda^m m_4|&\le C_{nm},\\
 |\partial_\xi^n\partial_\Lambda^m m_2|&\le C_{nm}\langle\xi_-\rangle^m,\\
 |\partial_\xi^n\partial_\Lambda^m m_3|&\le C_{nm}\langle\xi_+\rangle^m.
 \end{aligned}
\end{align}
The vectors $m_2(\cdot,0)$ and $m_3(\cdot,0)$ are exponentially
localized. Set
\begin{align*}
 Z_i=*(f_i\wedge f_1\wedge f_4)\quad(i=2,3),\qquad
 H_3=e^{\mu_2\xi}Z_3,\qquad H_2=e^{\mu_3\xi}Z_2.
\end{align*}
Then
\begin{align}\label{eq:b13_estimate_partial}
 |\partial_\xi^n\partial_\Lambda^m H_3|
 \le C_{nm}\langle\xi_+\rangle^m,\qquad
 |\partial_\xi^n\partial_\Lambda^m H_2|
 \le C_{nm}\langle\xi_-\rangle^m.
\end{align}
Both $Z_i$ vanish to order at least three in $\Lambda$ at zero.
The vectors
\begin{align}\label{eq::b13at03}
 T_i(\xi)=\partial_\Lambda^3Z_i(\xi,0)\quad(i=2,3)
\end{align}
are bounded solutions of $T_i'=-(A_\infty^{(0)}+q)^\intercal T_i$ and, for some $\delta>0$,
$|T_i'(\xi)|\le C e^{-\delta|\xi|}$.
\end{lemma}
\begin{proof}
The first line is Lemma \ref{lem:strong-global}. On the right halfline,
the definition of $f_2$ proves the bound with no polynomial loss:
the ratio $e^{(\mu_1-\mu_2)\xi}$ has a fixed strict decay rate.
On the left, \eqref{eq:evans} and \eqref{eq:defc0} give
\begin{align*}
 \Lambda c_{21}-c_2c_{11}
 =\Lambda c_{21}-\frac{\Lambda c_{21}}D D=0.
\end{align*}
Substituting \eqref{eq:connections} into \eqref{eq:deff2} therefore yields
\begin{align*}
 f_2={}&(\Lambda c_{22}-c_2c_{12})F_2 +(\Lambda c_{23}-c_2c_{13})\widehat f_3
 +(\Lambda c_{24}-c_2c_{14})f_4.
\end{align*}
After normalization by $e^{-\mu_2\xi}$, the latter two ratios are
$e^{(\mu_3-\mu_2)\xi}$ and $e^{(\mu_4-\mu_2)\xi}$.
The second has strict decay; the first has modulus at most one on
$\xi\le0$ and its $m$-th parameter derivative costs at most
$\langle\xi_-\rangle^m$. This proves the assertion for $m_2$.
For $f_3$, the corresponding right-end calculation gives
\begin{align*}
 \Lambda c_{34}-c_3c_{44}
 &=\Lambda c_{34}-\frac{\Lambda c_{34}}{D(-\Lambda)}D(-\Lambda)=0,\\
 f_3&=(\Lambda c_{31}-c_3c_{41})f_1 +(\Lambda c_{32}-c_3c_{42})\widehat f_2
 +(\Lambda c_{33}-c_3c_{43})F_3.
\end{align*}
Reflection of the same estimates proves its bound.
At zero, \eqref{eq:deff2} gives
$f_2(0)=-c_2(0)f_1(0)$ and
$f_3(0)=-c_3(0)f_4(0)$, proving localization.

The zero sum of the four roots gives
\begin{align*}
 H_3=*(m_3\wedge m_1\wedge m_4),\qquad
 H_2=*(m_2\wedge m_1\wedge m_4),
\end{align*}
so \eqref{eq:b13_estimate_partial} follows. Equation \eqref{eq:strong-wedge-jets} gives
$f_1\wedge f_4=O(\Lambda^2)$.
Moreover $Z_i=\Lambda*(\widehat f_i\wedge f_1\wedge f_4)$, proving
the order-three vanishing. Differentiating its adjoint equation three
times thus leaves the homogeneous equation for $T_i$.
The directional bounds imply polynomial growth of $T_i$.
Lemma \ref{lem:adjoint-classification} proves its boundedness and the
exponential decay of its derivative.
\end{proof}

\section{Proof of the small energy part of Proposition \ref{prop:smooth2}}\label{sec:proofsmooth2}

Using   the notation in Lemma \ref{lem:transys} it is enough to prove
that for   $\mathcal P_{[c]}  {F}=0$
\begin{align}\label{eq:pfsmooth21}&
       \left \|        \widetilde{V}  \right \| _{       \widehat{\Sigma}    }\le C(  p, \gamma , \varepsilon _1)  \|  {F} \| _{   L ^{1,1}\( \R\)    }  \text{ for all $\Lambda \in \overline{\C} _+ \backslash \{  0 \}$
 with $|\Lambda |\le \varepsilon _1$} 
   \end{align}
 for \begin{align}\label{eqvarpartilde}
    \widetilde{V}=\widehat{R}(\Lambda) \widetilde{\mathbf{F}} = \sum _{j=1}^{4}\widehat{R}_j(\Lambda) \widetilde{\mathbf{F}}   \end{align}
  where   $ \widetilde{\mathbf{F}} $ is defined in \eqref{eq:relkern2} and
    \begin{align}& \widehat{R}_1(\Lambda) \widetilde{\mathbf{F}}=
    \frac{f_1(\xi,\Lambda  )}{\widehat{W}(\Lambda   ) } \int_  {-\infty} ^\xi    \det \( \widetilde{\mathbf{F}}(y) ,   \widehat{f}_2(y,\Lambda   ) ,  \widehat{f}_3(y,\Lambda   ) ,   {f}_4(y,\Lambda  )   \)  dy  ,    \label{eq:r1}  \\&   \widehat{R}_2(\Lambda) \widetilde{\mathbf{F}}=
    \frac{\widehat{f}_2(\xi,\Lambda   )}{\widehat{W}(\Lambda   ) }   \int_  {-\infty} ^\xi    \det \( {f}_1(y,\Lambda   )  ,  \widetilde{\mathbf{F}}(y) ,  \widehat{f}_3(y,\Lambda  ) ,   f _4(y,\Lambda  )   \) dy  , \label{eq:r2}
   \\&  \widehat{R}_3(\Lambda) \widetilde{\mathbf{F}}= -
    \frac{\widehat{f}_3(\xi,\Lambda  )}{\widehat{W}(\Lambda   ) }   \int_  {\xi}  ^{+\infty}    \det \( {f}_1(y,\Lambda    )  ,  \widehat{f}_2(y,\Lambda   ) ,  \widetilde{\mathbf{F}}(y)   ,   {f}_4(y,\Lambda   )   \)   dy, \label{eq:r3} \\&  \widehat{R}_4(\Lambda) \widetilde{\mathbf{F}}= -
     \frac{f_4(\xi,\Lambda  )}{\widehat{W}(\Lambda   ) } \int_  {\xi}  ^{+\infty}   \det \( {f}_1(y,\Lambda      )  ,  \widehat{f}_2(y,\Lambda     ) ,  \widehat{f}_3(y,\Lambda     ) , \widetilde{\mathbf{F}} (y)      \)  dy . \label{eq:r4}
 \end{align}
Notice that \eqref{eqvarpartilde} gives a $ \widetilde{V}$ which is a solution of \eqref{A_Decompose} as can be seen 
applying the operator in the left in \eqref{A_Decomposehom} to it, which gives the following, computed  in $\xi$,  where the equality holds for any $ \widetilde{\mathbf{F}}$,  \small
\begin{align}\label{eq:identdet}
   \frac{f_1  \det  ( \widetilde{\mathbf{F}} ,   \widehat{f}_2  ,  \widehat{f}_3  ,   {f}_4    )+\widehat f_2  \det  ( f_1 , \widetilde{\mathbf{F}} ,    \widehat{f}_3  ,   {f}_4    )+ \widehat f_3  \det  ( f_1 ,    \widehat{f}_2  ,  \widetilde{\mathbf{F}} ,   {f}_4    )  + f_4 \det ( {f}_1   ,  \widehat{f}_2  ,  \widehat{f}_3  , \widetilde{\mathbf{F}}        )}{\det ( {f}_1   ,  \widehat{f}_2  ,  \widehat{f}_3  , {f}_4        )} = \widetilde{\mathbf{F}} .
\end{align}
\normalsize
For $\Re \Lambda >0$, by Lemmas \ref{lem:roots1}   and \ref{lem:estimate_f} there exist constants $C_\Lambda , \delta _\Lambda \in \R _+$ such that  pointwise 
\begin{align*}
    |\widehat{R}(\Lambda) \widetilde{\mathbf{F}} |\le C_\Lambda  \int _\R e^{-\delta _\Lambda | \xi -y |}|\widetilde{\mathbf{F}} (y)| dy.
\end{align*}
 By Young's convolution estimate we conclude that $\widehat{R}(\Lambda)$ for $\Re \Lambda >0$ is a a bounded in operator in $L^2(\R , \C ^4)$. 
Notice that, for the same reasons and using   \eqref{eq:deff2}--\eqref{eq:deff23}, we also have
\begin{align}\label{eqvarpar}
    \widetilde{V}= {R}(\Lambda) \widetilde{\mathbf{F}} = \sum _{j=1}^{4} {R}_j(\Lambda) \widetilde{\mathbf{F}}   \end{align}
  where
    \begin{align}&  {R}_1(\Lambda) \widetilde{\mathbf{F}}=
    \frac{f_1(\xi,\Lambda   )}{\Lambda ^2 \widehat{W}(\Lambda  ) } \int_  {-\infty} ^\xi    \det \( \widetilde{\mathbf{F}}(y) ,   {f}_2(y,\Lambda   ) ,   {f}_3(y,\Lambda   ) ,   {f}_4(y,\Lambda  )   \)  dy  ,    \label{eq:1r1}  \\&    {R}_2(\Lambda) \widetilde{\mathbf{F}}=
    \frac{ {f}_2(\xi,\Lambda  )}{\Lambda ^2  \widehat{W}(\Lambda   ) }   \int_  {-\infty} ^\xi    \det \( {f}_1(y,\Lambda  )  ,  \widetilde{\mathbf{F}}(y) ,   {f}_3(y,\Lambda   ) ,   f _4(y,\Lambda   )   \) dy ,  \label{eq:1r2}
   \\&   {R}_3(\Lambda) \widetilde{\mathbf{F}}= -
    \frac{ {f}_3(\xi,\Lambda   )}{\Lambda ^2  \widehat{W}(\Lambda   ) }   \int_  {\xi}  ^{+\infty}    \det \( {f}_1(y,\Lambda    )  ,   {f}_2(y,\Lambda    ) ,  \widetilde{\mathbf{F}}(y)   ,   {f}_4(y,\Lambda   )   \)   dy, \label{eq:1r3} \\&   {R}_4(\Lambda) \widetilde{\mathbf{F}}= -
     \frac{f_4(\xi,\Lambda   )}{\Lambda ^2  \widehat{W}(\Lambda   ) } \int_  {\xi}  ^{+\infty}   \det \( {f}_1(y,\Lambda      )  ,   {f}_2(y,\Lambda    ) ,   {f}_3(y,\Lambda    ) , \widetilde{\mathbf{F}} (y)      \)  dy . \label{eq:1r4}
 \end{align}
From the definition of $\widehat{Z}_i$  for $i\in \{ 1,4\}$ and \eqref{eq:hodgestar}
 \begin{align*}&
    \widehat{R}_1(\Lambda) \widetilde{\mathbf{F}}=-  \frac{  f_1(\xi,\Lambda   ) }{ \widehat{W}(\Lambda  ) }   \int_  {-\infty} ^\xi   \<  \widetilde{\mathbf{F}}(y),  \widehat{Z}_4(y, \Lambda    )\> _{\C^4} dy  \text{ and}\\&   \widehat{R}_4(\lambda) \widetilde{\mathbf{F}}=-  \frac{  f_4(\xi,\Lambda     ) }{\widehat{W}(\Lambda  )  }   \int_  {\xi }  ^{+\infty}  \<  \widetilde{\mathbf{F}}(y),  \widehat{Z}_1(y,\Lambda     )\> _{\C^4} dy .
 \end{align*}
 Similarly, 
  \begin{align*}&
    \widehat{R}_2(\Lambda) \widetilde{\mathbf{F}}= - \frac{\widehat{f}_2(\xi,\Lambda   )}{\widehat{W}(\Lambda   ) } \int_  {-\infty} ^\xi   \< \widetilde{ \mathbf{F}}(y), \widehat{Z}_3(y, \Lambda    )\> _{\C^4} dy  \text{ and } \\&   \widehat{R}_3(\Lambda) \widetilde{\mathbf{F}}=   -  \frac{\widehat{f}_3(\xi,\Lambda   )}{\widehat{W}(\Lambda   ) }    \int_  {\xi}  ^{+\infty}  \<  \widetilde{\mathbf{F}}(y), \widehat{Z}_2(y, \Lambda    )\> _{\C^4} dy .
 \end{align*}
Let us now write
 \begin{align}\label{eq:resolvf}
    \widehat{R} (\Lambda) \widetilde{\mathbf{F}}  = \frac{\widehat{N}(\Lambda)\widetilde{\mathbf{F}}}{\widehat{W}(\Lambda  ) } = \frac{ {N}(\Lambda)\widetilde{\mathbf{F}}}{\Lambda ^2\widehat{W}(\Lambda   ) }= {R} (\Lambda) \widetilde{\mathbf{F}}.
 \end{align}

\begin{lemma}\label{lem:numerator-jets}
We have for $b_c=\alpha b_4-2\alpha\beta$
\begin{align}\label{eq:numerator-jets}
 \widehat N(0) \widetilde{\mathbf{F}}
 &=2\alpha^2\Xi_1[c]\<   \widetilde{\mathbf{F}}  ,E_2\>,\\
 \widehat N'(0)\widetilde{\mathbf{F}}
 &=-2\alpha^2\Xi_2[c]\<   \widetilde{\mathbf{F}}   ,E_2\>
   +2\alpha^2\Xi_1[c]\<  \widetilde{\mathbf{F}}    ,E_1[c]\>
   +b_c\Xi_1[c]\<   \widetilde{\mathbf{F}}   ,E_2\> . \nonumber
\end{align}
\end{lemma}
\begin{proof}
By \eqref{eqd01hat23} we have 
 \begin{align*}
 \widehat N(0) \widetilde{\mathbf{F}}= -f_1(\xi , 0)\int_{-\infty}^{\xi}  \<  \widetilde{\mathbf{F}},    \widehat Z_4\> dy
 -f_4(\xi, 0)\int_{\xi}^{\infty}\<  \widetilde{\mathbf{F}},    \widehat Z_1\> dy  
\end{align*}
and the 1st line in \eqref{eq:numerator-jets} follows by
\eqref{eq:jostlambda01} and \eqref{eq:partlambdaz}  from
\begin{align*}
 -f_1(\xi,0)\widehat Z_4(y,0)^\intercal
 &=-f_4(\xi,0)\widehat Z_1(y,0)^\intercal =2\alpha^2\Xi_1[c](\xi)E_2(y)^\intercal.
\end{align*}
  For the  first derivatives,
\eqref{eq:jostlambda02} and \eqref{eq:partlambdaz} yield
\begin{align*}
 &-\left.\partial_\Lambda
    \bigl(f_1(\xi,\Lambda)\widehat Z_4(y,\Lambda)^\intercal\bigr)
   \right|_{\Lambda=0} =-\dot f_1(\xi)\widehat Z_4(y,0)^\intercal
         -f_1(\xi,0)\dot{\widehat Z}_4(y)^\intercal\\
 & =-2\alpha^2\Xi_2[c](\xi)E_2(y)^\intercal
        +2\alpha^2\Xi_1[c](\xi)E_1[c](y)^\intercal 
        +(\alpha b_4-2\alpha\beta)
           \Xi_1[c](\xi)E_2(y)^\intercal,
\end{align*}
and
\begin{align*}
 &-\left.\partial_\Lambda
    \bigl(f_4(\xi,\Lambda)\widehat Z_1(y,\Lambda)^\intercal\bigr)
   \right|_{\Lambda=0} =-\dot f_4(\xi)\widehat Z_1(y,0)^\intercal
         -f_4(\xi,0)\dot{\widehat Z}_1(y)^\intercal\\
 &\quad=-2\alpha^2\Xi_2[c](\xi)E_2(y)^\intercal
        +2\alpha^2\Xi_1[c](\xi)E_1[c](y)^\intercal 
        +(2\alpha\beta-\alpha b_1)
           \Xi_1[c](\xi)E_2(y)^\intercal.
\end{align*}
The last coefficients agree because
\begin{align*}
 (\alpha b_4-2\alpha\beta)
 -(2\alpha\beta-\alpha b_1)
 =\alpha(b_1+b_4-4\beta)=0.
\end{align*}
 Then the 2nd equality in \eqref{eq:numerator-jets} ensues.
\end{proof}

\begin{lemma}\label{lem:smooth2small} There exists an $a _4(p, \gamma  )>0$  and a  $C( p, \gamma )>0$ such that for $\widetilde{\mathbf{F}}$ 
such that  $\mathcal P_{[c]}  {F}=0$ we have
  \begin{align}\label{eq:smooth21small}
     \sup  _{ \substack{  \Lambda  \in \overline{\C} _+ \backslash \{  0 \}  \\ |\Lambda |\le a _4(p, \gamma  ) }} \left \|    \widehat{R}  (\Lambda )\widetilde{\mathbf{F}}    \right \| _{      \widehat{\Sigma}    }\le C( p, \gamma )     \|  \widetilde{\mathbf{F}} \| _{ L ^{1,1}\( \R\)} .
   \end{align}

\end{lemma}
\begin{proof}
By \eqref{eq:resolvf}, Lemma \ref{lem:numerator-jets} and Lemma \ref{lem:relkern}    it follows that 
\small
 \begin{align*}
  \widehat{R} (\Lambda) \widetilde{\mathbf{F}} =   \widehat{R} ^{(1)}(\Lambda) \widetilde{\mathbf{F}} :&=   \widehat{R} (\Lambda) \widetilde{\mathbf{F}} - \frac{ \sum _{0\le a \le 1}   \Lambda ^a  \partial _\Lambda ^a \widehat{N}(0)\widetilde{\mathbf{F}}}{\widehat{W}(\Lambda , \gamma )  } \\&  =   {R} (\Lambda) \widetilde{\mathbf{F}} - \frac{ \sum _{0\le a \le 3} \frac{\Lambda ^a }{a!} \partial _\Lambda ^a N(0)\widetilde{\mathbf{F}}}{\Lambda ^2\widehat{W}(\Lambda , \gamma ) }   =  \sum  _{1\le j \le 4} {R}_j^{(1)}(\Lambda) \widetilde{\mathbf{F}}
 \end{align*} \normalsize
where  
 \small
  \begin{align*}& {R}_1^{(1)}(\Lambda) \widetilde{\mathbf{F}}:=\frac{1 }{\Lambda ^2\widehat{W}(\Lambda   ) }
  \int _0 ^\Lambda dz \frac{(\Lambda -z )^3}{6}  \int_  {-\infty} ^\xi    \partial _z ^4\(   f_1(\xi ,z  )  \det \( \widetilde{\mathbf{F}}(y)  ,    {f}_2(y,z , \gamma ) ,  {f}_3(y,z   ) ,   {f}_4(y,z   )   \) \) dy\\ &{R}_2^{(1)}(\Lambda) \widetilde{\mathbf{F}}:=\frac{1 }{\Lambda ^2\widehat{W}(\Lambda   ) }
  \int _0 ^\Lambda dz \frac{(\Lambda -z )^3}{6}  \int_  {-\infty} ^\xi  \partial _z ^4\(
     {f}_2(\xi,z  )   \det \( {f}_1(y,z   )  ,  \widetilde{\mathbf{F}} (y),   {f}_3(y,z   ) ,   f _4(y,z  )   \)   \) dy   \\&   {R}_3^{(1)}(\Lambda) \widetilde{\mathbf{F}}:=-\frac{1 }{\Lambda ^2\widehat{W}(\Lambda , \gamma ) }
  \int _0 ^\Lambda dz \frac{(\Lambda -z )^3}{6}  \int_  {\xi}  ^{+\infty}  \partial _z ^4\(
     {f}_3(\xi,z  )     \det \( {f}_1(y,z   )  ,   {f}_2(y,z  ) ,  \widetilde{\mathbf{F}} (y)  ,   {f}_4(y,z   )   \)  \) dy   \\&   {R}_4^{(1)}(\Lambda) \widetilde{\mathbf{F}}:=-\frac{1 }{\Lambda ^2\widehat{W}(\Lambda  ) }
  \int _0 ^\Lambda dz \frac{(\Lambda -z )^3}{6}  \int_  {\xi}  ^{+\infty} \partial _z ^4\(
     {f}_4(\xi,z  )    \det \( {f}_1(y,z  )  ,   {f}_2(y,z  ) ,   {f}_3(y,z   ) , \widetilde{\mathbf{F} }  (y)    \)  \) dy .
 \end{align*}\normalsize
We have  \small
\begin{align*}&
  {R}_1^{(1)}(\Lambda) \widetilde{\mathbf{F}} =\frac{ 1 }{\Lambda ^2\widehat{W}(\Lambda)  }
  \int _0 ^\Lambda dz \frac{(\Lambda -z )^3}{6}  \int_  {-\infty} ^\xi \partial _z ^4\(
      e  ^{\mu _1(z) (\xi-y) }  m_1(\xi,z)  \< \widetilde{ \mathbf{F}}(y),  \widetilde{{n}}_1(y,z)   \> _{\C^4}\)  dy \text{ where} \\&  \widetilde{g}_1(y,z):=
-*\(  {f}_2(y,z) \wedge   {f}_3(y,z)  \wedge   {f}_4(y,z) \) \text{ and } \widetilde{n}_1(y,z):=  e  ^{\mu _1(z) y }\widetilde{g}_1(y,z).
\end{align*}\normalsize
Then since  by \eqref{eq:mus22} we have
\begin{equation*}
   \widetilde{n}_1(y,z) =
-*\(  m_2(y,z) \wedge   m_3(y,z)  \wedge   m_4(y,z) \) ,
\end{equation*}
from  \eqref{eq:m1_estimate_partial}, \eqref{eq:m4_estimate_partial} and \eqref{eq:estm21} and from Lemma \ref{lem:roots1} we can conclude  there exists a $ \kappa _0 >0$ such that
\begin{align}\label{eq:estR1(1)}
  \left |{R}_1^{(1)}(\Lambda) \widetilde{\mathbf{F}} (\xi) \right | \lesssim   \int_  {-\infty} ^\xi    (1+|\xi |+|y|) ^4    \exp \(-  \kappa _0  |\xi-y|\)  |\widetilde{\mathbf{F}} (y) | dy
\end{align}
so that  for $j=1$ the following is true,
\begin{align}\label{eq:estRj(1)}
   \| {R}_j^{(1)}(\Lambda) \widetilde{\mathbf{F}}  \| _{ \widehat{\Sigma}}\lesssim  C   \| \widetilde{\mathbf{F}}\| _{L^1(\R )}  \text{ for $j=1$ and $j=4$}.
\end{align}
The proof for $j=4$  is similar. We focus now on the remaining terms   ${R}_2^{(1)}(\Lambda) \widetilde{\mathbf{F}}$  and ${R}_3^{(1)}(\Lambda) \widetilde{\mathbf{F}}$.  We can write  schematically
\begin{align*}&
   {R}_2^{(1)}(\Lambda) \widetilde{\mathbf{F}} =  \frac{ -1 }{\Lambda ^2 \widehat{W}(\Lambda  )   }  \int_  {-\infty }  ^{\xi }
   \mathbf{a}(\xi,y, \Lambda )  : \widetilde{\mathbf{F}}(y)   dy \text{ where}
\\& \mathbf{a}(\xi,y, \Lambda )  =
     {f}_2(\xi,\Lambda )     Z_3(y,\Lambda  )  + \frac{\Lambda ^3}{3!} c_2(0)
     {f}_1(\xi,0 )      \partial _z ^3 Z_3(y,z )| _{z=0}
\end{align*} where we used \eqref{eq:defc0} and \eqref{eq:deff2}.
In the notation in Lemma  \ref{lem:bdm2}, we have  
\begin{align*} &
  \mathbf{a}(\xi,y, \Lambda )=    A(\Lambda , \xi ,y) +B(\Lambda , \xi,y)   \text{ where}    \\&  A(\Lambda , \xi,y) =
  e^{\mu _2(\Lambda ) (\xi-y)}  \(     {m}_2(\xi,\Lambda )   H_3(y,\Lambda ) - \frac{\Lambda ^3}{3!}
     {m}_2(\xi,0 )        T_3(y )    \)
\\& B(\Lambda , \xi ,y) = \frac{\Lambda ^3}{3!}  \( e^{\mu _2(\Lambda ) (\xi-y)}-1    \)    {m}_2(\xi,0 )       T_3( y )  .
\end{align*}
Since
\begin{align*}
   A(\Lambda , \xi ,y) =  e^{ \mu _2(\Lambda   ) (\xi-y)}\int _0 ^\Lambda dz \frac{(\Lambda -z )^3}{6}  \partial _z ^4\(     {m } _2(\xi ,z    )     H_3(y,z   ) \)
\end{align*}
by Lemma \ref{lem:bdm2} we have
\begin{align*}
  |A(\Lambda , \xi  ,y)|\lesssim   |\Lambda |^4  (1+|\xi | +|y_+|)^4   
\end{align*}
so that
we have
\begin{align}\label{eq:R121}
  \left \|   {R} _{21}^{(1)}(\Lambda) \widetilde{\mathbf{F}}   \right \| _{ \widehat{\Sigma}}   \lesssim   \|   \widetilde{\mathbf{F}} \| _{L^1(\R )}  \text{ for }  {R} _{21}^{(1)}(\Lambda) \widetilde{\mathbf{F}}  :=  \frac{  -1 }{\Lambda ^2 \widehat{W}(\Lambda    )   }  \int_  {-\infty}  ^{\xi}
   A(\Lambda , \xi  ,y)  : \widetilde{\mathbf{F}}(y)    dy.
\end{align}
By \eqref{eq:jostlambda01}, Lemma \ref{lem:bdm2} and $|e^{\mu_2(\Lambda)(\xi-y)}-1|\lesssim |\Lambda|(\xi-y) $, we have 
\begin{align*}
  |B(\Lambda , \xi  ,y)|\lesssim        |  \Xi_1[c](\xi)   \Lambda ^4 (\xi-y)|    
\end{align*}
so that by \eqref{lem:roots12} we have 
\begin{align}\label{eq:R122}
  \left \|   {R} _{22}^{(1)}(\Lambda) \widetilde{\mathbf{F}}   \right \| _{ \widehat{\Sigma}}   \lesssim   \|   \widetilde{\mathbf{F}} \| _{L ^{1,1}(\R )}  \text{ for }  {R} _{22}^{(1)}(\Lambda) \widetilde{\mathbf{F}}  :=  \frac{  -1 }{\Lambda ^2 \widehat{W}(\Lambda    )   }  \int_  {-\infty}  ^{\xi}
   B(\Lambda , \xi  ,y)  : \widetilde{\mathbf{F}}(y)    dy.
\end{align}
Hence we have proved for $j=2$, with for   $j=3$  a similar proof, 
\begin{align}\label{eq:estRj(2)}
   \| {R}_j^{(1)}(\Lambda) \widetilde{\mathbf{F}}  \| _{ \widehat{\Sigma}}\lesssim  C   \|   \widetilde{\mathbf{F}} \| _{L ^{1,1}(\R )}     .
\end{align}
 The implicit constants in \eqref{eq:estRj(1)} and \eqref{eq:estRj(2)} are uniform in $\Lambda$, hence 
the lemma is proved.
   
\end{proof}

\begin{lemma}\label{lem:low-derivative}
 There exists $C(p,c)>0$ such that
\begin{align}\label{eq:low-derivative}
\sup_{\substack{0<|\lambda |<\gamma ^3\varepsilon _1\\  \lambda \in  \overline{\C}_+ }}
\left\|\pi_1R_{\mathcal L_{[c]}}(\lambda )\mathcal Q_{[c]}
            (0,\partial_xg)^\intercal\right\|_{\widehat\Sigma}
 \le C(p,c)\|g\|_{L^1(\R )}.
\end{align}
\end{lemma}
\begin{proof}
Setting $\mathbf v =R_{\mathcal L_{[c]}}(\lambda )\mathcal Q_{[c]}
            (0,\partial_xg)^\intercal$, $F_g=(0,\partial_xg)^\intercal$  and for $\lambda =\gamma ^3\Lambda $, then 
\begin{align*} \widetilde{V} =   \widehat{R} ^{(1)}(\Lambda)\widetilde{\mathbf F_g}
  - \widehat{R} ^{(1)}(\Lambda)\widetilde{\mathcal P_{[c]}F_g}.
\end{align*}
Then 
\begin{align*}
    \|   \widehat{R} ^{(1)}(\Lambda)\widetilde{\mathcal P_{[c]}F_g} \| _{\widehat\Sigma} \lesssim  \sum _{j=1,2} \| \xi _j [c]\| _{L^{1,1}}     | \<  \partial _x\eta _j [c], (0,g) ^\intercal \>   |  \lesssim \|g\|_{L^1(\R )}.
\end{align*}
Next, we have  \begin{align*}       \widetilde{\mathbf F_g}=-\gamma^{-2}e_3\partial_\xi\widetilde g,\qquad\|\widetilde g\|_{L^1(d\xi)}=\gamma\|g\|_{L^1(dx)}.\end{align*}
Integrating by parts in $y$, we obtain \small
\begin{align} \label{eq:intpar}
&\pi_1 R^{(1)}(\Lambda)\widetilde{\mathbf F_g}
 =\sum_{j=1}^4\pi_1 J_j \text{ where, for $\sigma_j=-$ if $j\le 2$ and $\sigma_j=+$ if
$j\ge 3$, }\\
J_{j_1}
&=-\frac{\gamma^{-2}}{\Lambda^2\widehat W(\Lambda)}
\int_0^\Lambda \frac{(\Lambda-z)^3}{6}
\int_{\sigma_{j_1}\infty}^{\xi}
\partial_z^4\left(
(f_{j_1})_1(\xi,z)
\left\langle
e_3\widetilde g(y),
\partial_y *\bigl(
f_{j_2}\wedge f_{j_3}\wedge f_{j_4}
\bigr)(y,z)
\right\rangle
\right)\,dy\,dz \nonumber
\end{align} \normalsize and $(j_1,j_2,j_3,j_4)$   any even permutation
of $(1,2,3,4)$.  By
\eqref{eq:identdet}, the sum of the boundary contributions
at $y=\xi$, before taking the first component, is
proportional to
$e_3\partial_z^4\bigl(z^2\widehat W(z)\widetilde g(\xi)\bigr)$ which vanishes 
when we apply $\pi _1$, by $ \pi _1 e_3 =0$.
Then
$\|J_j\|_{\widehat{\Sigma}}\lesssim\|g\|_{L^1(\R)}$,
completing the proof.

\end{proof}

\section{Proof of the intermediate energy part of Proposition \ref{prop:smooth2}}
\label{sec:proofmedsmooth}

The key technical point in this section is the following.

\begin{lemma}\label{lem:nozerodet}
For any $c\in (\sqrt{p-1}/2,1)$ we have
\begin{align}
    \widehat W(\Lambda )\neq0
    \text{ for all $\Lambda\in\im\R\backslash\{0\}$.}
    \label{eq:nozeroW}
\end{align}
\end{lemma}
\begin{proof}
Let $(\widetilde v_1,\widetilde v_4)$ be a solution of
\eqref{eq:systres1} for $\widetilde F_1=\widetilde F_2=0$, and set
$u:=\widetilde v_1$ and $r:=\widetilde v_4+c\widetilde v_1$.
For $\Lambda=\im\tau$ we have
\begin{align}
    r'&=\im\gamma^2\tau u,
    \label{eq:fluxr1}\\
    (L_+^{(U)}u)'&=\im\tau(r-2cu),
    \label{eq:fluxr2}
\end{align}
where $U:=f'(Q)$ and $L_+^{(U)}:=-\partial_\xi^2+1-U$.
We claim that the following is independent of $\xi$:
\begin{align}
    \mathcal J[\widetilde V,U]
    :=\operatorname{Re}\left(
        \overline r\,L_+^{(U)}u
        +\overline{u'}\,r'
        +c\gamma^{-2}|r|^2
    \right).
    \label{eq:def_flux}
\end{align}
Indeed,
\small
\begin{align*}
    \frac{d}{d\xi}\mathcal J[\widetilde V,U]
    &=\operatorname{Re}\Big(
        \overline{r'}\,L_+^{(U)}u
        +\overline r\,(L_+^{(U)}u)'
        +\overline{u''}\,r'
        +\overline{u'}\,r''
        +c\gamma^{-2}
        (\overline{r'}\,r+\overline r\,r')
    \Big)\\
    &=\operatorname{Re}\Big(
        -\im\gamma^2\tau\overline u\,L_+^{(U)}u
        +\im\tau|r|^2
        -2\im c\tau\overline r\,u
        +\im\gamma^2\tau\overline{u''}u
        +\im\gamma^2\tau|u'|^2
        +\im c\tau(-\overline u\,r+\overline r\,u)
    \Big)\\
    &=\gamma^2\tau\operatorname{Re}\Big(
        \im(\overline u\,u''+\overline{u''}u)
    \Big)=0.
\end{align*}
\normalsize
Notice that the constancy of $\mathcal J[\widetilde V,U]$ remains
true for any real valued smooth potential $U$ in
\eqref{eq:fluxr1}--\eqref{eq:def_flux}.

We prove \eqref{eq:nozeroW} by contradiction. Fix
$\tau\in\R\backslash\{0\}$, choose $T>|\tau|$, and use the Jost
solutions supplied by Lemma \ref{lem:estimate_f} on a neighborhood
of $\im[-T,T]$. Suppose that $\widehat W(\im\tau,\gamma)=0$.
The pairs $\{f_1,\widehat f_2\}$ and $\{\widehat f_3,f_4\}$ are
linearly independent by their asymptotic behavior at $+\infty$
and $-\infty$, respectively. Hence their spans have a nontrivial
intersection, and there exist a nonzero solution $\widetilde V$
and constants $\alpha_j\in\C$ such that
\begin{align}
    \widetilde V
    =\alpha_1f_1+\alpha_2\widehat f_2
    =\alpha_3\widehat f_3+\alpha_4f_4.
    \label{eq:outgoing_intersection}
\end{align}
Here all the Jost functions are evaluated at $(\Lambda,\gamma)
=(\im\tau,\gamma)$.
By Lemma \ref{lem:estimate_f},
$f_1$ decays exponentially at $+\infty$ and $f_4$ decays
exponentially at $-\infty$. Thus, for the free solutions in
\eqref{eq:vj0},
\begin{align}
    \mathcal J[\widetilde V,U]
    &=|\alpha_2|^2\lim_{\xi\to+\infty}
        \mathcal J[\widehat f_2,U]
      =|\alpha_2|^2\mathcal J[\widetilde V_2^{(0)},0]
      \nonumber\\
    &=|\alpha_3|^2\lim_{\xi\to-\infty}
        \mathcal J[\widehat f_3,U]
      =|\alpha_3|^2\mathcal J[\widetilde V_3^{(0)},0].
    \label{eq:flux00}
\end{align}

Write $\mu_j=\im k_j$ for $j=2,3$. Then
$k_2=\tau_-^{-1}(\tau)$ and $k_3=\tau_+^{-1}(\tau)$ by
\eqref{eq:mus3pm}; in particular, $k_2,k_3\neq0$.
For $\widetilde V_j^{(0)}$ we have
$u=e^{\im k_j\xi}$,
$r=(\gamma^2\tau/k_j)e^{\im k_j\xi}$,
$L_+^{(0)}u=(1+k_j^2)u$ and $r'=\im\gamma^2\tau u$.
Consequently,
\begin{align*}
    \mathcal J[\widetilde V_j^{(0)},0]
    =\frac{\gamma^2\tau}{k_j}
        \left(1+2k_j^2+c\frac{\tau}{k_j}\right).
\end{align*}
Using \eqref{eq:mus221} and $\gamma^2=1-c^2$, we obtain
\begin{align}
    \mathcal J[\widetilde V_2^{(0)},0]
    &=\gamma^{-2}
        \left(c-\sqrt{1+\gamma^2k_2^2}\right)
        \left[1+2\gamma^2k_2^2
            -c\sqrt{1+\gamma^2k_2^2}\right]<0,
    \label{eq:JV2}\\
    \mathcal J[\widetilde V_3^{(0)},0]
    &=\gamma^{-2}
        \left(c+\sqrt{1+\gamma^2k_3^2}\right)
        \left[1+2\gamma^2k_3^2
            +c\sqrt{1+\gamma^2k_3^2}\right]>0.
    \label{eq:JV3}
\end{align}
Indeed, $0<c<1$ and
\begin{align*}
    1+2\gamma^2k_2^2-c\sqrt{1+\gamma^2k_2^2}
    \geq 1-c+(2-c)\gamma^2k_2^2>0.
\end{align*}
It follows from \eqref{eq:flux00} that $\alpha_2=\alpha_3=0$.
Hence \eqref{eq:outgoing_intersection} reduces to
$\widetilde V=\alpha_1f_1=\alpha_4f_4$, so that $\widetilde V$
decays exponentially at both ends of the line.
In particular, $\widetilde V\in L^2(\R,\C^4)$, and
\eqref{A_Decomposehom} gives
$\widetilde v_1\in H^3(\R)$ and
$\widetilde v_4\in H^1(\R)$.
By Lemma \ref{lem:transys},
\begin{align*}
    \mathbf v(x)
    :=\bigl(\widetilde v_1(\gamma x),
            \widetilde v_4(\gamma x)\bigr)^\intercal
    \text{ belongs to } H^3(\R)\times H^1(\R), 
\end{align*}
is nonzero and satisfies
$\mathcal L_{[c]}\mathbf v=\im\gamma^3\tau\mathbf v$.
This contradicts Claim \textbf{iii} in Proposition
\ref{prop:noeigh}, proving \eqref{eq:nozeroW}.
\end{proof}

We now prove the intermediate energy estimate using the  representation \eqref{eqvarpartilde}. This representation
requires only the nonvanishing of $\widehat W$.

\begin{lemma}\label{lem:interm}
For any $c\in(\sqrt{p-1}/2,1)$ and
$\mathbf c_1>\varepsilon_1>0$   there are  constants  
  $C(c,\mathbf c_1,\varepsilon_1)$  and $\varepsilon _*(c)$ such that,
for $z= \im\lambda  +\varepsilon $    with
  $0\le\varepsilon\le\varepsilon_*$ and
$\gamma^3\varepsilon_1\le|\lambda|\le\gamma^3\mathbf c_1$,
\begin{align}
 \|\sech(\gamma\kappa x)R_{\mathcal L_{[c]}}(z)F\|_{L^2(\R,\C^2)}
 &\le C(c,\mathbf c_1,\varepsilon_1)\|F\|_{L^1(\R,\C^2)},
 \label{eq:interm1}\\
 \|\pi_1R_{\mathcal L_{[c]}}(z)\mathcal Q_{[c]}
       (0,\partial_xg)^\intercal\|_{\widehat\Sigma}
 &\le C(c,\mathbf c_1,\varepsilon_1)\|g\|_{L^1(\R)}.
 \label{eq:interm-derivative}
\end{align}
 
\end{lemma}

\begin{proof}
Choose $T>\mathbf c_1$ and use the Jost solutions in
Lemma \ref{lem:estimate_f} on a neighborhood of $\im[-T,T]$.
Throughout the proof $ \lambda =\gamma ^3 \Lambda $ and
$\varepsilon_1\leq|\Lambda|\leq\mathbf c_1$.
Constants are uniform in this range.
By Lemma \ref{lem:nozerodet} and the continuity of the Jost
functions,
\begin{align}
    \inf_{\varepsilon_1\leq|\tau|\leq\mathbf c_1}
        |\widehat W(\im\tau )|>0.
    \label{eq:interm-W-bound}
\end{align}
In particular, the matrix
\begin{align*}
    \mathbf Y(\xi,\Lambda)
    :=\bigl(f_1,\widehat f_2,\widehat f_3,f_4\bigr)(\xi,\Lambda )
\end{align*}
is invertible. Grouping the first two and the last two terms in
\eqref{eqvarpartilde}, we write
\begin{align}
    \widetilde V(\xi)
    &=\int_\R\mathbf G(\xi,y,\Lambda)
        \widetilde{\mathbf F}(y)\,dy,
    \label{eq:interm-Green-rep}\\
    \mathbf G(\xi,y,\Lambda)
    &:=\begin{cases}
        \mathbf Y(\xi,\Lambda)\diag(1,1,0,0)
            \mathbf Y(y,\Lambda)^{-1},&y<\xi,\\
        -\mathbf Y(\xi,\Lambda)\diag(0,0,1,1)
            \mathbf Y(y,\Lambda)^{-1},&\xi<y ,
    \end{cases}
    \label{eq:interm-Green}
\end{align}
where we used   \ref{itm:diag} Notation \ref{not:notation}.
We claim the following,  
\begin{align}
    \sup_{\substack{\Lambda\in\im\R\\
        \varepsilon_1\leq|\Lambda|\leq\mathbf c_1}}
    \sup_{\xi\neq y}|\mathbf G(\xi,y,\Lambda)|
    \leq C(p, \varepsilon_1,\mathbf c_1).
    \label{eq:interm-Green-bound}
\end{align}
To prove this, use the   matrices
$\mathbf Y_\pm:=Y_\pm$ in \eqref{eq:complete-bases}, whose determinants are constant in $\xi$ with $\det Y_\pm  =\det\mathbf A\ne0$, see \ref{eq:matrixbfA}.
By Lemma \ref{lem:estimate_f}, the matrices
\begin{align*}
    \mathbf M_\pm(\xi,\Lambda)
    :=\mathbf Y_\pm(\xi,\Lambda)
        \diag(e^{-\mu_1\xi},e^{-\mu_2\xi},
              e^{-\mu_3\xi},e^{-\mu_4\xi})
\end{align*}
and their inverses are uniformly bounded for $\pm\xi\geq0$.
Here and below  in this proof, $\mu_j=\mu_j(\Lambda )$.
The following   changes of basis   have the block forms,
\begin{align*}
    \mathbf Y=\mathbf Y_+\mathbf C_+
        =\mathbf Y_-\mathbf C_-,\qquad
    \mathbf C_+=\begin{pmatrix}I_2&B_+\\0&A_+\end{pmatrix},
    \qquad
    \mathbf C_-=\begin{pmatrix}A_-&0\\B_-&I_2\end{pmatrix}.
\end{align*}
All these blocks are independent from $\xi$ and  uniformly bounded by continuity on  parameters $(\Lambda , \gamma )$ which vary in a  compact
  set. Taking determinants gives
\begin{align}\label{eq:interm-block-determinants}
    \det A_\pm=\frac{\widehat W}{\det\mathbf A}.
\end{align}
Thus $A_\pm^{-1}$ are uniformly bounded by
\eqref{eq:interm-W-bound}. Set $P:=\diag(I_2,0)$. Direct multiplication gives
\begin{align}\label{eq:interm-block-products}
    \mathbf C_+P\mathbf C_+^{-1}
    &=\begin{pmatrix}I_2&-B_+A_+^{-1}\\0&0\end{pmatrix},&
    \mathbf C_-P\mathbf C_-^{-1}
    &=\begin{pmatrix}I_2&0\\B_-A_-^{-1}&0\end{pmatrix},\nonumber\\
    \mathbf C_+P\mathbf C_-^{-1}
    &=\begin{pmatrix}A_-^{-1}&0\\0&0\end{pmatrix},&
    \mathbf C_-(I_4-P)\mathbf C_+^{-1}
    &=\begin{pmatrix}0&0\\0&A_+^{-1}\end{pmatrix}.
\end{align}
Since
$$
\mathbf C_\pm(I_4-P)\mathbf C_\pm^{-1}
=I_4-\mathbf C_\pm P\mathbf C_\pm^{-1},
$$
we also obtain
$$
\mathbf C_+(I_4-P)\mathbf C_+^{-1}
=\begin{pmatrix}
0&B_+A_+^{-1}\\
0&I_2
\end{pmatrix},
\qquad
\mathbf C_-(I_4-P)\mathbf C_-^{-1}
=\begin{pmatrix}
0&0\\
-B_-A_-^{-1}&I_2
\end{pmatrix}.
$$
To display the exponential factors in the two blocks, set
\begin{align*}
 E_{12}(t,\Lambda)&:=\diag(e^{\mu_1t},e^{\mu_2t}),&
 E_{34}(t,\Lambda)&:=\diag(e^{\mu_3t},e^{\mu_4t}).
\end{align*}
Substitution of
$\mathbf Y_\pm(\xi)=\mathbf M_\pm(\xi)
\diag(E_{12}(\xi),E_{34}(\xi))$
into \eqref{eq:interm-Green}, using
\eqref{eq:interm-block-products}, gives, for $\xi,y\geq0$,
\begin{align*}
 &\mathbf M_+(\xi)^{-1}\mathbf G(\xi,y,\Lambda)\mathbf M_+(y)\\
 &\quad=
 \begin{cases}
 \begin{pmatrix}
 E_{12}(\xi-y)&-E_{12}(\xi)B_+A_+^{-1}E_{34}(-y)\\
 0&0
 \end{pmatrix},&y<\xi,\\[3mm]
 \begin{pmatrix}
 0&-E_{12}(\xi)B_+A_+^{-1}E_{34}(-y)\\
 0&-E_{34}(\xi-y)
 \end{pmatrix},&\xi<y.
 \end{cases}
\end{align*}
For $\xi,y\leq0$, the corresponding identity is
\begin{align*}
 &\mathbf M_-(\xi)^{-1}\mathbf G(\xi,y,\Lambda)\mathbf M_-(y)\\
 &\quad=
 \begin{cases}
 \begin{pmatrix}
 E_{12}(\xi-y)&0\\
 E_{34}(\xi)B_-A_-^{-1}E_{12}(-y)&0
 \end{pmatrix},&y<\xi,\\[3mm]
 \begin{pmatrix}
 0&0\\
 E_{34}(\xi)B_-A_-^{-1}E_{12}(-y)&-E_{34}(\xi-y)
 \end{pmatrix},&\xi<y.
 \end{cases}
\end{align*}
On the two remaining regions, the second line of
\eqref{eq:interm-block-products} gives
\begin{align*}
 &\mathbf M_+(\xi)^{-1}\mathbf G(\xi,y,\Lambda)\mathbf M_-(y)=
 \begin{pmatrix}
 E_{12}(\xi)A_-^{-1}E_{12}(-y)&0\\
 0&0
 \end{pmatrix},
 \qquad y\leq0\leq\xi, \\[2mm]
 &\mathbf M_-(\xi)^{-1}\mathbf G(\xi,y,\Lambda)\mathbf M_+(y)=
 \begin{pmatrix}
 0&0\\
 0&-E_{34}(\xi)A_+^{-1}E_{34}(-y)
 \end{pmatrix},
 \qquad \xi\leq0\leq y  .
\end{align*}
By \eqref{eq:mus21}, the Euclidean operator norms of the
diagonal matrices satisfy
\begin{align*}
 \|E_{12}(t,\Lambda)\|
 &=\max_{j=1,2}e^{t\Re\mu_j}\leq1  \text{ for }  t\geq0    \text{ and } \\
 \|E_{34}(t,\Lambda)\|
 &=\max_{j=3,4}e^{t\Re\mu_j}\leq1  \text{ for } t\leq0.
\end{align*}
In particular, 
\begin{align*}
 \|E_{12}(\xi-y,\Lambda)\|&\leq1   \text{ for }    y<\xi   \text{ and  }
 \|E_{34}(\xi-y,\Lambda)\| \leq1    \text{ for }   \xi<y.
\end{align*}
The other blocks obey the following estimates:
\begin{align*}
 \|E_{12}(\xi)B_+A_+^{-1}E_{34}(-y)\|
 &\leq  \|E_{12}(\xi) \|   \| E_{34}(-y)\|   \|B_+\|\,\|A_+^{-1}\|   \le \|B_+\|\,\|A_+^{-1}\|,
 &&\xi,y\geq0,\\
 \|E_{34}(\xi)B_-A_-^{-1}E_{12}(-y)\|
 & \le   \|E_{34}(\xi) \|    \| E_{12}(-y)\| \|B_-\|\,\|A_-^{-1}\|\leq\|B_-\|\,\|A_-^{-1}\|,
 &&\xi,y\leq0,\\
 \|E_{12}(\xi)A_-^{-1}E_{12}(-y)\|
 &\le\|E_{12}(\xi) \| \| E_{12}(-y)\| \| A_-^{-1} \| \leq\|A_-^{-1}\|,
 &&y\leq0\leq\xi,\\
 \|E_{34}(\xi)A_+^{-1}E_{34}(-y)\|
 &\le  \|E_{34}(\xi) \|   \| E_{34}(-y)\|  \| A_+^{-1} \| \leq\|A_+^{-1}\|,
 &&\xi\leq0\leq y.
\end{align*}
Multiplication by the uniformly bounded   factors
$\mathbf M_\pm(\xi,\Lambda)$ and
$\mathbf M_\pm(y,\Lambda)^{-1}$ therefore yields
\begin{align*}
 \|\mathbf G(\xi,y,\Lambda)\|
 \leq C\bigl(1+\|A_+^{-1}\|+\|A_-^{-1}\|
       +\|B_+\|\,\|A_+^{-1}\|
       +\|B_-\|\,\|A_-^{-1}\|\bigr)
 \leq C.
\end{align*}
This proves \eqref{eq:interm-Green-bound} directly from the
 kernel \eqref{eq:interm-Green}.
By \eqref{eq:interm-Green-bound} and \eqref{eq:relkern2},
\begin{align*}
    \|\widetilde V\|_{L^\infty(\R)}
    \leq C(p,\gamma,\varepsilon_1,\mathbf c_1)
        \|\widetilde{\mathbf F}\|_{L^1(\R)}
    \leq C(p,\gamma,\varepsilon_1,\mathbf c_1)
        \|F\|_{L^1(\R)}.
\end{align*}
Since $\sech(\kappa\xi)\in L^2(\R)$, this yields
\begin{align}
    \|\sech(\kappa\xi)\widetilde V\|_{L^2(\R)}
    \leq C(p,\gamma,\varepsilon_1,\mathbf c_1)
        \|F\|_{L^1(\R)}.
    \label{eq:pfsmooth21int}
\end{align}
The weight becomes exactly that of \eqref{eq:sigmatilde} after
returning to $x=\gamma^{-1}\xi$.
This proves \eqref{eq:interm1}. 
Since $\mathcal Q_{[c]}$ is a bounded in $L^1$,
it also proves the corresponding projected estimate.

For $F_g=(0,\partial_xg)^\intercal$,  
$\widetilde{\mathbf F_g}=-\gamma^{-2}e_3\widetilde g'$,
where $\widetilde g(\xi)=g(\xi/\gamma)$,
as in the proof of Lemma \ref{lem:low-derivative}, using formulas similar to \eqref{eq:intpar} but with $\widehat{f}_j$ replacing 
$ {f}_j$ for $j=2,3$ and $ \widehat W(\Lambda)$ replacing $\Lambda^2\widehat W(\Lambda)$,
integration by parts in $y$ again  gives no boundary contribution
to the first component. 
Since 
\[
\partial_y\mathbf G(\xi,y,\Lambda)
=-\mathbf G(\xi,y,\Lambda)
(A_\infty^{(0)}+\Lambda A_\infty^{(1)}+q(y)),
\]
\eqref{eq:interm-Green-bound}   bounds also
$\partial_y\mathbf G$.
Thus, multiplying by $\sech(\kappa\xi)$ and returning to
$x=\xi/\gamma$, we obtain   \eqref{eq:interm-derivative} without $\mathcal Q_{[c]}$.
Finally, by 
$\|\mathcal P_{[c]}F_g\|_{L^1}\lesssim\|g\|_{L^1}$, the proof of \eqref{eq:interm-derivative} is completed using
\eqref{eq:interm1}.

\end{proof}

\section{Proof of Propositions \ref{prop:smooth2} and \ref{prop:smoothreg}}\label{sec:proofsmoothreg}

\begin{proof}[Proof of Proposition \ref{prop:smooth2}]
Lemma \ref{lem:smooth2small} and  Lemma \ref{lem:interm}  give the bound in \eqref{eq:smooth21strip} for $|\lambda |\le \mathbf c_1$. Since we have   proved    in Lemma \ref{lem:largterm}  the bound in \eqref{eq:smooth21strip} for $|\lambda |\ge \mathbf c_1$, the proof   is completed.

\end{proof}

\begin{proof}[Proof of Proposition \ref{prop:smoothreg}]
Proceeding like in the proof of Proposition \ref{prop:smooth}, the proof reduces to proving 
\begin{align}
 \sup_ {\substack{\lambda\in\R \\ 0\le \varepsilon \le \varepsilon _*} }
 \|\pi_1R_{\mathcal L_{[c]}}( \im\lambda +\varepsilon)
       \mathcal Q_{[c]}(0,\partial_xg)^\intercal\|_{\widehat\Sigma}
 \le C(c)\|g\|_{L^1(\R)}.
 \label{eq:derivative-resolvent-strip}
\end{align}
This estimate is already proved for $|\lambda |\le \mathbf c_1$ 
by  Lemmas \ref {lem:low-derivative} and   \ref{lem:interm}. At higher frequencies, for $\lambda =\gamma ^3 \Lambda $,   
$\widetilde g(\xi)=g(\xi/\gamma)$  and  $\mathbf v=R_{\mathcal L_{[c]}}(z)(0,\partial_xg)^\intercal$, \eqref{eq:high-scalar-resolvent} and
$L_{13}\partial_\xi=L_{23}$ give
\begin{align}
 \widetilde v_1
 =L_{23}(q_{32}\widetilde v_1)
       -\gamma^{-2}L_{23}\widetilde g.
 \label{eq:derivative-high-scalar}
\end{align}
By \eqref{eq:high-kernel-bounds} and
\eqref{eq:high-scalar-contraction}, absorption at a sufficiently
large fixed threshold gives the following, finishing the proof of \eqref{eq:derivative-resolvent-strip}
\begin{align}
 \|\pi_1R_{\mathcal L_{[c]}}(z)(0,\partial_xg)^\intercal
       \|_{\widehat\Sigma}
 \le C(c)\langle z\rangle^{-1/2}\|g\|_{L^1}.
 \label{eq:derivative-high-bound}
\end{align}

\end{proof}

\appendix
\section{Proof of Theorem \ref{them:locwp}}\label{app:them:locwp}

All functions in this appendix are real valued. The argument
applies to both nonlinearities in \eqref{eq:gB1} and all $p>1$.
Put $\alpha=\min\{p-1,1\}$. On every bounded real interval,
$f'$ is H\"older continuous with exponent $\alpha$.
The Sobolev embedding $H^1(\R)\subset L^\infty(\R)$ and the
chain rule therefore imply, for $\|v\|_{H^1},\|w\|_{H^1}\le A$,
\begin{align}
 \|f(v)\|_{H^1}&\le C(A),
 &\|f(v)-f(w)\|_{L^2}&\le C(A)\|v-w\|_{L^2},
 \label{eq:lwp-nonlinearity-L2}\\
 \|f(v)-f(w)\|_{H^1}
 &\le C(A)\bigl(\|v-w\|_{H^1}
                  +\|v-w\|_{L^2}^{\alpha/2}\bigr).
 \label{eq:lwp-nonlinearity-H1}
\end{align}
For the last inequality, write
$\partial_x(f(v)-f(w))=f'(v)(v'-w')+(f'(v)-f'(w))w'$
and use
$\|v-w\|_\infty\le C\|v-w\|_2^{1/2}
\|v'-w'\|_2^{1/2}$.
In particular, $f:H^1\to H^1$ is continuous.

Write $U_j(t)\mathbf u_0=\pi_jU(t)\mathbf u_0$, where $U(t)$
is the group in \eqref{eq:group} and $\pi_j$ selects the $j$-th
component. This group is an isometry on
$\mathcal H^1$. The first component of \eqref{eq:gB3} is
\begin{align}
 u_1(t)&=U_1(t)\mathbf u_0-\Phi(f(u_1))(t),
 \label{eq:welp1}\\
 \Phi(g)(t)&:=\pi_1\int_0^tU(t-s)(0,\partial_xg(s))^\intercal\,ds
 \nonumber\\
 &=-\int_0^t\im\partial_x(1-\partial_x^2)^{-1/2}
       \sin\bigl(\im(t-s)\partial_x\sqrt{1-\partial_x^2}\bigr)
       g(s)\,ds.
 \label{eq:lwp-Phi}
\end{align}
The Fourier multiplier in the last integral has absolute value
at most one. Thus for $s=0,1$,
\begin{align}
 \|\Phi(g)\|_{C([0,T],H^s)}
 \le\int_0^T\|g(t)\|_{H^s}\,dt.
 \label{eq:lwp-Phi-bound}
\end{align}
Fix $A>2\|\mathbf u_0\|_{\mathcal H^1}$ and consider the metric
space \small
\begin{align}
 \mathcal X_{T,A}
 :=\{v\in C([0,T],L^2): 
                     \ \sup_{0\le t\le T}\|v(t)\|_{H^1}\le A\},
 \quad d(v,w):=\|v-w\|_{C([0,T],L^2)}.
 \label{eq:lwp-complete-space}
\end{align}\normalsize
This space is complete: an $L^2$-uniform limit retains the
$H^1$ bound at every time by weak compactness and lower
semicontinuity. Every member is strongly measurable into
$H^1$, since its Fourier truncations are continuous into
$H^1$ and converge pointwise there. Consequently
$f(v)\in L^\infty([0,T],H^1)$ is strongly measurable.
By \eqref{eq:lwp-nonlinearity-L2}--\eqref{eq:lwp-Phi-bound},
the map $v\mapsto U_1(t)\mathbf u_0-\Phi(f(v))(t)$ maps
$\mathcal X_{T,A}$ into itself and is a strict contraction
when $T=T(A)>0$ is sufficiently small. Indeed,
\begin{align}
 \|\Phi(f(v))-\Phi(f(w))\|_{C([0,T],L^2)}
 \le TC(A)d(v,w).
 \label{eq:lips1}
\end{align}
Its fixed point belongs to $C([0,T],H^1)$, because the image
of the map does, by \eqref{eq:lwp-Phi-bound}. Define
\begin{align}
 u_2(t):=U_2(t)\mathbf u_0
       -\int_0^t\cos\bigl(\im(t-s)\partial_x
                           \sqrt{1-\partial_x^2}\bigr)
                         \partial_xf(u_1(s))\,ds.
 \label{eq:lwp-second-component}
\end{align}
Then $u_2\in C([0,T],L^2)$ and the two components satisfy
\eqref{eq:gB3}. This proves existence.
For two solutions with the same initial datum, the integral
version of \eqref{eq:lips1} and Gronwall's inequality imply
equality of their first components; \eqref{eq:lwp-second-component}
then gives equality of their second components.

For continuous dependence, let two initial data differ by
$\delta$ in $\mathcal H^1$, and suppose the corresponding
solutions have a common bound depending only on $A$ on the
local interval. The first-component equation and
\eqref{eq:lwp-nonlinearity-L2} give
\begin{align}
 \|v_1-w_1\|_{C([0,T],L^2)}\le C(A)\delta.
 \label{eq:lwp-dependence-L2}
\end{align}
Using \eqref{eq:lwp-nonlinearity-H1} in the same equation and
then applying Gronwall's inequality yields
\begin{align}
 \|v_1-w_1\|_{C([0,T],H^1)}
 \le C(A)(\delta+\delta^{\alpha/2}).
 \label{eq:lwp-dependence-H1}
\end{align}
The second-component formula gives the same modulus for
$\|v_2-w_2\|_{C([0,T],L^2)}$. Restarting this construction on
successive intervals proves the continuity asserted in Claim 3,
uniqueness of the maximal solution, and continuation for as long
as its $\mathcal H^1$ norm stays bounded.

We finish by proving conservation of $E$ and $M$ without
assuming additional differentiability of $f$. Let $P_n$ be the
orthogonal Fourier projection onto $\{|k|\le n\}$, acting on
both components, and consider
\begin{align}
 \partial_tu_1^{(n)}&=\partial_xu_2^{(n)},\nonumber\\
 \partial_tu_2^{(n)}
 &=\partial_x\bigl((1-\partial_x^2)u_1^{(n)}
                      -P_nf(u_1^{(n)})\bigr),
 \qquad \mathbf u^{(n)}(0)=P_n\mathbf u_0,
 \quad P_n\mathbf u^{(n)}=\mathbf u^{(n)}.
 \label{eq:gB2n}
\end{align}
On the closed space of band-limited vectors the right-hand
side is locally Lipschitz in $\mathcal H^1$: the derivatives
are bounded Fourier multipliers there, and $f$ is locally
Lipschitz from bounded $H^1$ sets into $L^2$.
The preceding fixed-point construction applies with $P_nf$
in place of $f$ and $P_n\mathbf u_0$ in place of $\mathbf u_0$.
Since $P_n$ is a contraction on $L^2$ and $H^1$, it supplies a
common interval $[0,T(A)]$ and uniform energy-space bounds.

These approximating solutions conserve the original functionals
$E$ and $M$. Indeed, their Hamiltonian equation is
\begin{align}
 \partial_t\mathbf u^{(n)}
 =P_n\partial_x\sigma_1
   \binom{(1-\partial_x^2)u_1^{(n)}-f(u_1^{(n)})}{u_2^{(n)}}.
 \label{eq:lwp-truncated-Hamiltonian}
\end{align}
The operator $P_n\partial_x\sigma_1$ is bounded and skew-adjoint
on $L^2(\R,\R^2)$, so differentiating $E$ along
\eqref{eq:lwp-truncated-Hamiltonian} gives zero.
For $M$, self-adjointness of $P_n$ and $P_nu_1^{(n)}=u_1^{(n)}$
allow the projection to be removed in its pairing. Integration
by parts then gives
\begin{align*}
 \frac{d}{dt}M(\mathbf u^{(n)})
 =\int_\R \partial_xu_1^{(n)}f(u_1^{(n)})\,dx
 =\int_\R\partial_xF(u_1^{(n)})\,dx=0.
\end{align*}
These operations are justified since the approximations are
band limited and $F(u_1^{(n)})\in W^{1,1}(\R)$.

For completeness we verify strong convergence of the
approximations. Their first-component equations and
\eqref{eq:welp1} give
\begin{align}
 u_1^{(n)}-u_1
 &=r_n-\Phi\bigl(P_n(f(u_1^{(n)})-f(u_1))\bigr),\nonumber\\
 r_n&:=U_1(t)(P_n-I)\mathbf u_0
                 -\Phi((P_n-I)f(u_1)).
 \label{eq:lwp-approximation-difference}
\end{align}
Since $f(u_1)\in C([0,T],H^1)$, strong convergence of the
uniformly bounded operators $P_n$ is uniform on its compact
range. Hence $r_n\to0$ in $C([0,T],H^1)$.
Applying \eqref{eq:lwp-nonlinearity-L2} and Gronwall to
\eqref{eq:lwp-approximation-difference} first gives
$u_1^{(n)}\to u_1$ in $C([0,T],L^2)$.
Next \eqref{eq:lwp-nonlinearity-H1} and Gronwall give convergence
in $C([0,T],H^1)$. Formula \eqref{eq:lwp-second-component},
with the projected nonlinearity for the approximations, then
gives convergence of the second components in
$C([0,T],L^2)$. We have proved
\begin{align}
 \mathbf u^{(n)}\xrightarrow{n\to +\infty}\mathbf u
       \quad\text{in }C([0,T],\mathcal H^1).
 \label{eq:lwp-strong-approximation}
\end{align}
Both $E$ and $M$ are continuous on $\mathcal H^1$; continuity
of the potential term follows from boundedness in $H^1$,
the embedding into $L^\infty$, and convergence in $L^2$.
Thus \eqref{eq:lwp-strong-approximation} passes both conservation
laws to $\mathbf u$. Restarting the argument proves them on
the entire maximal interval, completing Theorem \ref{them:locwp}.

\section*{Acknowledgments}

J.B. was  supported by the National Research Foundation of Korea(NRF) grant funded by the Korea government(MSIT) (RS-2026-25599593) and  by the Basic Science Research Program through the National Research Foundation of Korea(NRF) funded by the Ministry of Education(RS-2026-25580254). 

\noindent M.M.  was supported by the JSPS KAKENHI Grant Number 23H01079 and 24K06792.

\section*{Data availability}
Data sharing is not applicable to this article as no datasets
were generated or analysed during the current study.

\section*{Statement of the use of AI}
The overall approach was developed by the authors. ChatGPT 5.6 and Astra were used to generate proofs of some lemmas and portions of the text. Specifically, ChatGPT suggested the functional $\mathcal J$ in \eqref{eq:def_flux}, which contributed to the proof that $\widehat W(\Lambda )\neq 0$ in Lemma \ref{lem:nozerodet}. This contribution enabled us to capture all  velocities $\frac{\sqrt{p-1}}{2}<c<1$ for all $1<p<5$. Without this contribution the result 
would have been valid for all $c\in \left (\frac{\sqrt{p-1}}{2}, 1\right ) \setminus C_p$ for all $1<p<5$, with $C_p =\{  c_n \} _{n\in A}$ 
with $A$ either finite or $A=\N$ and, if $A=\N$,  with $c_n \xrightarrow{n\to +\infty}\frac{\sqrt{p-1}}{2}$.   All AI-generated material was carefully reviewed and verified by the authors, who take full responsibility for the content of the manuscript.

Department of Mathematics, Kyungpook National University, 80 Daehakro, Bukgu, Daegu 41566 Korea.
{\it E-mail Address}: {\tt junsikbae@knu.ac.kr}

Department of Mathematics, Informatics and Geosciences,  University
of Trieste, via Valerio  12/1  Trieste, 34127  Italy.
{\it E-mail Address}: {\tt scuccagna@units.it, delsanto@units.it }

Department of Mathematics and Informatics,
Graduate School of Science,
Chiba University,
Chiba 263-8522, Japan.
{\it E-mail Address}: {\tt maeda@math.s.chiba-u.ac.jp}

\end{document}